\documentclass[10pt]{amsart}

\usepackage[english]{babel}

\usepackage{amssymb}
\usepackage[leqno]{amsmath}
\usepackage{mathrsfs}
\usepackage{stmaryrd}
\usepackage{chemarrow}
\usepackage[norelsize]{algorithm2e}
\usepackage{enumerate}
\usepackage{graphicx}
\usepackage[all]{xy}
\usepackage{tikz}
\usetikzlibrary{arrows}
\usepackage{multirow}
\usepackage[colorinlistoftodos]{todonotes}
\usepackage[colorlinks=true, allcolors=blue]{hyperref}
\usepackage[hyperpageref]{backref}

\newtheorem{theorem}{Theorem}[section]
\newtheorem{lemma}[theorem]{Lemma}
\newtheorem{corollary}[theorem]{Corollary}

\newcommand{\dd}{\,{\rm d}}
\newcommand{\bs}{\boldsymbol}

\DeclareMathOperator*{\img}{img}
\newcommand{\sign}{\operatorname{sign}}
\newcommand{\curl}{\operatorname{curl}}
\renewcommand{\div}{\operatorname{div}}
\newcommand{\grad}{\operatorname{grad}}
\DeclareMathOperator*{\tr}{tr}
\DeclareMathOperator*{\rot}{rot}

\newcommand{\dev}{\operatorname{dev}}
\newcommand{\sym}{\operatorname{sym}}
\newcommand{\skw}{\operatorname{skw}}

\newcommand{\mskw}{\operatorname{mskw}}
\newcommand{\vskw}{\operatorname{vskw}}

\newcommand{\hess}{\operatorname{hess}}

\begin{document}
\title[Finite elements for divdiv-conforming symmetric tensors]{Finite elements for divdiv-conforming symmetric tensors in three dimensions}
\author{Long Chen}%
\address{Department of Mathematics, University of California at Irvine, Irvine, CA 92697, USA}%
\email{chenlong@math.uci.edu}%
\author{Xuehai Huang}%
\address{School of Mathematics, Shanghai University of Finance and Economics, Shanghai 200433, China}%
\email{huang.xuehai@sufe.edu.cn}%

\thanks{The first author was supported by NSF DMS-2012465, and in part by DMS-1913080.}
\thanks{The second author is the corresponding author. The second author was supported by the National Natural Science Foundation of China Project 11771338 and the Fundamental Research Funds for the Central
Universities 2019110066.}


\begin{abstract}
Two types of finite element spaces on a tetrahedron are constructed for divdiv conforming symmetric tensors in three dimensions. The key tools of the construction are the decomposition of polynomial tensor spaces and the characterization of the trace operators. First, the divdiv Hilbert complex and its corresponding polynomial complexes are presented. Several decompositions of polynomial vector and tensors spaces are derived from the polynomial complexes.  Then, traces for div-div operator are characterized through a Green's identity. Besides the normal-normal component, another trace involving combination of first order derivatives of the tensor is continuous across the face. Due to the smoothness of polynomials, the symmetric tensor element is also continuous at vertices, and on the plane orthogonal to each edge. Third, a finite element for sym curl-conforming trace-free tensors is constructed following the same approach. Finally, a finite element divdiv complex, as well as the bubble functions complex, in three dimensions are established.   
\end{abstract}

\maketitle


\section{Introduction}
In this paper, we shall construct finite elements for space \[
\boldsymbol{H}(\div{\div },\Omega; \mathbb{S}):=\{\boldsymbol{\tau}\in \boldsymbol{L}^{2}(\Omega; \mathbb{S}): \div {\div}\boldsymbol{\tau}\in L^{2}(\Omega)\}, \; \Omega\subset \mathbb R^3,
\]
which consists of symmetric tensors such that $\div {\div}\boldsymbol{\tau}\in L^{2}(\Omega)$ with the inner $\div$ applied row-wisely to $\boldsymbol \tau$ resulting a column vector for which the outer $\div$ operator is applied. $H(\div\div)$-conforming finite elements can be applied to discretize the linearized Einstein-Bianchi system~\cite[Section~4.11]{Quenneville-Belair2015} and the mixed formulation of the biharmonic equation~\cite{PaulyZulehner2020}. 

The construction in three dimensions is much harder than that in two dimensions. 
The essential difficulty arises from the underline divdiv Hilbert complex
\begin{align*}
\resizebox{1.0\hsize}{!}{$
\boldsymbol{RT}\xrightarrow{\subset} \boldsymbol H^1(\Omega;\mathbb R^3)\xrightarrow{\dev\grad}\boldsymbol H(\sym\curl,\Omega;\mathbb T)\xrightarrow{\sym\curl} \boldsymbol H(\div\div, \Omega;\mathbb S) \xrightarrow{\div{\div}} L^2(\Omega)\xrightarrow{}0,
$}
\end{align*}
where $\bs{RT}= \{a\boldsymbol x + \boldsymbol b: a\in \mathbb R, \boldsymbol b \in \mathbb R^3\}$, $\boldsymbol H^1(\Omega;\mathbb R^3)$ and $L^2(\Omega)$ are standard Sobolev spaces, and $\boldsymbol H(\sym\curl,\Omega;\mathbb T)$ is the space of traceless tensor $\boldsymbol \sigma\in L^2(\Omega;\mathbb T)$ such that $\sym \curl \boldsymbol \sigma \in L^2(\Omega; \mathbb S)$ with the row-wise $\curl$ operator. In the divdiv complex in three dimensions, the Sobolev space before $\boldsymbol H(\div\div, \Omega;\mathbb S)$ consists of tensor functions rather than vector functions in two dimensions.
By comparison, the divdiv Hilbert complex in two dimensions is
\begin{equation*}
\boldsymbol {RT}\xrightarrow{\subset} \boldsymbol  H^1(\Omega;\mathbb R^2)\xrightarrow{\sym\curl} \boldsymbol{H}(\div{\div},\Omega; \mathbb{S}) \xrightarrow{\div{\div }} L^2(\Omega)\xrightarrow{}0.
\end{equation*}
Finite element spaces for $\boldsymbol  H^1(\Omega;\mathbb R^2)$ are relatively mature.
 Then the design of divdiv conforming finite elements in two dimensions is relatively easy; see~\cite{ChenHuang2020} and also Section \S~\ref{sec:2D}. 


We start our construction from the two polynomial complexes
\begin{equation}\label{eq:introdivdivcomplex3dPolydouble}
\resizebox{.92\hsize}{!}{$
\xymatrix{
\boldsymbol{RT}\ar@<0.4ex>[r]^-{\subset} & \mathbb P_{k+2}(\Omega; \mathbb R^3)\ar@<0.4ex>[r]^-{\dev\grad}\ar@<0.4ex>[l]^-{\boldsymbol \pi_{RT}} & \mathbb P_{k+1}(\Omega; \mathbb T)\ar@<0.4ex>[r]^-{\sym\curl}\ar@<0.4ex>[l]^-{\cdot\boldsymbol x}  & \mathbb P_k(\Omega; \mathbb S) \ar@<0.4ex>[r]^-{\div{\div}}\ar@<0.4ex>[l]^-{\times\boldsymbol x} & \mathbb P_{k-2}(\Omega)  \ar@<0.4ex>[r]^-{} \ar@<0.4ex>[l]^-{\boldsymbol x\boldsymbol x^{\intercal}}
& 0 \ar@<0.4ex>[l]^-{\supset} }
$}
\end{equation}
 and reveal several decompositions of polynomial vector and tensors spaces from~\eqref{eq:introdivdivcomplex3dPolydouble}. We then present a Green's identity 
\begin{align*}
(\div\div\boldsymbol \tau, v)_K&=(\boldsymbol \tau, \nabla^2v)_K -\sum_{F\in\mathcal F(K)}\sum_{e\in\mathcal E(F)}(\boldsymbol n_{F,e}^{\intercal}\boldsymbol \tau \boldsymbol n, v)_e\\
&\quad - \sum_{F\in\mathcal F(K)}\left[(\boldsymbol  n^{\intercal}\boldsymbol \tau\boldsymbol  n, \partial_n v)_{F} -  ( 2\div_F(\boldsymbol \tau\boldsymbol n)+\partial_n (\boldsymbol  n^{\intercal}\boldsymbol \tau\boldsymbol  n), v)_F\right], 
\end{align*}
and give a characterization of two traces for $\boldsymbol \tau \in \boldsymbol{H}(\div{\div },K; \mathbb{S})$
$$
\boldsymbol  n^{\intercal}\boldsymbol \tau\boldsymbol  n \in H_n^{-1/2}(\partial K), \quad \text{ and } \; 2\div_F(\boldsymbol\tau \boldsymbol n)+ \partial_n(\boldsymbol n^{\intercal} \boldsymbol \tau\boldsymbol n) \in H_t^{-3/2}(\partial K),
$$
see Section \S~\ref{subsec:trace} for detailed definition of the negative Sobolev space for traces. 

Based on the  decomposition of polynomial tensors and the characterization of traces, we are able to construct two types of $H(\div\div)$-conforming finite element spaces on a tetrahedron. Here we present the BDM-type (full polynomial) space below. Let $K$ be a tetrahedron and let $k\geq 3$ be a positive integer. The shape function space is simply $\mathbb P_k(K;\mathbb S)$. The set of edges of $K$ is denoted by $\mathcal E(K)$, the set of faces by $\mathcal F(K)$, and the set of vertices by $\mathcal V(K)$. 
For each edge, we chose two normal vectors $\boldsymbol n_1$ and $\boldsymbol n_2$. The degrees of freedom (d.o.f) are given by
\begin{align}
\boldsymbol \tau (\delta) & \quad\forall~\delta\in \mathcal V(K), \label{intro:Hdivdivfem3ddof1}\\
(\boldsymbol  n_i^{\intercal}\boldsymbol \tau\boldsymbol n_j, q)_e & \quad\forall~q\in\mathbb P_{k-2}(e),  e\in\mathcal E(K),\; i,j=1,2,\label{intro:Hdivdivfem3ddof2}\\
(\boldsymbol  n^{\intercal}\boldsymbol \tau\boldsymbol  n, q)_F & \quad\forall~q\in\mathbb P_{k-3}(F),  F\in\mathcal F(K),\label{intro:Hdivdivfem3ddof3}\\
(2\div_F(\boldsymbol\tau \boldsymbol n)+ \partial_n(\boldsymbol n^{\intercal} \boldsymbol \tau\boldsymbol n), q)_F & \quad\forall~q\in\mathbb P_{k-1}(F),  F\in\mathcal F(K),\label{intro:Hdivdivfem3ddof4}\\
(\boldsymbol \tau, \boldsymbol \varsigma)_K & \quad\forall~\boldsymbol \varsigma\in\nabla^2\mathbb P_{k-2}(K), \label{intro:Hdivdivfem3ddof5} \\
(\boldsymbol \tau, \boldsymbol \varsigma)_K & \quad\forall~\boldsymbol \varsigma\in \sym(\mathbb P_{k-2}(K; \mathbb T)\times\boldsymbol x), \label{intro:Hdivdivfem3ddof55} \\
(\boldsymbol \tau\boldsymbol n, \boldsymbol  n\times \boldsymbol x q)_{F_1} & \quad\forall~q\in\mathbb P_{k-2}(F_1),\label{intro:Hdivdivfem3ddof6}
\end{align}
where $F_1\in\mathcal F(K)$ is an arbitrary but fixed face. The last degrees of freedom~\eqref{intro:Hdivdivfem3ddof6} will be regarded as interior degrees of freedom to the tetrahedron $K$. Namely when a face $F$ is chosen in different elements, the degrees of freedom~\eqref{intro:Hdivdivfem3ddof6} are double-valued when defining the global finite element space. The RT-type (incomplete polynomial) space can be obtained by further reducing the index of degree of freedoms by $1$ except the moment with $\nabla^2\mathbb P_{k-2}(K)$. 
To the best of our knowledge, these are the first $H(\div\div)$-conforming finite elements for symmetric tensors in three dimensions.
We notice in the recent work \cite{Hu;Ma;Zhang:2020family}, a new family of divdiv-conforming finite elements are introduced for triangular and tetrahedral grids in a more unified way. The constructed finite element spaces there are in $\bs H(\div\div, \Omega; \mathbb S)\cap \bs H(\div, \Omega; \mathbb S)$, which are smoother than ours. 


To help the understanding of our construction, we sketch a decomposition of a finite element space associated to a generic differential operator $\dd$ in Fig.~\ref{fig:femdec}, where $\dd^*$ is the $L^2$ adjoint of $\dd$. 
\begin{figure}[htbp]
\begin{center}
\includegraphics[width=5.4cm]{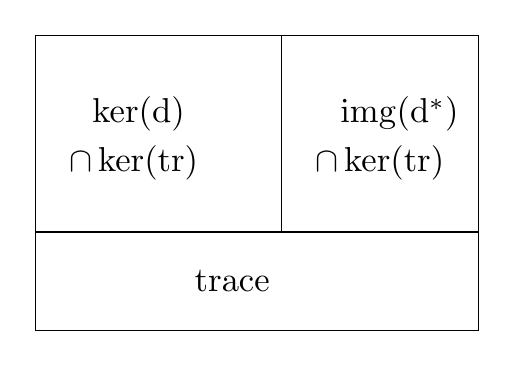}
\caption{Decomposition of a generic finite element space}
\label{fig:femdec}
\end{center}
\end{figure}
The boundary degree of freedoms~\eqref{intro:Hdivdivfem3ddof3}-\eqref{intro:Hdivdivfem3ddof4} are obviously motivated by the Green's formulae and the characterization of the trace of $\boldsymbol{H}(\div{\div },\Omega; \mathbb{S})$. The extra continuity~\eqref{intro:Hdivdivfem3ddof1}-\eqref{intro:Hdivdivfem3ddof2} is to ensure the cancellation of the edge term when adding element-wise Green's identity over a mesh. All together~\eqref{intro:Hdivdivfem3ddof1}-\eqref{intro:Hdivdivfem3ddof4} will determine the trace on the boundary of a tetrahedron. 

The interior moments of $\nabla^2\mathbb P_{k-2}(K)$ is to determine the image $\div\div (\mathbb P_{k}(K;\mathbb S)\cap \ker(\tr))$, which is isomorphism to ${\rm img}(\nabla^2)\cap \ker(\tr)$ -- the upper right block in Fig.~\ref{fig:femdec}. Together with $\sym(\mathbb P_{k-2}(K; \mathbb T)\times\boldsymbol x)$, the volume moments can determine the polynomial of degree only up to $k-1$. We then use the vanished trace and the symmetry of the tensor to figure out the rest d.o.f. 
The degrees of freedom~\eqref{intro:Hdivdivfem3ddof55}-\eqref{intro:Hdivdivfem3ddof6} will determine $\ker(\div\div)\cap \ker(\tr)$ -- the upper left block in Fig.~\ref{fig:femdec}.

For the symmetric tensor space, it seems odd to have degrees of freedom not symmetric, as a face is singled out. In view of Fig.~\ref{fig:femdec} and the exactness of the polynomial divdiv complex~\eqref{eq:introdivdivcomplex3dPolydouble}, a symmetric set of d.o.f. is replacing~\eqref{intro:Hdivdivfem3ddof55}-\eqref{intro:Hdivdivfem3ddof6} by 
\begin{equation}\label{intro:Hdivdivfem3ddofbubble}
(\boldsymbol \tau, \boldsymbol \varsigma)_K  \quad\forall~\boldsymbol \varsigma\in \sym\curl\boldsymbol B_{k+1}(\sym\curl, K;\mathbb T),
\end{equation}
where $\boldsymbol B_{k+1}(\sym\curl, K;\mathbb T) = \mathbb P_{k+1}(K;\mathbb T)\cap \boldsymbol H_0(\sym\curl, K;\mathbb T)$ is the so-called bubble function space and will be characterized precisely in Section \S~\ref{sec:bubble}. 
Although~\eqref{intro:Hdivdivfem3ddofbubble} is more symmetric, it is indeed not simpler than~\eqref{intro:Hdivdivfem3ddof55}-\eqref{intro:Hdivdivfem3ddof6} in implementation as the formulation of $\sym\curl\boldsymbol B_{k+1}(\sym\curl, K;\mathbb T)$ is much more complicated than polynomials on a face.


With the help of the $H(\div\div)$-conforming finite elements for symmetric tensors and two traces $\bs n\times \sym (\boldsymbol\tau\times\boldsymbol n)\times \bs n$ and $\boldsymbol n\cdot \boldsymbol\tau\times\boldsymbol n$ of space $\boldsymbol H(\sym \curl,K; \mathbb T)$, we construct ${H}(\sym\curl)$-conforming finite elements for trace-free tensors with $\mathbb P_{\ell+1}(K;\mathbb T)$ as
the space of shape functions with $\ell\geq \max\{k-1, 3\}$.
The degrees of freedom are
{\small
\begin{align}
\boldsymbol \tau(\delta) & \quad\forall~\delta\in \mathcal V(K), \label{intro:Hsymcurlfem3ddof1}\\
(\sym\curl\boldsymbol \tau )(\delta) & \quad\forall~\delta\in \mathcal V(K), \label{intro:Hsymcurlfem3ddof2}\\
(\boldsymbol  n_i^{\intercal}(\sym\curl\boldsymbol \tau )\boldsymbol n_j, q)_e & \quad\forall~q\in\mathbb P_{\ell-2}(e),  e\in\mathcal E(K), i,j=1,2, \label{intro:Hsymcurlfem3ddof3}\\
(\boldsymbol  n_i^{\intercal}\boldsymbol \tau\boldsymbol t, q)_e & \quad\forall~q\in\mathbb P_{\ell-1}(e),  e\in\mathcal E(K), i=1,2,\label{intro:Hsymcurlfem3ddof4}\\
(\boldsymbol n_{2}^{\intercal}(\curl\boldsymbol \tau)\boldsymbol n_1 + \partial_{t}(\boldsymbol t^{\intercal}\bs\tau\boldsymbol t), q)_e & \quad\forall~q\in\mathbb P_{\ell}(e),  e\in\mathcal E(K),\label{intro:Hsymcurlfem3ddof5}\\
(\boldsymbol n\times\sym(\boldsymbol\tau\times\boldsymbol n)\times\boldsymbol n, \boldsymbol \varsigma)_F &\quad\forall~\boldsymbol \varsigma\in (\nabla_F^{\bot})^2\, \mathbb P_{\ell-1}(F)\oplus
 \sym (\boldsymbol x\otimes \mathbb P_{\ell-1}(F;\mathbb R^2)), \label{intro:Hsymcurlfem3ddof6}\\
(\boldsymbol n\cdot \boldsymbol\tau\times\boldsymbol n, \boldsymbol q)_F & \quad\forall~\boldsymbol q\in\nabla_F\mathbb P_{\ell-3}(F)\oplus\boldsymbol x^{\perp}\mathbb P_{\ell-1}(F),  F\in\mathcal F(K),\label{intro:Hsymcurlfem3ddof7}\\
(\boldsymbol \tau, \boldsymbol q)_K & \quad\forall~\boldsymbol q\in\boldsymbol B_{\ell+1}(\sym\curl, K;\mathbb T). \label{intro:Hsymcurlfem3ddof8} 
\end{align}}

Combining previous finite elements for tensors and the vectorial Hermite element in three dimensions,
we arrive at a finite element divdiv complex in three dimensions, and the associated finite element bubble divdiv complex. 
Recently a finite element divdiv complex in three dimensions involving the $H(\div\div)$-conforming finite elements for symmetric tensors constructed in this paper is devised in \cite{Hu;Liang;Ma:2021Finite}. The ${H}(\sym\curl)$-conforming finite elements for trace-free tensors and $H^1$-conforming finite elements for vectors employed in \cite{Hu;Liang;Ma:2021Finite} are smoother than ours.
Two dimensional finite element divdiv complexes can be found in~\cite{Chen;Hu;Huang:2018Multigrid,ChenHuang2020,Hu;Ma;Zhang:2020family}. 

The rest of this paper is organized as follows. We present some operations for vectors and tensors in Section~\ref{sec:notation}. Two polynomial complexes related to the divdiv complex, and direct sum decompositions of polynomial spaces are shown in Section~\ref{sec:polycomplex}.
We derive the Green's identity and characterize the trace of $ \boldsymbol{H}(\div\div, \Omega; \mathbb{S}) $ on polyhedrons in Section~\ref{sec:greentrace}, and then construct the conforming finite elements for $\boldsymbol{H}(\div{\div },\Omega; \mathbb{S})$ in three dimensions in Section~\ref{sec:fem}.
In Section~\ref{sec:tracefreefem} we construct conforming finite elements for $\boldsymbol{H}(\sym\curl,\Omega; \mathbb{T})$.
With previous devised finite elements for tensors, we form a finite element divdiv complex in three dimensions in Section~\ref{sec:femdivdivcomplex}.

\section{Matrix and Vector Operations}\label{sec:notation}
In this section, we shall survey operations for vectors and tensors. In particular, we shall distinguish operators applied to columns and rows of a matrix.

\subsection{Matrix-vector products}
The matrix-vector product $\boldsymbol A\boldsymbol b$ can be interpreted as the inner product of $\boldsymbol b$ with the row vectors of $\boldsymbol A$. We thus define the dot operator
$\boldsymbol A\cdot \boldsymbol b := \boldsymbol A \boldsymbol b.$ Similarly we can define the row-wise cross product from the right $\boldsymbol A\times \boldsymbol b$. 
Here rigorously speaking when a column vector $\boldsymbol b$ is treat as a row vector, notation $\boldsymbol b^{\intercal}$ should be used. In most places, however, we will sacrifice this precision for the ease of notation. When the vector is on the left of the matrix, the operation is defined column-wise. For example, $\boldsymbol b \cdot \boldsymbol A : = \boldsymbol b^{\intercal}\boldsymbol A$. For dot products, we will still mainly use the conventional notation, e.g. $\boldsymbol b\cdot \boldsymbol A\cdot \boldsymbol c = \boldsymbol b^{\intercal} \boldsymbol A\boldsymbol c$. But for the cross products, we emphasize again the cross product of a vector from the left is column-wise and from the right is row-wise. The transpose rule still works, i.e. $\boldsymbol b\times \boldsymbol A = -(\boldsymbol A^{\intercal}\times \boldsymbol b )^{\intercal}$. Here again, we mix the usage of column vector $\boldsymbol b$ and row vector $\boldsymbol b^{\intercal}$. 

%

The ordering of performing the row and column products does not matter which leads to the associative rule of the triple products
$$
\boldsymbol b\times \boldsymbol A\times \boldsymbol c := (\boldsymbol b\times \boldsymbol A)\times \boldsymbol c = \boldsymbol b\times (\boldsymbol A\times \boldsymbol c).
$$
Similar rules hold for $\boldsymbol b\cdot \boldsymbol A\cdot \boldsymbol c$ and $\boldsymbol b\cdot \boldsymbol A\times \boldsymbol c$ and thus parentheses can be safely skipped when no differentiation is involved. 

For two column vectors $\boldsymbol u, \boldsymbol v$, the tensor product $\boldsymbol u\otimes \boldsymbol v := \boldsymbol u\boldsymbol v^{\intercal}$ is a matrix which is also known as the dyadic product $\boldsymbol u\boldsymbol v: = \boldsymbol u\boldsymbol v^{\intercal}$ with more clean notation (one $^{\intercal}$ is skipped). The row-wise product and column-wise product with another vector will be applied to the neighboring vector
\begin{align}
\label{eq:xuv}
\boldsymbol x\cdot (\boldsymbol u\boldsymbol v) = (\boldsymbol x\cdot \boldsymbol u) \boldsymbol v^{\intercal}, \quad (\boldsymbol u\boldsymbol v)\cdot \boldsymbol x = \boldsymbol u (\boldsymbol v\cdot \boldsymbol x), \\
\label{eq:xtimesuv}
\boldsymbol x\times (\boldsymbol u\boldsymbol v) = (\boldsymbol x\times \boldsymbol u) \boldsymbol v, \quad (\boldsymbol u\boldsymbol v)\times \boldsymbol x = \boldsymbol u (\boldsymbol v\times \boldsymbol x).
\end{align}

\subsection{Differentiation}
We treat Hamilton operator $\nabla = (\partial_1, \partial_2, \partial_3)^{\intercal}$ as a column vector. For a vector function $\boldsymbol u = (u_1, u_2, u_3)^{\intercal}$, $\curl \boldsymbol u= \nabla \times \boldsymbol u$, and $\div \boldsymbol u = \nabla \cdot \boldsymbol u$ are standard differential operations. Define $\nabla \boldsymbol u := \nabla \boldsymbol u^{\intercal} = (\partial_i u_j)$, which can be understood as the dyadic product of Hamilton operator $\nabla$ and column vector $\boldsymbol u$.

Apply these matrix-vector operations to the Hamilton operator $\nabla$, we get column-wise differentiation $\nabla \cdot \boldsymbol A, \nabla \times \boldsymbol A,$
and row-wise differentiation
$\boldsymbol A\cdot \nabla, \boldsymbol A\times \nabla.$ Conventionally, the differentiation is applied to the function after the $\nabla$ symbol. So a more conventional notation is
\begin{align*}
\boldsymbol A\cdot \nabla  : = (\nabla \cdot \boldsymbol A^{\intercal})^{\intercal}, \quad \boldsymbol A\times \nabla : = - (\nabla \times \boldsymbol A^{\intercal})^{\intercal}.
\end{align*}
By moving the differential operator to the right, the notation is simplified and the transpose rule for matrix-vector products can be formally used. Again the right most column vector $\nabla$ is treated as a row vector $\nabla^{\intercal}$ to make the notation cleaner. 

In the literature, differential operators are usually applied row-wisely to tensors. To distinguish with $\nabla$ notation, we define operators in letters as
\begin{align*}
\grad \boldsymbol u &:= \boldsymbol u \nabla^{\intercal} = (\partial_j u_i ) = (\nabla \boldsymbol u)^{\intercal},\\
\curl \boldsymbol A &: = - \boldsymbol A\times \nabla = (\nabla \times \boldsymbol A^{\intercal})^{\intercal},\\
\div \boldsymbol A &: = \boldsymbol A\cdot \nabla = (\nabla \cdot \boldsymbol A^{\intercal})^{\intercal}.
\end{align*}
Then the double divergence operator can be written as
\begin{equation*}
\div\div\boldsymbol A := \nabla\cdot \boldsymbol A\cdot \nabla.
\end{equation*}
Again as the column and row operations are independent, the ordering of operations is not important and parentheses is skipped. 



\subsection{Matrix decompositions}
Denote the space of all  $3\times3$ matrices by $\mathbb{M}$, all symmetric $3\times3$ matrices by $\mathbb{S}$, all skew-symmetric $3\times3$ matrices by $\mathbb{K}$, and all trace-free $3\times3$ matrices by $\mathbb{T}$. 
For any matrix $\boldsymbol B\in \mathbb M$, we can decompose it into symmetric and skew-symmetric parts as
$$
\boldsymbol B = {\rm sym}(\boldsymbol B) + {\rm skw}(\boldsymbol B):= \frac{1}{2}(\boldsymbol B + \boldsymbol B^{\intercal}) + \frac{1}{2}(\boldsymbol B - \boldsymbol B^{\intercal}).
$$
We can also decompose it into a direct sum of a trace free matrix and a diagonal matrix as
\begin{equation}\label{eq:devtr}
\boldsymbol B = {\rm dev} \boldsymbol B + \frac{1}{3}\tr(\boldsymbol B)\boldsymbol I := (\boldsymbol B - \frac{1}{3}\tr(\boldsymbol B)\boldsymbol I) + \frac{1}{3}\tr(\boldsymbol B)\boldsymbol I.
\end{equation}
Define $\sym\curl$ operator for a matrix $\boldsymbol A$
$$
\sym\curl \boldsymbol A := \frac{1}{2}( \nabla \times \boldsymbol A^{\intercal} + (\nabla \times \boldsymbol A^{\intercal})^{\intercal}) = \frac{1}{2}( \nabla \times \boldsymbol A^{\intercal} - \boldsymbol A\times\nabla ).
$$



We define an isomorphism of $\mathbb R^3$ and the space of skew-symmetric matrices $\mathbb K$ as follows: for a vector $\boldsymbol \omega =
( \omega_1, \omega_2, \omega_3)^{\intercal}
\in \mathbb R^3,$
$$
\mskw \boldsymbol \omega := 
\begin{pmatrix}
 0 & -\omega_3 & \omega_2 \\
\omega_3 & 0 & - \omega_1\\
-\omega_2 & \omega_1 & 0
\end{pmatrix}. 
$$
Obviously $\mskw: \mathbb R^3 \to \mathbb K$ is a bijection. We define $\vskw: \mathbb M\to \mathbb R^3$ by $\vskw := \mskw^{-1}\circ \skw$.
%
%

We will use the following identities  which can be verified by direct calculation. 
\begin{align}
{\rm skw}(\nabla \boldsymbol u) &= \frac{1}{2} (\mskw \nabla \times \boldsymbol u),\notag\\
\label{eq:skwcurl}{\rm skw}(\curl\boldsymbol A) &= \frac{1}{2} \mskw\left[\div(\boldsymbol A^{\intercal})-\grad(\tr(\boldsymbol A))\right], \\
\label{eq:divmskw} \div \mskw \boldsymbol u &= - \curl \boldsymbol u,\\
\label{eq:curlgrad} \curl (u \boldsymbol I)&=- \mskw \grad(u),\\
\label{eq:trcross}\tr(\boldsymbol\tau\times\boldsymbol x) &=-2\boldsymbol x\cdot\vskw\boldsymbol\tau.
\end{align}
More identities involving the matrix operation and differentiation are summarized in~\cite{ArnoldHu2020}; see also~\cite{Chen;Huang:2021Finite}. 


\subsection{Projections to a plane}
Given a plane $F$ with normal vector $\boldsymbol n$, for a vector $\boldsymbol v\in \mathbb R^3$, we have the orthogonal decomposition
$$
\boldsymbol v = \Pi_n \boldsymbol v + \Pi_F \boldsymbol v := (\boldsymbol v\cdot \boldsymbol n)\boldsymbol n + (\boldsymbol n\times \boldsymbol v)\times \boldsymbol n.
$$
The vector $\Pi_F^{\bot}\boldsymbol v :=\boldsymbol n\times \boldsymbol v$ is also on the plane $F$ and is a rotation of $\Pi_F \boldsymbol v$ by $90^{\circ}$ counter-clockwise with respect to $\boldsymbol n$. 
We treat Hamilton operator $\nabla = (\partial_1, \partial_2, \partial_3)^{\intercal}$ as a column vector and define
$$
\nabla_F^{\bot} := \boldsymbol n\times \nabla, \quad \nabla_F: = \Pi_F \nabla = (\boldsymbol n\times \nabla)\times \boldsymbol n.
$$
For a scalar function $v$,
\begin{align*}
\grad_F v : = \nabla_F v = \Pi_F (\nabla v), \\
 \curl_F v := \nabla_F^{\bot} v = \boldsymbol n \times \nabla v,
\end{align*}
 are the surface gradient of $v$ and surface $\curl$, respectively. For a vector function $\boldsymbol v$, $\nabla_F\cdot \boldsymbol v$ is the surface divergence
$$
\div_F\boldsymbol v := \nabla_F\cdot \boldsymbol v = \nabla_F\cdot(\Pi_F\boldsymbol v).
$$
By the cyclic invariance of the mix product and the fact $\boldsymbol  n$ is constant, the surface rot operator is
\begin{equation*}
{\rm rot}_F \boldsymbol  v := \nabla_F^{\bot}\cdot \boldsymbol  v = (\boldsymbol  n\times \nabla)\cdot \boldsymbol  v = \boldsymbol  n\cdot (\nabla \times \boldsymbol  v),
\end{equation*}
which is the normal component of $\nabla \times \boldsymbol  v$. 
The tangential trace of $\nabla \times \boldsymbol  u$ is 
\begin{equation}\label{eq:tangentialtrace}
\boldsymbol  n\times (\nabla \times \boldsymbol  v) = \nabla (\boldsymbol  n\cdot \boldsymbol  v) - \partial_n \boldsymbol  v. 
\end{equation}
By definition,
\begin{equation}\label{eq:rotFdivF}
{\rm rot}_F \boldsymbol  v = - \div_F (\boldsymbol  n\times \boldsymbol  v), \quad
\div_F \boldsymbol  v = {\rm rot}_F (\boldsymbol  n\times \boldsymbol  v).
\end{equation}
Note that the three dimensional $\curl$ operator restricted to a two dimensional plane $F$ results in two operators: $\curl_F$ maps a scalar to a vector, which is a rotation of $\grad_F$, and $\rot_F$ maps a vector to a scalar which can be thought as a rotated version of $\div_F$. The surface differentiations satisfy the property $\div_F\curl_F = 0$ and $\rot_F\grad_F = 0$ and when $F$ is simply connected, $\ker(\div_F) = {\rm img}(\curl_F)$ and $\ker(\rot_F) = {\rm img}(\grad_F)$.

Differentiation for two dimensional tensors can be defined similarly.  

\section{Divdiv Complex and Polynomial Complexes}\label{sec:polycomplex}
In this section, we shall consider the divdiv complex and establish two related polynomial complexes. We assume $\Omega \subset \mathbb R^3$ is a bounded and Lipschitz domain, which is topologically trivial in the sense that it is homeomorphic to a ball. Without loss of generality, we also assume $(0,0,0) \in \Omega$. 

Recall that a Hilbert complex is a sequence of Hilbert spaces connected by a sequence of linear operators satisfying the property: the composition of two consecutive operators is vanished. A Hilbert complex is exact means the range of each map is the kernel of the succeeding map. As $\Omega$ is topologically trivial, the following de Rham Complex of $\Omega$ is exact
\begin{equation}\label{eq:derham}
0\xrightarrow{} H^1(\Omega)\xrightarrow{\grad}\boldsymbol H(\curl;\Omega)\xrightarrow{\curl}\boldsymbol H(\div;\Omega)\xrightarrow{\div}L^2(\Omega) \xrightarrow{}0.
\end{equation}
We will abbreviate a Hilbert complex as a complex. 

\subsection{The $\div\div$ complex}
The $\div\div$ complex in three dimensions reads as~\cite{ArnoldHu2020,PaulyZulehner2020}
\begin{equation}\label{eq:divdivcomplex3d}
\resizebox{.92\hsize}{!}{$
\boldsymbol{RT}\xrightarrow{\subset} \boldsymbol H^1(\Omega;\mathbb R^3)\xrightarrow{\dev\grad}\boldsymbol H(\sym\curl,\Omega;\mathbb T)\xrightarrow{\sym\curl} \boldsymbol H(\div\div, \Omega;\mathbb S) \xrightarrow{\div{\div}} L^2(\Omega)\xrightarrow{}0,
$}
\end{equation}
where $\boldsymbol{RT}:= \{a\boldsymbol x + \boldsymbol b: a\in \mathbb R, \boldsymbol b \in \mathbb R^3\}$ is the space of shape funcions of the lowest order Raviart-Thomas element~\cite{RaviartThomas1977}.
For completeness, we prove the exactness of the complex~\eqref{eq:divdivcomplex3d} following~\cite{PaulyZulehner2020}.

\begin{theorem}\label{thm:divdivcomplex}
Assume $\Omega$ is a bounded and topologically trivial Lipschitz domain in $\mathbb R^3$. Then~\eqref{eq:divdivcomplex3d} is an exact Hilbert complex.   
\end{theorem}
\begin{proof}
We verify the composition of consecutive operators is vanished from the left to the right. Take a function $\boldsymbol v = a\boldsymbol x + \boldsymbol b\in \bs{RT}$, then $\grad \boldsymbol v = a \boldsymbol I$ and $\dev \boldsymbol I = 0$. For any $\boldsymbol v\in \mathcal C^2(\Omega;\mathbb R^3)$, it holds from~\eqref{eq:curlgrad} that
\begin{align*}
\sym\curl\dev\grad\boldsymbol v&=\sym\curl\left(\grad\boldsymbol v-\frac{1}{3}(\div\boldsymbol v) \boldsymbol I\right)=-\frac{1}{3}\sym\curl((\div\boldsymbol v) \boldsymbol I) \\
&=\frac{1}{3}\sym\mskw(\grad(\div\boldsymbol v))=0.
\end{align*}
By the density argument, we get $\sym\curl\dev\grad\boldsymbol H^1(\Omega;\mathbb R^3)=\bs0$. 
For any $\bs\tau\in\mathcal C^3(\Omega;\mathbb T)$, 
\[
\div\div\sym\curl\bs\tau= \frac{1}{2}\nabla \cdot (\nabla \times \boldsymbol A^{\intercal} - \boldsymbol A\times\nabla )\cdot \nabla =0.
\]
Again by the density argument, $\div\div\sym\curl\boldsymbol H(\sym\curl,\Omega;\mathbb T)=0$.
Thus~\eqref{eq:divdivcomplex3d} is a complex.

We then verify the exactness of~\eqref{eq:divdivcomplex3d} from the right to the left.

\medskip
\noindent {\em 1. $\div\div\boldsymbol H(\div\div, \Omega;\mathbb S)=L^2(\Omega)$.} 

Recursively applying the exactness of de Rham complex~\eqref{eq:derham}, we can prove $\div\div\boldsymbol H(\div\div, \Omega;\mathbb M)=L^2(\Omega)$ without the symmetry requirement, where the space $\boldsymbol H(\div\div, \Omega;\mathbb M)=\{\boldsymbol \tau \in L^2(\Omega; \mathbb M): \div\div \boldsymbol \tau \in L^2(\Omega)\}$.

Any skew-symmetric $\boldsymbol \tau$ can be written as $\boldsymbol \tau =\mskw\boldsymbol v$ for $\boldsymbol v = \vskw (\boldsymbol \tau)$. Assume $\boldsymbol v\in \mathcal C^2(\Omega; \mathbb R^3)$, it follows from~\eqref{eq:divmskw} that
\begin{equation}\label{eq:divdivskw0}
\div{\div}\boldsymbol \tau=\div{\div}\mskw\boldsymbol v= - \div(\curl\boldsymbol v)=0.
\end{equation}
Since $\div\div\bs\tau=0$ for any smooth skew-symmetric tensor field $\bs\tau$, we obtain
$$
\div\div\boldsymbol H(\div\div, \Omega;\mathbb S)=\div\div\boldsymbol H(\div\div, \Omega;\mathbb M)=L^2(\Omega).
$$

\medskip
\noindent {\em 2. $\boldsymbol H(\div\div, \Omega;\mathbb S) \cap\ker(\div\div)=\sym\curl\boldsymbol H(\sym\curl,\Omega;\mathbb T)$, i.e. if $\div\div\boldsymbol \sigma = 0$ and $\bs\sigma\in \boldsymbol H(\div\div, \Omega;\mathbb S)$, then there exists a $\bs\tau\in\boldsymbol H(\sym\curl,\Omega;\mathbb T)$, s.t. $\boldsymbol \sigma = \sym\curl\bs\tau$}. 

Since $\div(\div\boldsymbol \sigma) = 0$, by the exactness of the de Rham complex and identity ~\eqref{eq:divmskw},  there exists $\boldsymbol v \in\boldsymbol L^2(\Omega;\mathbb R^3)$ such that
\[
\div\bs\sigma=\curl\boldsymbol v=-\div(\mskw\boldsymbol v).
\]
Namely $\div (\boldsymbol \sigma + \mskw\boldsymbol v) = 0$. Again by the exactness of the de Rham complex, there exists $\widetilde{\bs\tau} \in\boldsymbol H^1(\Omega;\mathbb M)$ such that
\[
\bs\sigma=-\mskw\boldsymbol v+\curl\widetilde{\bs\tau}.
\]
By the symmetry of $\bs\sigma$, we have
$$
\bs\sigma=\sym\curl\widetilde{\bs\tau}=\sym\curl(\dev\widetilde{\bs\tau})+\frac{1}{3}\sym\curl\left((\tr\widetilde{\bs\tau})\boldsymbol I\right).
$$
From~\eqref{eq:curlgrad} we get
$$
\sym\curl\left((\tr\widetilde{\bs\tau})\boldsymbol I\right)=-\sym(\mskw\grad(\tr\widetilde{\bs\tau}))=0,
$$
which indicates $\bs\sigma=\sym\curl\bs\tau$ with $\bs\tau=\dev\widetilde{\bs\tau}\in\boldsymbol H^1(\Omega;\mathbb T)$. 

\medskip
\noindent {\em 3. $\boldsymbol H(\sym\curl,\Omega;\mathbb T)\cap\ker(\sym\curl)=\dev\grad\boldsymbol H^1(\Omega;\mathbb R^3)$, i.e. if $\sym\curl\boldsymbol \tau = \boldsymbol 0$ and $\bs\tau\in \boldsymbol H(\sym\curl,\Omega;\mathbb T)$, then there exists a $\boldsymbol v\in\boldsymbol H^1(\Omega;\mathbb R^3)$, s.t. $\boldsymbol \tau = \dev\grad\boldsymbol v$}.

Since $\sym(\curl\boldsymbol \tau)= \boldsymbol 0$ and $\tr\bs\tau=0$, we have from~\eqref{eq:skwcurl} that
\[
\curl\bs\tau=\skw(\curl\bs\tau)=\frac{1}{2}\mskw\left[\div(\bs\tau^{\intercal})-\grad(\tr(\bs\tau))\right]=\frac{1}{2}\mskw(\div(\bs\tau^{\intercal})).
\]
Then by~\eqref{eq:divmskw},
\[
\curl(\div(\bs\tau^{\intercal}))=-\div(\mskw\div(\bs\tau^{\intercal}))=-2\div(\curl\bs\tau)=\bs0.
\]
Thus there exists $w\in H^1(\Omega)$ satsifying $\div(\bs\tau^{\intercal})=2\grad w$, which together with~\eqref{eq:curlgrad} implies
\[
\curl\bs\tau=\mskw\grad w=-\curl(w\boldsymbol I).
\]
Namely $\curl (\boldsymbol \tau + w\boldsymbol I) = 0$. Hence there exists $\boldsymbol v\in \boldsymbol H^1(\Omega;\mathbb R^3)$ such that $\bs\tau=-w\boldsymbol I+\grad \boldsymbol v$. Noting that $\bs\tau$ is trace-free, we achieve
\[
\bs\tau=\dev\bs\tau=\dev\grad \boldsymbol v.
\]

\medskip
\noindent {\em 4. $\boldsymbol H^1(\Omega;\mathbb R^3)\cap\ker(\dev\grad)=\bs{RT}$, i.e. if $\dev\grad\boldsymbol v = \boldsymbol 0$ and $\boldsymbol v\in \boldsymbol H^1(\Omega;\mathbb R^3)$, then $\boldsymbol v\in\bs{RT}$}. 

Notice that
\begin{equation}\label{eq:20210205-1}
\grad\boldsymbol v =\frac{1}{3}(\div\boldsymbol v)\boldsymbol I.
\end{equation}
Apply $\curl$ on both sides of~\eqref{eq:20210205-1} and use~\eqref{eq:curlgrad} to get
$$
-\mskw\grad(\div\boldsymbol v)=\curl((\div\boldsymbol v)\boldsymbol I)=3\curl(\grad\boldsymbol v)=\bs0.
$$
Hence $\div\boldsymbol v$ is a constant, which combined with~\eqref{eq:20210205-1} implies that $\boldsymbol v$ is a linear function. Assume $\boldsymbol v=\boldsymbol A\boldsymbol x+\boldsymbol b$ with $\boldsymbol A\in\mathbb M$ and $\boldsymbol b\in\mathbb R^3$, then~\eqref{eq:20210205-1} becomes $\boldsymbol A=\frac{1}{3}\tr(\boldsymbol A)\boldsymbol I$, i.e. $\boldsymbol A$ is diagonal and consequently $\boldsymbol v\in\bs{RT}$.

Thus the complex~\eqref{eq:divdivcomplex3d} is exact.
\end{proof}

\subsection{A polynomial divdiv complex}
Given a bounded domain $G\subset\mathbb{R}^{3}$ and a
non-negative integer $m$, 
let $\mathbb P_m(G)$ stand for the set of all polynomials in $G$ with the total degree no more than $m$, and $\mathbb P_m(G; \mathbb{X})$ with $\mathbb X$ being $\mathbb{M}, \mathbb{S}, \mathbb{K}, \mathbb{T}$ or $\mathbb{R}^3$ denote the tensor or vector version. Recall that $\dim \mathbb P_{k}(G) = { k + 3 \choose 3}$, $\dim \mathbb M = 9, \dim \mathbb S = 6, \dim \mathbb K = 3$ and $\dim \mathbb T = 8$. For a linear operator $T$ defined on a finite dimensional linear space $V$, we have the relation
\begin{equation}\label{eq:dim}
\dim V = \dim \ker(T) + \dim \img(T),
\end{equation}
which can be used to count $\dim \img(T)$ provided the space $\ker(T)$ is identified and vice verse. 

The polynomial de Rham complex is
\begin{equation}\label{eq:polyderham}
\mathbb R\xrightarrow{\subset} \mathbb P_{k+1}(\Omega)\xrightarrow{\grad} \mathbb P_{k}(\Omega;\mathbb R^3)\xrightarrow{\curl}\mathbb P_{k-1}(\Omega;\mathbb R^3) \xrightarrow{\div} \mathbb P_{k-2}(\Omega)\xrightarrow{}0.
\end{equation}
As $\Omega$ is topologically trivial, complex~\eqref{eq:polyderham} is also exact, i.e., the range of each map is the kernel of the succeeding map. 

\begin{lemma}
The polynomial complex
\begin{equation}\label{eq:divdivcomplex3dPoly}
\resizebox{.92\hsize}{!}{$
\boldsymbol{RT}\xrightarrow{\subset} \mathbb P_{k+2}(\Omega; \mathbb R^3)\xrightarrow{\dev\grad}\mathbb P_{k+1}(\Omega; \mathbb T)\xrightarrow{\sym\curl} \mathbb P_k(\Omega; \mathbb S) \xrightarrow{\div{\div}} \mathbb P_{k-2}(\Omega)\xrightarrow{}0
$}
\end{equation}
is exact.
\end{lemma}
\begin{proof}
Clearly~\eqref{eq:divdivcomplex3dPoly} is a complex due to Theorem~\ref{thm:divdivcomplex}.
We then verify the exactness.

\medskip
\noindent {\em 1. $\mathbb P_{k+2}(\Omega; \mathbb R^3) \cap\ker(\dev\grad)=\bs{RT}$.} By the exactness of the complex~\eqref{eq:divdivcomplex3d}, 
\[
\bs{RT}\subseteq\mathbb P_{k+2}(\Omega; \mathbb R^3) \cap\ker(\dev\grad)\subseteq\boldsymbol H^1(\Omega;\mathbb R^3) \cap\ker(\dev\grad)=\bs{RT}.
\]

\medskip
\noindent {\em 2. $\mathbb P_{k+1}(\Omega; \mathbb T)\cap\ker(\sym\curl)=\dev\grad\mathbb P_{k+2}(\Omega;\mathbb R^3)$, i.e. if $\sym\curl\boldsymbol \tau = 0$ and $\boldsymbol\tau\in\mathbb P_{k+1}(\Omega; \mathbb T)$, then there exists a $\boldsymbol v\in\mathbb P_{k+2}(\Omega;\mathbb R^3)$, s.t. $\bs\tau =\dev\grad\boldsymbol v$}.

By $\sym\curl\boldsymbol \tau = 0$, there exists $\boldsymbol v\in\boldsymbol H^1(\Omega; \mathbb R^3)$ satisfying $\boldsymbol\tau=\dev\grad\boldsymbol v$, i.e. $\boldsymbol\tau=\grad\boldsymbol v-\frac{1}{3}(\div\boldsymbol v)\boldsymbol I$. Then we get from~\eqref{eq:curlgrad} that
\[
\mskw(\grad\div\boldsymbol v)=-\curl((\div\boldsymbol v)\boldsymbol I)=3\curl(\boldsymbol\tau-\grad\boldsymbol v)=3\curl\boldsymbol\tau, 
\]
which implies $\grad\div\boldsymbol v=3\vskw(\curl\boldsymbol\tau)\in\mathbb P_{k}(\Omega; \mathbb R^3)$. Hence
$\div\boldsymbol v\in\mathbb P_{k+1}(\Omega)$. And thus $\grad\boldsymbol v=\boldsymbol\tau+\frac{1}{3}(\div\boldsymbol v)\boldsymbol I\in\mathbb P_{k+1}(\Omega; \mathbb M)$. As a result $\boldsymbol v\in\mathbb P_{k+2}(\Omega; \mathbb R^3)$. 

\medskip
\noindent {\em 3. $\div\div\mathbb P_k(\Omega;\mathbb S)=\mathbb P_{k-2}(\Omega)$.} 
Recursively applying the exactness of de Rham complex~\eqref{eq:polyderham}, we can prove $\div\div\mathbb P_k(\Omega; \mathbb M)=\mathbb P_{k-2}(\Omega)$. 
Then from~\eqref{eq:divdivskw0} we have that
\[
\div{\div}\, \mathbb P_k(\Omega; \mathbb S)=\div{\div}\, \mathbb P_k(\Omega; \mathbb M)=\mathbb P_{k-2}(\Omega).
\]

\medskip
\noindent {\em 4. $\mathbb P_k(\Omega; \mathbb S) \cap\ker(\div\div)=\sym\curl\mathbb P_{k+1}(\Omega;\mathbb T)$}. 

Obviously $\sym\curl\mathbb P_{k+1}(\Omega;\mathbb T) \subseteq (\mathbb P_k(\Omega; \mathbb S) \cap\ker(\div\div))$.
As $\div\div: \mathbb P_k(\Omega; \mathbb S)\to\mathbb P_{k-2}(\Omega)$ is surjective by step 3, using~\eqref{eq:dim}, we have
\begin{align}
\dim\mathbb P_k(\Omega; \mathbb S)\cap\ker(\div{\div})&=\dim\mathbb P_k(\Omega; \mathbb S)-\dim\mathbb P_{k-2}(\Omega) \notag\\
& = 6 { k + 3 \choose 3} - { k + 1 \choose 3} \notag\\
&= \frac{1}{6}(5k^3+36k^2+67k+36). \label{eq:20200507-2} 
\end{align}
Thank to results in steps 1 and 2, we can count the dimension of $\sym\curl \, \mathbb P_{k+1}(\Omega; \mathbb T)$
\begin{align}
\dim\sym\curl \, \mathbb P_{k+1}(\Omega; \mathbb T) &= \dim\mathbb P_{k+1}(\Omega; \mathbb T)-\dim\dev\grad\mathbb P_{k+2}(\Omega; \mathbb R^3) \notag\\
& = \dim\mathbb P_{k+1}(\Omega; \mathbb T)-(\dim \mathbb P_{k+2}(\Omega; \mathbb R^3) - \dim \bs{RT}) \notag\\
& = 8 { k + 4 \choose 3} - 3{ k + 5 \choose 3} + 4 \notag\\
&= \frac{1}{6}(5k^3+36k^2+67k+36). \label{eq:20200507-3}
\end{align}
We conclude that $\mathbb P_k(\Omega; \mathbb S)\cap\ker(\div{\div})=\sym\curl \, \mathbb P_{k+1}(\Omega; \mathbb T)$ as the dimensions matches, cf.~\eqref{eq:20200507-2} and~\eqref{eq:20200507-3}.

Therefore the complex~\eqref{eq:divdivcomplex3dPoly} is exact.
\end{proof}

\subsection{A Koszul complex}
The Koszul complex corresponding to the de Rham complex~\eqref{eq:polyderham} is
\begin{equation}\label{eq:Koszul}
0
\xrightarrow{}
\mathbb P_{k-2}(\Omega) 
\xrightarrow{\boldsymbol x}
\mathbb P_{k-1}(\Omega;\mathbb R^3)
\xrightarrow{\times \boldsymbol  x}
\mathbb P_{k}(\Omega;\mathbb R^3)
\xrightarrow{\cdot \boldsymbol  x} \mathbb P_{k+1}(\Omega) \xrightarrow{} 0,
\end{equation}
where the operators are appended to the right of the polynomial, i.e. $v \boldsymbol x$, $\boldsymbol v\times \boldsymbol x $, or $\boldsymbol v\cdot \boldsymbol x$. 
The following complex is a generalization of the Koszul complex~\eqref{eq:Koszul} to the divdiv complex~\eqref{eq:divdivcomplex3dPoly}, where operator $\boldsymbol \pi_{RT}: \mathcal C^1(\Omega; \mathbb R^3)\to \boldsymbol{RT}$ is defined as
\[
\boldsymbol \pi_{RT}\boldsymbol  v:=\boldsymbol  v(0,0,0)+\frac{1}{3}(\div\boldsymbol  v)(0,0,0)\boldsymbol  x,
\]
and other operators are appended to the right of the polynomial, i.e., $p \boldsymbol x\boldsymbol x^{\intercal}$, $\boldsymbol \tau\times \boldsymbol x $, or $\boldsymbol \tau\cdot \boldsymbol x$. The Koszul operator $\boldsymbol x\boldsymbol x^{\intercal}$ can be constructed based on Poincar\'e operators constructed in~\cite{ChristiansenHuSande2020}, but others are simpler. 

\begin{lemma}\label{lem:Koszul}
The following polynomial sequence
\begin{equation}\label{eq:divdivKoszulcomplex3dPoly}
\resizebox{.922\hsize}{!}{$
0\xrightarrow{\subset}\mathbb P_{k-2}(\Omega) \xrightarrow{\boldsymbol x\boldsymbol x^{\intercal}} \mathbb P_k(\Omega; \mathbb S) \xrightarrow{\times\boldsymbol x} \mathbb P_{k+1}(\Omega; \mathbb T)\xrightarrow{\cdot\boldsymbol x} \mathbb P_{k+2}(\Omega; \mathbb R^3)\xrightarrow{\boldsymbol \pi_{RT}}\boldsymbol{RT}\xrightarrow{}0
$}
\end{equation}
is an exact Hilbert complex.
\end{lemma}
\begin{proof}
In the sequence~\eqref{eq:divdivKoszulcomplex3dPoly} only the mapping $\mathbb P_k(\Omega; \mathbb S) \stackrel{\times \boldsymbol x}{\longrightarrow}  \mathbb P_{k+1}(\Omega; \mathbb T)$ is less obvious, which can be justified by the identity~\eqref{eq:trcross}. 

To verify~\eqref{eq:divdivKoszulcomplex3dPoly} is a complex, we use the product rule~\eqref{eq:xuv}-\eqref{eq:xtimesuv}:
$$
p \boldsymbol x\boldsymbol x^{\intercal} \times \boldsymbol x = p \boldsymbol x(\boldsymbol x \times \boldsymbol x )^{\intercal} = 0, \quad (\boldsymbol\tau\times\boldsymbol x)\cdot \boldsymbol x = 0.
$$
To verify $\boldsymbol \pi_{RT}(\boldsymbol \tau \cdot \boldsymbol x) = 0$ for $\boldsymbol \tau \in \mathbb P_{k+1}(\Omega; \mathbb T)$, we use the formulae
\begin{equation}\label{eq:divtaux}
\div(\boldsymbol \tau\cdot \boldsymbol x)=\div(\boldsymbol\tau^{\intercal})\cdot \boldsymbol x + \tr\boldsymbol\tau = \boldsymbol x^{\intercal}\div(\boldsymbol\tau^{\intercal}),
\end{equation}
and therefore evaluating at $\boldsymbol 0$ is zero. 


We then verify the exactness from  right-to-left.

\medskip
\noindent {\em 1. $\boldsymbol\pi_{RT}\mathbb P_{k+2}(\Omega; \mathbb R^3)=\boldsymbol{RT}$.} 

It is straightforward to verify
\begin{equation}\label{eq:piRTprop}
\boldsymbol \pi_{RT}\boldsymbol  v=\boldsymbol  v\quad \forall~\boldsymbol  v\in\boldsymbol{RT}.
\end{equation}
Namely $\boldsymbol \pi_{RT}$ is a projector. Consequently, the operator $\boldsymbol \pi_{RT}: \mathbb P_{k+2}(\Omega; \mathbb R^3)\to\boldsymbol {RT}$ is surjective as $\boldsymbol {RT}\subset \mathbb P_{1}(\Omega; \mathbb R^3)$. 

\medskip
\noindent {\em 2. $\mathbb P_{k+2}(\Omega; \mathbb R^3)\cap\ker(\boldsymbol \pi_{RT})=\mathbb P_{k+1}(\Omega; \mathbb T)\cdot\boldsymbol x$, i.e. if $\boldsymbol\pi_{RT}\boldsymbol v=\boldsymbol0$ and $\boldsymbol v\in\mathbb P_{k+2}(\Omega; \mathbb R^3)$, then there exists a $\bs\tau\in\mathbb P_{k+1}(\Omega; \mathbb T)$, s.t. $\boldsymbol v =\bs\tau\cdot\boldsymbol x$}. 

Since $\boldsymbol  v(0,0,0)=\boldsymbol 0$, by the fundamental theorem of calculus, $$\boldsymbol v = \left (\int_0^1 \grad \boldsymbol v( t \boldsymbol x) \dd t \right )\boldsymbol x.$$ Using the decomposition~\eqref{eq:devtr}, we conclude that there exist $\boldsymbol \tau_1\in\mathbb P_{k+1}(\Omega; \mathbb T)$ and $q\in\mathbb P_{k+1}(\Omega)$ such that
$\boldsymbol  v=\boldsymbol \tau_1\boldsymbol x+q\boldsymbol x$. 
Again by~\eqref{eq:divtaux}, we have
$$
\boldsymbol\pi_{RT}(q\boldsymbol x)=\boldsymbol\pi_{RT}\boldsymbol v-\boldsymbol\pi_{RT}(\boldsymbol \tau_1\boldsymbol x)=\boldsymbol 0,
$$
which indicates $(\div(q\boldsymbol  x))(0,0,0)=0$. As $\div(q\boldsymbol  x )=(\boldsymbol x\cdot\nabla)q+3q$, we conclude $q(0,0,0)=0$. Again using the fundamental theorem of calculus to conclude that there exists $\boldsymbol  q_1\in\mathbb P_{k}(\Omega; \mathbb R^3)$ such that $q=\boldsymbol  q_1^{\intercal}\boldsymbol x$.
Taking $\boldsymbol \tau=\boldsymbol \tau_1+\frac{3}{2}\boldsymbol  x\boldsymbol  q_1^{\intercal}-\frac{1}{2}\boldsymbol  q_1^{\intercal}\boldsymbol  x\boldsymbol  I\in\mathbb P_{k+1}(\Omega; \mathbb T)$, we get
\[
\boldsymbol \tau\boldsymbol x=\boldsymbol \tau_1\boldsymbol x+\boldsymbol  x\boldsymbol  q_1^{\intercal}\boldsymbol x=\boldsymbol \tau_1\boldsymbol x+q\boldsymbol x=\boldsymbol  v.
\]

\medskip
\noindent {\em 3. $\mathbb P_k(\Omega; \mathbb S)\cap\ker((\cdot)\times\boldsymbol x)=\mathbb P_{k-2}(\Omega)\boldsymbol x\boldsymbol  x^{\intercal}$, i.e. if $\boldsymbol \tau\times\boldsymbol x=\boldsymbol 0$ and $\boldsymbol \tau\in\mathbb P_k(\Omega; \mathbb S)$, then there exists a $q\in\mathbb P_{k-2}(\Omega)$, s.t. $\boldsymbol \tau=q\boldsymbol x\boldsymbol  x^{\intercal}$}. 
 
Thanks to $\boldsymbol \tau\times\boldsymbol x=\boldsymbol 0$, there exists $\boldsymbol v\in \mathbb P_{k-1}(\Omega; \mathbb R^3)$ such that $\boldsymbol \tau=\boldsymbol  v\boldsymbol  x^{\intercal}$. By the symmetry of $\boldsymbol \tau$, it follows
\[
(\boldsymbol x\boldsymbol v^{\intercal})\times\boldsymbol x=(\boldsymbol  v\boldsymbol  x^{\intercal})^{\intercal}\times\boldsymbol x=\boldsymbol \tau\times\boldsymbol x=\boldsymbol 0,
\]
which indicates $\boldsymbol v\times\boldsymbol x=\boldsymbol 0$. Then there exists $q\in\mathbb P_{k-2}(\Omega)$ satisfying $\boldsymbol  v=q\boldsymbol  x$.
Hence $\boldsymbol \tau=q\boldsymbol x\boldsymbol  x^{\intercal}$. 

\medskip
\noindent {\em 4. $\mathbb P_{k+1}(\Omega; \mathbb T)\cap\ker((\cdot)\cdot\boldsymbol x)=\mathbb P_k(\Omega; \mathbb S)\times\boldsymbol x$.} 

It follows from steps 1 and 2 that
\begin{align}
\dim(\mathbb P_{k+1}(\Omega; \mathbb T)\cap\ker((\cdot)\cdot\boldsymbol x)) &= \dim\mathbb P_{k+1}(\Omega; \mathbb T)-\dim(\mathbb P_{k+1}(\Omega; \mathbb T)\boldsymbol x) \notag\\
&= \dim\mathbb P_{k+1}(\Omega; \mathbb T)-\dim (\mathbb P_{k+2}(\Omega; \mathbb R^3)\cap\ker(\boldsymbol \pi_{RT})) \notag\\
&= \dim\mathbb P_{k+1}(\Omega; \mathbb T)-\dim \mathbb P_{k+2}(\Omega; \mathbb R^3)+4 \notag\\
&=\frac{1}{6}(5k^3+36k^2+67k+36). \label{eq:20200508}
\end{align}
And by step 3,
\begin{align*}
\dim(\mathbb P_k(\Omega; \mathbb S)\times\boldsymbol x)&=\dim\mathbb P_k(\Omega; \mathbb S)-\dim(\mathbb P_k(\Omega; \mathbb S)\cap\ker((\cdot)\times\boldsymbol x)) \\
&=\dim\mathbb P_k(\Omega; \mathbb S)-\dim(\mathbb P_{k-2}(\Omega)\boldsymbol x\boldsymbol  x^{\intercal}) \\
&=\frac{1}{6}(5k^3+36k^2+67k+36),
\end{align*}
which together with~\eqref{eq:20200508} implies $\mathbb P_{k+1}(\Omega; \mathbb T)\cap\ker((\cdot)\cdot\boldsymbol x)=\mathbb P_k(\Omega; \mathbb S)\times\boldsymbol x$.

Therefore the complex~\eqref{eq:divdivKoszulcomplex3dPoly} is exact.
\end{proof}

\subsection{Decomposition of polynomial tensors}
Those two complexes~\eqref{eq:divdivcomplex3dPoly} and~\eqref{eq:divdivKoszulcomplex3dPoly} can be combined into one double direction complex
\begin{equation*}
\resizebox{.92\hsize}{!}{$
\xymatrix{
\boldsymbol{RT}\ar@<0.4ex>[r]^-{\subset} & \mathbb P_{k+2}(\Omega; \mathbb R^3)\ar@<0.4ex>[r]^-{\dev\grad}\ar@<0.4ex>[l]^-{\boldsymbol \pi_{RT}} & \mathbb P_{k+1}(\Omega; \mathbb T)\ar@<0.4ex>[r]^-{\sym\curl}\ar@<0.4ex>[l]^-{\cdot\boldsymbol x}  & \mathbb P_k(\Omega; \mathbb S) \ar@<0.4ex>[r]^-{\div{\div}}\ar@<0.4ex>[l]^-{\times\boldsymbol x} & \mathbb P_{k-2}(\Omega)  \ar@<0.4ex>[r]^-{} \ar@<0.4ex>[l]^-{\boldsymbol x\boldsymbol x^{\intercal}}
& 0 \ar@<0.4ex>[l]^-{\supset} }.
$}
\end{equation*}
Unlike the Koszul complex for vectors functions, we do not have the identity property applied to homogenous polynomials. Fortunately decomposition of polynomial spaces using Koszul and differential operators still holds.

Let $\mathbb H_k(\Omega):=\mathbb P_k(\Omega)/\mathbb P_{k-1}(\Omega)$ be the space of homogeneous polynomials of degree $k$. Then by Euler's formula
\begin{equation}\label{eq:homogeneouspolyprop}
\boldsymbol x\cdot\nabla q=kq\quad\forall~q\in\mathbb H_k(\Omega).
\end{equation}
Due to~\eqref{eq:homogeneouspolyprop}, we have
\begin{align}\label{eq:radialderivativeprop}
\mathbb P_k(\Omega)\cap\ker(\boldsymbol x\cdot\nabla) & =\mathbb P_0(\Omega),\\
\label{eq:radialderivativeprop1}
\mathbb P_k(\Omega)\cap\ker(\boldsymbol x\cdot\nabla +\ell) &=0
\end{align}
for any positive number $\ell$.

It follows from~\eqref{eq:piRTprop} and the complex~\eqref{eq:divdivKoszulcomplex3dPoly} that
\[
\mathbb P_{k+2}(\Omega; \mathbb R^3)= \mathbb P_{k+1}(\Omega; \mathbb T)\boldsymbol x\oplus\boldsymbol{RT}.
\]
We then move to the space $\mathbb P_{k+1}(\Omega; \mathbb T)$.
\begin{lemma}
We have the decomposition
\begin{equation}\label{eq:polyspacedecomp2}
\mathbb P_{k+1}(\Omega; \mathbb T) = (\mathbb P_k(\Omega; \mathbb S)\times\boldsymbol x)\oplus\dev\grad\mathbb P_{k+2}(\Omega; \mathbb R^3).
\end{equation}
\end{lemma}
\begin{proof}
Let us count the dimension. 
$$
\dim \mathbb P_{k+1}(\Omega; \mathbb T) = 8 { k + 4 \choose 3},
$$
while by the exactness of the Koszul complex~\eqref{eq:divdivKoszulcomplex3dPoly}
\begin{align*}
\dim \mathbb P_k(\Omega; \mathbb S)\times\boldsymbol x &= \dim \mathbb P_k(\Omega; \mathbb S) - \boldsymbol x\boldsymbol x^{\intercal} \mathbb P_{k-2}(\Omega) \\
&= 6 { k + 3 \choose 3} -  { k + 1 \choose 3}, \\
\dim \dev\grad\mathbb P_{k+2}(\Omega; \mathbb R^3) &= \dim \mathbb P_{k+2}(\Omega; \mathbb R^3) - \ker(\dev\grad)\\
& = 3 { k + 5 \choose 3} - 4. 
\end{align*} 
By direct computation, the dimension of space in the left hand side is the summation of the dimension of the two spaces in the right hand side in~\eqref{eq:polyspacedecomp2}. So we only need to prove that the sum in~\eqref{eq:polyspacedecomp2} is a direct sum.

Take $\bs\tau=\dev\grad\boldsymbol q$  for some $\boldsymbol q\in\mathbb P_{k+2}(\Omega; \mathbb R^3)$, and also assume $\boldsymbol \tau\in\mathbb P_k(\Omega; \mathbb S)\times\boldsymbol x$. We have $\bs\tau\cdot\boldsymbol x=(\dev\grad\boldsymbol q)\cdot\boldsymbol x=\boldsymbol 0$, that is
\begin{equation}\label{eq:20210206}
(\grad\boldsymbol q)\cdot\boldsymbol x=\frac{1}{3}(\div\boldsymbol q)\boldsymbol x.
\end{equation}
Since $\div((\grad\boldsymbol q)\cdot\boldsymbol x)=(1+\boldsymbol x\cdot\grad)\div\boldsymbol q$, applying the divergence operator $\div$ on both side of~\eqref{eq:20210206} gives
$$
(1+\boldsymbol x\cdot\grad)\div\boldsymbol q=\frac{1}{3}(3+\boldsymbol x\cdot\grad)\div\boldsymbol q.
$$
Hence $(\boldsymbol x\cdot\grad)\div\boldsymbol q=0$, which together with~\eqref{eq:radialderivativeprop} indicates $\div\boldsymbol q\in\mathbb P_{0}(\Omega)$. Due to~\eqref{eq:20210206}, $(\grad\boldsymbol q)\cdot\boldsymbol x$ is a linear function. It follows from ~\eqref{eq:homogeneouspolyprop} that $\boldsymbol q\in\mathbb P_1(\Omega)$ and  $\bs\tau=\dev\grad\boldsymbol q\in\mathbb P_{0}(\Omega;\mathbb T)$, which together with $\bs\tau\cdot\boldsymbol x=\bs0$ implies $\bs\tau=\bs0$. 
\end{proof}

Finally we present a decomposition of space $\mathbb P_{k}(\Omega; \mathbb S)$.
Let
\[
\mathbb C_k(\Omega; \mathbb S):=\sym \curl \, \mathbb  P_{k+1}(\Omega; \mathbb T),\quad \mathbb C_k^{\oplus}(\Omega; \mathbb S):=\boldsymbol  x\boldsymbol  x^{\intercal}\mathbb P_{k-2}(\Omega).
\]
Their dimensions are
\begin{equation}\label{eq:dimC}
\dim\mathbb C_k(\Omega; \mathbb S)=\frac{1}{6}(5k^3+36k^2+67k+36),\quad \dim\mathbb C_k^{\oplus}(\Omega; \mathbb S)=\frac{1}{6}(k^3-k).
\end{equation}
The calculation of $\dim\mathbb C_k^{\oplus}(\Omega; \mathbb S)$ is easy and $\dim\mathbb C_k(\Omega; \mathbb S)$ is detailed in~\eqref{eq:20200507-3}. 

\begin{lemma}\label{lem:symmpolyspacedirectsum}
We have
\begin{enumerate}[\rm (i)]
\item $\displaystyle \div\div (\boldsymbol x\boldsymbol x^{\intercal} q) = (k+4)(k+3) q$ for any $q\in \mathbb H_k(\Omega).$

\medskip
\item $\div\div: \mathbb C_k^{\oplus}(\Omega; \mathbb S)\to\mathbb P_{k-2}(\Omega)$ is a bijection.

\medskip
 \item
$\displaystyle
\mathbb P_{k}(\Omega; \mathbb S)=\mathbb C_k(\Omega; \mathbb S)\oplus \mathbb C_k^{\oplus}(\Omega; \mathbb S).
$
\end{enumerate}
\end{lemma}
\begin{proof}
Since
$
{\div}(\boldsymbol  x\boldsymbol  x^{\intercal}q)=(\div(\boldsymbol  x q)+q)\boldsymbol  x
$ and $\div(\boldsymbol  x q)=(\boldsymbol x\cdot\nabla)q+3q$,  we get
\begin{equation}\label{eq:20200512}
\div\div (\boldsymbol x\boldsymbol x^{\intercal} q)=\div(((\boldsymbol x\cdot\nabla+4)q)\boldsymbol x)=(\boldsymbol x\cdot\nabla+3)(\boldsymbol x\cdot\nabla+4)q.
\end{equation}
Hence property (i) follows from~\eqref{eq:homogeneouspolyprop}.
Property (ii) is obtained by writing $\mathbb P_{k-2}(\Omega) = \bigoplus_{i=0}^{k-2} \mathbb H_i(\Omega)$.
Now we prove property (iii). 
First the dimension of space in the left hand side is the summation of the dimension of the two spaces in the right hand side in (iii).
Assume $q\in\mathbb P_{k-2}(\Omega)$ satisfies $\boldsymbol x\boldsymbol  x^{\intercal}q\in\mathbb C_k(\Omega; \mathbb S)$, which means
\[
\div{\div}(\boldsymbol  x\boldsymbol  x^{\intercal}q)=0.
\]
Thus $q=0$ from~\eqref{eq:20200512} and~\eqref{eq:radialderivativeprop1} and consequently property (iii) holds.
\end{proof}

For the simplification of the degree of freedoms, we need another decomposition of the symmetric tensor polynomial space, which can be derived from the polynomial Hessian complex 
\begin{equation}\label{eq:hesscomplex3dPolydouble}
\resizebox{.91\hsize}{!}{$
\xymatrix{
\mathbb P_1(\Omega)\ar@<0.4ex>[r]^-{\subset} & \mathbb P_{k+2}(\Omega)\ar@<0.4ex>[r]^-{\hess}\ar@<0.4ex>[l]^-{\pi_{1}v} & \mathbb P_{k}(\Omega;\mathbb S)\ar@<0.4ex>[r]^-{\curl}\ar@<0.4ex>[l]^-{\boldsymbol x^{\intercal}\boldsymbol\tau\boldsymbol x}  & \mathbb P_{k-1}(\Omega;\mathbb T) \ar@<0.4ex>[r]^-{{\div}}\ar@<0.4ex>[l]^-{\sym(\boldsymbol\tau\times\boldsymbol x)} & \mathbb P_{k-2}(\Omega;\mathbb R^3)  \ar@<0.4ex>[r]^-{} \ar@<0.4ex>[l]^-{\dev(\boldsymbol v\boldsymbol x^{\intercal})}
& 0 \ar@<0.4ex>[l]^-{\supset} },
$}
\end{equation}
where $\pi_{1}v:=v(0,0,0)+\boldsymbol  x^{\intercal}(\nabla v)(0,0,0).$
A proof of the exactness of~\eqref{eq:hesscomplex3dPolydouble} is similar to that of Lemma~\ref{lem:Koszul} and can be found in~\cite{Chen;Huang:2020Discrete}. 
Based on~\eqref{eq:hesscomplex3dPolydouble}, we have the following decomposition of symmetric polynomial  tensors. 

\begin{lemma}\label{lem:PkS}
It holds
 \begin{equation}\label{eq:hesspolyspacedecomp2}
\mathbb P_{k}(\Omega; \mathbb S) = \nabla^2 \mathbb P_{k+2}(\Omega)\oplus\sym(\mathbb P_{k-1}(\Omega; \mathbb T)\times\boldsymbol x).
\end{equation}
\end{lemma}
\begin{proof}
Obviously the space on the right is contained in the space on the left. We then count the dimensions of spaces on both sides:
\begin{align}
\dim \mathbb P_{k}(\Omega; \mathbb S) &= 6 {k+3 \choose 3} = (k+3)(k+2)(k+1), \notag\\
\dim \nabla^2 \mathbb P_{k+2}(\Omega) &= \dim \mathbb P_{k+2}(\Omega)  - \dim \mathbb P_1(\Omega) = {k+5 \choose 3} - 4, \notag\\
\dim \sym(\mathbb P_{k-1}(\Omega; \mathbb T)\times\boldsymbol x) & = \dim \mathbb P_{k-1}(\Omega; \mathbb T) - \dim \mathbb P_{k-2}(\Omega;\mathbb R^3) \notag \\
& = 8  {k+2 \choose 3} - 3 {k+1 \choose 3} = \frac{1}{6} (k + 1) k (5 k + 19) \label{eq:dimsymx}.
\end{align}
Then by direct calculation, 
$$
\dim \nabla^2 \mathbb P_{k+2}(\Omega) + \dim \sym(\mathbb P_{k-1}(\Omega; \mathbb T)\times\boldsymbol x) = \dim \mathbb P_{k}(\Omega; \mathbb S) = k^3+6k^2+11k+6.
$$
We only need to prove that the sum is direct.

For any $\boldsymbol\tau=\nabla^2q$ with $q\in\mathbb P_{k+2}(\Omega)$ satisfying $\boldsymbol\tau\in\sym(\mathbb P_{k-1}(\Omega; \mathbb T)\times\boldsymbol x)$, it follows $(\boldsymbol x\cdot\nabla)((\boldsymbol x\cdot\nabla)q-q)=\boldsymbol x^{\intercal}(\nabla^2q)\boldsymbol x=0$. Applying~\eqref{eq:radialderivativeprop} and~\eqref{eq:homogeneouspolyprop}, we get $q\in\mathbb P_{1}(\Omega)$ and $\nabla^2q = 0$. Thus the decomposition~\eqref{eq:hesspolyspacedecomp2} holds.
\end{proof}

Similarly for a two dimensional domain $F\subset \mathbb R^2$, we have the following divdiv polynomial complex and its Koszul complex
\begin{equation}\label{eq:divdivcomplexPolydouble2D}
\xymatrix{
\boldsymbol{RT}\ar@<0.4ex>[r]^-{\subset} & \; \mathbb P_{k+1}(F;\mathbb R^2)\; \ar@<0.4ex>[r]^-{\sym\curl_F}\ar@<0.4ex>[l]^-{\boldsymbol \pi_{RT}}  & \; \mathbb P_k(F;\mathbb S) \ar@<0.4ex>[r]^-{\div _F {\div}_F}\; \ar@<0.4ex>[l]^-{\cdot\boldsymbol x^{\bot}} & \; \mathbb P_{k-2}(F)  \; \ar@<0.4ex>[r]^-{} \ar@<0.4ex>[l]^-{\boldsymbol x\boldsymbol x^{\intercal}}
& 0 \ar@<0.4ex>[l]^-{\supset} },
\end{equation}
where $\boldsymbol \pi_{RT}\boldsymbol  v:=\boldsymbol  v(0,0)+\frac{1}{2}(\div\boldsymbol  v)(0,0)\boldsymbol  x$, $\boldsymbol x^{\bot} = (x_2, - x_1)^{\intercal}$ is the rotation of $\boldsymbol x = (x_1, x_2)^{\intercal}$. A two dimensional Hessian polynomial complex and its Koszul complex are
\begin{equation}\label{eq:hessiancomplexPolydouble2D}
\xymatrix{
\mathbb P_1(F)\ar@<0.4ex>[r]^-{\subset} & \; \mathbb P_{k+1}(F)\; \ar@<0.4ex>[r]^-{\nabla^2_F}\ar@<0.4ex>[l]^-{\pi_1}  & \; \mathbb P_k(F;\mathbb S) \ar@<0.4ex>[r]^-{{\rm rot}_F}\; \ar@<0.4ex>[l]^-{\boldsymbol x^{\intercal}\boldsymbol \tau\boldsymbol x} & \; \mathbb P_{k-2}(F)  \; \ar@<0.4ex>[r]^-{} \ar@<0.4ex>[l]^-{\sym(\boldsymbol x^{\bot}\boldsymbol v^{\intercal})}
& 0 \ar@<0.4ex>[l]^-{\supset} },
\end{equation}
where $\pi_{1}v:=v(0,0)+\boldsymbol  x^{\intercal}(\nabla v)(0,0).$ Verification of the exactness of these two complexes and corresponding space decompositions can be found in~\cite{ChenHuang2020}.

\section{Green's Identities and Traces}\label{sec:greentrace}
We first present a Green's identity based on which we can characterize two traces of $ \boldsymbol{H}(\div\div, \Omega; \mathbb{S}) $ on polyhedrons and give a sufficient continuity condition for a piecewise smooth function to be in $ \boldsymbol{H}(\div\div, \Omega; \mathbb{S}) $. 

\subsection{Notation}
Let $\{\mathcal {T}_h\}_{h>0}$ be a regular family of polyhedral meshes
of $\Omega$. Our finite element spaces are constructed for tetrahedrons but some results, e.g., traces and Green's formulae etc, hold for general polyhedrons.   For each element $K\in\mathcal{T}_h$, denote by $\boldsymbol{n}_K $ the
unit outward normal vector to $\partial K$,  which will be abbreviated as $\boldsymbol{n}$ for simplicity.
Let $\mathcal{F}_h$, $\mathcal{F}^i_h$, $\mathcal{E}_h$, $\mathcal{E}^i_h$, $\mathcal{V}_h$ and $\mathcal{V}^i_h$ be the union of all faces, interior faces, all edges, interior edges, vertices and interior vertices
of the partition $\mathcal {T}_h$, respectively.
For any $F\in\mathcal{F}_h$,
fix a unit normal vector $\boldsymbol{n}_F$ and two unit tangent vectors $\boldsymbol{t}_{F,1}$ and $\boldsymbol{t}_{F,2}$, which will be abbreviated as $\boldsymbol{t}_{1}$ and $\boldsymbol{t}_{2}$ without causing any confusions.
For any $e\in\mathcal{E}_h$,
fix a unit tangent vector $\boldsymbol{t}_e$ and two unit normal vectors $\boldsymbol{n}_{e,1}$ and $\boldsymbol{n}_{e,2}$, which will be abbreviated as $\boldsymbol{n}_{1}$ and $\boldsymbol{n}_{2}$ without causing any confusions.
For $K$ being a polyhedron, denote by $\mathcal{F}(K)$, $\mathcal{E}(K)$ and $\mathcal{V}(K)$ the set of all faces, edges and vertices of $K$, respectively. 
For any $F\in\mathcal{F}_h$, let $\mathcal{E}(F)$ be the set of all edges of $F$. And for each $e\in\mathcal{E}(F)$, denote by $\boldsymbol n_{F,e}$ the unit vector
being parallel to $F$ and outward normal to $\partial F$.
Furthermore, set
\[
\mathcal{F}^i(K):=\mathcal{F}(K)\cap\mathcal{F}^i_h, \quad\mathcal{E}^i(F):=\mathcal{E}(F)\cap\mathcal{E}^i_h. 
\]

\subsection{Green's identities}
We first derive a Green's identity for smooth functions on polyhedrons.

\begin{lemma} [Green's identity in 3D]\label{lm:Green}
Let $K$ be a polyhedron, and let $\boldsymbol  \tau\in \mathcal C^2(K; \mathbb S)$ and $v\in H^2(K)$. 
Then we have
\begin{align}
(\div\div\boldsymbol \tau, v)_K&=(\boldsymbol \tau, \nabla^2v)_K -\sum_{F\in\mathcal F(K)}\sum_{e\in\mathcal E(F)}(\boldsymbol n_{F,e}^{\intercal}\boldsymbol \tau \boldsymbol n, v)_e \notag\\
&\quad - \sum_{F\in\mathcal F(K)}\left[(\boldsymbol  n^{\intercal}\boldsymbol \tau\boldsymbol  n, \partial_n v)_{F} -  ( 2\div_F(\boldsymbol \tau\boldsymbol n)+\partial_n (\boldsymbol  n^{\intercal}\boldsymbol \tau\boldsymbol  n), v)_F\right]. \label{eq:greenidentitydivdiv}
\end{align}
\end{lemma}
\begin{proof}
We start from the standard integration by parts
\begin{align*}
\begin{aligned}
(\operatorname{div}\div \boldsymbol  \tau, v)_{K} &=-(\div\boldsymbol \tau, \nabla v)_{K}+\sum_{F \in \mathcal F(K)}(\boldsymbol  n^{\intercal}\div  \boldsymbol   \tau, v)_F \\
&=\left(\boldsymbol  \tau, \nabla^{2} v\right)_{K}-\sum_{F \in \mathcal F(K)} (\boldsymbol  \tau \boldsymbol  n, \nabla v)_F+\sum_{F \in \mathcal F(K)}(\boldsymbol  n^{\intercal} \div  \boldsymbol   \tau, v)_F.
\end{aligned}
\end{align*}
We then decompose $\nabla v = \partial_nv\boldsymbol n+ \nabla_F v$ and apply the Stokes theorem to get
\begin{align*}
(\boldsymbol  \tau \boldsymbol  n, \nabla v)_F&=(\boldsymbol \tau \boldsymbol n, \partial_nv\boldsymbol n+ \nabla_F v)_F \\
&=(\boldsymbol n^{\intercal}\boldsymbol \tau \boldsymbol n, \partial_nv)_F -( \div_F(\boldsymbol \tau\boldsymbol n),  v)_F + \sum_{e\in\mathcal E(F)} (\boldsymbol n_{F,e}^{\intercal}\boldsymbol \tau \boldsymbol n, v)_e.
\end{align*}
Now we rewrite the term
$$
(\boldsymbol  n^{\intercal}\div  \boldsymbol   \tau, v)_F = (\div  (\boldsymbol  \tau \boldsymbol n), v)_F = ( \div_F(\boldsymbol \tau\boldsymbol n),  v)_F + (\partial_n \boldsymbol (\boldsymbol  n^{\intercal}\boldsymbol \tau\boldsymbol  n),  v)_{F} .
$$
Thus the Green's identity~\eqref{eq:greenidentitydivdiv} follows by merging all terms.
\end{proof}
When the domain is smooth in the sense that $\mathcal E(K)$ is an empty set, the term $\sum\limits_{F\in\mathcal F(K)}\sum\limits_{e\in\mathcal E(F)}(\boldsymbol n_{F,e}^{\intercal}\boldsymbol \tau \boldsymbol n, v)_e $ disappears. When $v$ is continuous on edge $e$, this term will define a jump of the tensor. 

A similar Green's identity in two dimensions is included here for later usage. To avoid confusion with three dimensional version, $\bs n_e$ is used to emphasize it is a normal vector of edge $e$ of a polygon $F$ and differential operators with subscript $F$ are used. 
\begin{lemma} [Green's identity in 2D]\label{lm:Green2D}
Let $F$ be a polygon, and let $\boldsymbol  \tau\in \mathcal C^2(F; \mathbb S)$ and $v\in H^2(F)$. Then we have
\begin{align}
(\div_F\div_F\boldsymbol \tau, v)_F&=(\boldsymbol \tau, \nabla_F^2v)_F -\sum_{e\in\mathcal E(K)}\sum_{\delta\in\partial e}\sign_{e,\delta}(\boldsymbol  t^{\intercal}\boldsymbol \tau\boldsymbol  n_e)(\delta)v(\delta) \notag\\
&\quad - \sum_{e\in\mathcal E(K)}\left[(\boldsymbol  n^{\intercal}\boldsymbol \tau\boldsymbol  n_e, \partial_n v)_{e}-(2\partial_{t}(\boldsymbol  t^{\intercal}\boldsymbol \tau\boldsymbol  n)+\partial_n (\boldsymbol  n_e^{\intercal}\boldsymbol \tau\boldsymbol  n_e),  v)_{e}\right], \notag
\end{align}
where
\[
\sign_{e,\delta}:=\begin{cases}
1, & \textrm{ if } \delta \textrm{ is the end point of } e, \\
-1, & \textrm{ if } \delta \textrm{ is the start point of } e.
\end{cases}
\]
\end{lemma}
Here the trace $2\partial_{t}(\boldsymbol  t^{\intercal}\boldsymbol \tau\boldsymbol  n_e)+\partial_n (\boldsymbol  n_e^{\intercal}\boldsymbol \tau\boldsymbol  n_e) = \partial_{t}(\boldsymbol  t^{\intercal}\boldsymbol \tau\boldsymbol  n_e)+\boldsymbol  n_e^{\intercal}\div\boldsymbol \tau$ is called the effective transverse shear force respectively for $\bs\tau$ being a moment and $\boldsymbol  n_e^{\intercal}\boldsymbol \tau\boldsymbol  n_e$ is the normal bending moment in the context of elastic mechanics~\cite{FengShi1996}.


\subsection{Traces and continuity across the boundary}\label{subsec:trace}
The Green's identity~\eqref{eq:greenidentitydivdiv} motives the definition of two trace operators for function $\boldsymbol \tau\in \boldsymbol{H}(\div{\div },K; \mathbb{S})$:
$$
{\rm tr}_1(\boldsymbol \tau) = \boldsymbol n^{\intercal}\boldsymbol \tau\boldsymbol  n, \quad
{\rm tr}_2(\boldsymbol \tau) = 2\div_F(\boldsymbol \tau\boldsymbol n)+\partial_n (\boldsymbol  n^{\intercal}\boldsymbol \tau\boldsymbol  n). 
$$

We first recall the trace of the space $\boldsymbol{H}(\div{\div },K; \mathbb{S})$ on the boundary of polyhedron $K$  (cf.~\cite[Lemma 3.2]{Fuhrer;Heuer;Niemi:2019ultraweak} and~\cite{Sinwel2009,PechsteinSchoeberl2018}).
Let $H_{00}^{1/2}(F)$ be the closure of $\mathcal C_0^{\infty}(F)$ with respect to the norm $\|\cdot\|_{H^{1/2}(\partial K)}$, which includes all functions in $H^{1/2}(F)$ whose continuation to the whole boundary $\partial K$ by zero belongs to $H^{1/2}(\partial K)$.
Define trace spaces
\begin{align*}
H_{n,0}^{1/2}(\partial K)&:=\{\partial_n v|_{\partial K}: v\in H^2(K)\cap H_0^1(K)\} \\
&\;=\{g\in L^2(\partial K): g|_F\in H_{00}^{1/2}(F)\;\;\forall~F\in\mathcal F(K)\}
\end{align*}
with norm
\[
\|g\|_{H_{n,0}^{1/2}(\partial K)}:=\inf_{v\in H^2(K)\cap H_0^1(K)\atop \partial_n v=g}\|v\|_2,
\]
and
\begin{align*}
H_{t,0}^{3/2}(\partial K)&:=\{v|_{\partial K}: v\in H^2(K), \partial_nv|_{\partial K}=0, v|_{e}=0 \textrm{ for each edge } e\in\mathcal E(K) \}
\end{align*}
with norm
\[
\|g\|_{H_{t,0}^{3/2}(\partial K)}:=\inf_{v\in H^2(K)\atop \partial_n v=0, v=g}\|v\|_2.
\]
Let $H_n^{-1/2}(\partial K):=(H_{n,0}^{1/2}(\partial K))'$ for $\tr_1$, and $H_t^{-3/2}(\partial K):=(H_{t,0}^{3/2}(\partial K))'$ for $\tr_2$. 


\begin{lemma}[Lemma 3.2 in~\cite{Fuhrer;Heuer;Niemi:2019ultraweak}]\label{lem:Hdivdivtrace}
For any $\boldsymbol \tau\in\boldsymbol{H}(\div{\div },K; \mathbb{S})$,  it holds
\[
\|\boldsymbol  n^{\intercal}\boldsymbol \tau\boldsymbol  n\|_{H_n^{-1/2}(\partial K)} + \|2\div_F(\boldsymbol\tau \boldsymbol n)+ \partial_n(\boldsymbol n^{\intercal} \boldsymbol \tau\boldsymbol n)\|_{H_t^{-3/2}(\partial K)}\lesssim \|\boldsymbol{\tau}\|_{\boldsymbol{H}(\div{\div },K)}.
\]
Conversely, for any $g_n\in H_n^{-1/2}(\partial K)$ and $g_t\in H_t^{-3/2}(\partial K)$, there exists some $\boldsymbol \tau\in\boldsymbol{H}(\div{\div },K; \mathbb{S})$ such that
\[
\boldsymbol  n^{\intercal}\boldsymbol \tau\boldsymbol  n|_{\partial K}=g_n, \quad 2\div_F(\boldsymbol\tau \boldsymbol n)+ \partial_n(\boldsymbol n^{\intercal} \boldsymbol \tau\boldsymbol n)=g_t,
\]
\[ 
\|\boldsymbol{\tau}\|_{\boldsymbol{H}(\div{\div },K)} \lesssim \|g_n\|_{H_n^{-1/2}(\partial K)}+\|g_t\|_{H_t^{-3/2}(\partial K)}.
\]
The hidden constants depend only the shape of the domain $K$.
\end{lemma}

Notice that the term $(\boldsymbol n_{F,e}^{\intercal}\boldsymbol \tau \boldsymbol n, v)_e$ in the Green's identity~\eqref{eq:greenidentitydivdiv} is not covered by Lemma~\ref{lem:Hdivdivtrace}.
Indeed, the full characterization of the trace of $\boldsymbol{H}(\div{\div },K; \mathbb{S})$ is defined by $(\div{\div }\bs\tau, v)-\left(\boldsymbol  \tau, \nabla^{2} v\right)_{K}$, which cannot be equivalently decoupled~\cite[Lemma~3.2]{Fuhrer;Heuer;Niemi:2019ultraweak}.
It is possible, however, to face-wisely localize the trace if imposing additional smoothness. 

We then present a sufficient continuity condition for piecewise smooth functions to be in $\boldsymbol{H}(\div{\div },\Omega; \mathbb{S})$. 

\begin{lemma}[cf. Proposition 3.6 in~\cite{Fuhrer;Heuer;Niemi:2019ultraweak}]\label{lem:Hdivdivpatching}
Let $\boldsymbol \tau\in \boldsymbol  L^2(\Omega;\mathbb S)$ such that
\begin{enumerate}[\rm (i)]
\item $\boldsymbol \tau|_K\in \boldsymbol{H}(\div{\div },K; \mathbb{S})$ for each polyhedron  $K\in\mathcal T_h$;

\smallskip
\item $(2\div_F(\boldsymbol\tau \boldsymbol n_F)+ \partial_{n_F}(\boldsymbol n^{\intercal} \boldsymbol \tau\boldsymbol n))|_F\in L^2(F)$ is single-valued for each $F\in\mathcal F_h^i$;

\smallskip
\item $(\boldsymbol  n^{\intercal}\boldsymbol \tau\boldsymbol  n)|_F\in L^2(F)$ is single-valued for each $F\in\mathcal F_h^i$;

\smallskip
\item $(\boldsymbol  n_i^{\intercal}\boldsymbol \tau\boldsymbol  n_j)|_e\in L^2(e)$ is single-valued for each $e\in\mathcal E_h^i$, $\, i, j=1, 2$,

\end{enumerate}
then $\boldsymbol \tau\in \boldsymbol{H}(\div{\div },\Omega; \mathbb{S})$.
\end{lemma}
\begin{proof}
For any $v\in \mathcal C_0^{\infty}(\Omega)$, we get from the Green's identity~\eqref{eq:greenidentitydivdiv} that
\begin{align*}
(\boldsymbol \tau, \nabla^2v)&=\sum_{K\in\mathcal T_h}(\div\div\boldsymbol \tau, v)_K+\sum_{K\in\mathcal T_h}\sum_{F\in\mathcal F^i(K)}\sum_{e\in\mathcal E^i(F)}(\boldsymbol n_{F,e}^{\intercal}\boldsymbol \tau \boldsymbol n, v)_e\\
&\quad+\sum_{K\in\mathcal T_h}\sum_{F\in\mathcal F^i(K)}\left[(\boldsymbol  n^{\intercal}\boldsymbol \tau\boldsymbol  n, \partial_n v)_{F} -  ( 2\div_F(\boldsymbol \tau\boldsymbol n)+\partial_n (\boldsymbol  n^{\intercal}\boldsymbol \tau\boldsymbol  n), v)_F\right].
\end{align*}
Since the terms in (ii)-(iv) are single-valued and each interior face is repeated twice in the summation with opposite orientation, it follows
\[
\langle\div\div\boldsymbol \tau, v\rangle=\sum_{K\in\mathcal T_h}(\div\div\boldsymbol \tau, v)_K.
\]
Thus we have $\boldsymbol \tau\in \boldsymbol{H}(\div{\div },\Omega; \mathbb{S})$ by the definition of derivatives of the distribution, and $(\div\div\boldsymbol \tau)|_K=\div\div(\boldsymbol \tau|_K)$ for each $K\in\mathcal T_h$.
\end{proof}

For any piecewise smooth $\boldsymbol \tau\in \boldsymbol{H}(\div{\div },\Omega; \mathbb{S})$,    
the single-valued term $(\boldsymbol  n_i^{\intercal}\boldsymbol \tau\boldsymbol  n_j)|_e$ in (iv) in Lemma~\ref{lem:Hdivdivpatching} implies that there is some compatible condition for $\boldsymbol \tau$ at each vertex $\delta\in\mathcal V_h^i$.
Indeed, for any $\delta\in\mathcal V_h^i$ and $F\in\mathcal F_h^i$ with $\delta$ being a vertex of $F$, let $\boldsymbol n_1=\boldsymbol t_1\times\boldsymbol n_F$ and $\boldsymbol n_2=\boldsymbol t_2\times\boldsymbol n_F$, where $\boldsymbol t_1$ and $\boldsymbol t_2$ are the unit tangential vectors of two edges of $F$ sharing $\delta$. Then by (iv) we have
\[
\resizebox{.98\hsize}{!}{$
\llbracket\boldsymbol  n_1^{\intercal}\boldsymbol \tau\boldsymbol  n_1\rrbracket_F(\delta)=\llbracket\boldsymbol  n_2^{\intercal}\boldsymbol \tau\boldsymbol  n_2\rrbracket_F(\delta)=\llbracket\boldsymbol  n_F^{\intercal}\boldsymbol \tau\boldsymbol  n_F\rrbracket_F(\delta)=\llbracket\boldsymbol  n_1^{\intercal}\boldsymbol \tau\boldsymbol  n_F\rrbracket_F(\delta)=\llbracket\boldsymbol  n_2^{\intercal}\boldsymbol \tau\boldsymbol  n_F\rrbracket_F(\delta)=0,
$}
\]
where $\llbracket\cdot\rrbracket_F$ is the jump across $F$.
Hence this suggests the tensor value at vertex as the degree of freedom when defining the finite element.


Continuity of $(\boldsymbol  n_i^{\intercal}\boldsymbol \tau\boldsymbol  n_j)|_e$ is a sufficient but not necessary condition for functions in $\boldsymbol{H}(\div{\div },\Omega; \mathbb{S})$. Sufficient and necessary conditions are presented in~\cite[Proposition 3.6]{Fuhrer;Heuer;Niemi:2019ultraweak}.

\section{Didiv Conforming Finite Elements}\label{sec:fem}
In this section we construct conforming finite element space for $\boldsymbol{H}(\div{\div },\Omega; \mathbb{S})$ and prove the unisolvence.

\subsection{Finite element spaces for symmetric tensors}
Let $K$ be a tetrahedron. 
Take the space of shape functions
\[
\boldsymbol \Sigma_{\ell,k}(K):= \mathbb C_{\ell}(K;\mathbb S)\oplus\mathbb C_k^{\oplus}(K;\mathbb S)
\]
with $k\geq 3$ and $\ell\geq \max\{k-1, 3\}$. Recall that 
\[
\mathbb C_{\ell}(K; \mathbb S)=\sym \curl \, \mathbb  P_{\ell +1}(K; \mathbb T),\quad \mathbb C_k^{\oplus}(K; \mathbb S)=\boldsymbol  x\boldsymbol  x^{\intercal}\mathbb P_{k-2}(K).
\]
 By Lemma~\ref{lem:symmpolyspacedirectsum}, we have
\[
\mathbb P_{\min\{\ell,k\}}(K;\mathbb S)\subseteq\boldsymbol \Sigma_{\ell,k}(K) \subseteq \mathbb P_{\max\{\ell,k\}}(K;\mathbb S) \quad\textrm{ and }\quad \boldsymbol \Sigma_{k,k}(K)=\mathbb P_k(K;\mathbb S).
\]
The most interesting cases are $\ell=k-1$ and $\ell = k$, which correspond to RT (incomplete polynomial) and BDM (complete polynomial) $H(\div)$-conforming elements for the vector functions, respectively. 

For each edge, we chose two normal vectors $\boldsymbol n_1$ and $\boldsymbol n_2$. The degrees of freedom are given by
\begin{align}
\boldsymbol \tau (\delta) & \quad\forall~\delta\in \mathcal V(K), \label{Hdivdivfem3ddof1}\\
(\boldsymbol  n_i^{\intercal}\boldsymbol \tau\boldsymbol n_j, q)_e & \quad\forall~q\in\mathbb P_{\ell-2}(e),  e\in\mathcal E(K),\; i,j=1,2,\label{Hdivdivfem3ddof2}\\
(\boldsymbol  n^{\intercal}\boldsymbol \tau\boldsymbol  n, q)_F & \quad\forall~q\in\mathbb P_{\ell-3}(F),  F\in\mathcal F(K),\label{Hdivdivfem3ddof3}\\
(2\div_F(\boldsymbol\tau \boldsymbol n)+ \partial_n(\boldsymbol n^{\intercal} \boldsymbol \tau\boldsymbol n), q)_F & \quad\forall~q\in\mathbb P_{\ell-1}(F),  F\in\mathcal F(K),\label{Hdivdivfem3ddof4}\\
(\boldsymbol \tau, \boldsymbol \varsigma)_K & \quad\forall~\boldsymbol \varsigma\in\nabla^2\mathbb P_{k-2}(K), \label{Hdivdivfem3ddof5} \\
(\boldsymbol \tau, \boldsymbol \varsigma)_K & \quad\forall~\boldsymbol \varsigma\in \sym(\mathbb P_{\ell-2}(K; \mathbb T)\times\boldsymbol x), \label{Hdivdivfem3ddof55} \\
(\boldsymbol \tau\boldsymbol n, \boldsymbol  n\times \boldsymbol x q)_{F_1} & \quad\forall~q\in\mathbb P_{\ell-2}(F_1),\label{Hdivdivfem3ddof6}
\end{align}
where $F_1\in\mathcal F(K)$ is an arbitrary but fixed face. The degrees of freedom~\eqref{Hdivdivfem3ddof6} will be regarded as interior degrees of freedom to the tetrahedron $K$, that is the degrees of freedom~\eqref{Hdivdivfem3ddof6} will be double-valued if $F\in\mathcal F_h^i$ is selected in different elements. 

Before we prove the unisolvence, we give characterization of the space of shape functions restricted to edges and faces, and derive some consequence of vanishing degree of freedoms. 

\begin{lemma}\label{lem:boundpolyl}
For any $\boldsymbol \tau\in\boldsymbol \Sigma_{\ell,k}(K)$, we have
\[
\boldsymbol  n_i^{\intercal}\boldsymbol\tau\boldsymbol  n_j|_e\in\mathbb P_{\ell}(e),\quad \boldsymbol  n^{\intercal}\boldsymbol\tau\boldsymbol  n|_F\in\mathbb P_{\ell}(F),\quad
2\div_F(\boldsymbol\tau \boldsymbol n)+ \partial_n(\boldsymbol n^{\intercal} \boldsymbol \tau\boldsymbol n)|_F\in\mathbb P_{\ell-1}(F)
\]
for each edge $e\in\mathcal E(K)$, each face $F\in\mathcal F(K)$ and $i,j=1,2$.
\end{lemma}
\begin{proof}
Take any $\boldsymbol\tau=\boldsymbol  x\boldsymbol  x^{\intercal}q\in\mathbb C_k^{\oplus}(K;\mathbb S)$ with $q\in\mathbb P_{k-2}(K)$.
Since $\boldsymbol  n_i^{\intercal}\boldsymbol  x$ is constant on each edge of $K$ and $\boldsymbol  n^{\intercal}\boldsymbol  x$ is constant on each face of $K$,
\[
\boldsymbol  n_i^{\intercal}\boldsymbol \tau\boldsymbol n_j|_{e}=(\boldsymbol  n_i^{\intercal}\boldsymbol  x)(\boldsymbol  n_j^{\intercal}\boldsymbol  x)q\in\mathbb P_{k-2}(e),\quad \boldsymbol  n^{\intercal}\boldsymbol \tau\boldsymbol  n|_F=(\boldsymbol  n^{\intercal}\boldsymbol  x)^2q\in\mathbb P_{k-2}(F),
\]
and
\begin{align*}
2\div_F(\boldsymbol\tau \boldsymbol n)+ \partial_n(\boldsymbol n^{\intercal} \boldsymbol \tau\boldsymbol n) &= 
 (\div_F(\boldsymbol\tau \boldsymbol n)+\boldsymbol n^{\intercal}\div\boldsymbol \tau)|_F \\&=
 \boldsymbol  n^{\intercal}\boldsymbol  x(\div_F(\boldsymbol x q)+\div(\boldsymbol  x q)+q)\in\mathbb P_{k-2}(F).
\end{align*}
Thus we conclude the results from the requirement $\ell\geq k-1$.
\end{proof}

\begin{lemma}\label{lem:vanishdof}
For any $\boldsymbol \tau\in\boldsymbol \Sigma_{\ell,k}(K)$ with the degrees of freedom~\eqref{Hdivdivfem3ddof1}-\eqref{Hdivdivfem3ddof55} vanishing, we have 
\begin{align}
\label{eq:20200707-1}
 \boldsymbol  n_i^{\intercal}\boldsymbol\tau\boldsymbol  n_j|_e&=0\quad\forall~e\in\mathcal E(K),\; i,j=1,2,\\
\label{eq:20200707-21}
\boldsymbol  n^{\intercal}\boldsymbol\tau\boldsymbol  n|_F &= 0\quad\forall~F\in\mathcal F(K),\\
\label{eq:20200707-22}
(2\div_F(\boldsymbol\tau \boldsymbol n)+ \partial_n(\boldsymbol n^{\intercal} \boldsymbol \tau\boldsymbol n))|_F & = 0\quad\forall~F\in\mathcal F(K),
\\
\notag
\div\div \bs\tau &= 0,\\
\label{eq:20200707-3}
(\boldsymbol \tau, \boldsymbol \varsigma)_K &=0 \quad\forall~\boldsymbol \varsigma\in\mathbb P_{\ell-1}(K;\mathbb S).
\end{align}
\end{lemma}
\begin{proof}
According to Lemma~\ref{lem:boundpolyl}, we acquire~\eqref{eq:20200707-1}-\eqref{eq:20200707-22} from the vanishing degrees of freedom~\eqref{Hdivdivfem3ddof1}-\eqref{Hdivdivfem3ddof4} directly. The scalar function $\boldsymbol  n^{\intercal}\boldsymbol\tau\boldsymbol  n|_F$ is the standard Lagrange element and the vanishing function value $\boldsymbol \tau(\delta)$ at vertices are used to ensure~\eqref{eq:20200707-21}. 

Noting that
$\div\div\bs\tau\in \mathbb P_{k-2}(K)$, 
we get from the Green's identity~\eqref{eq:greenidentitydivdiv}, ~\eqref{eq:20200707-1}-\eqref{eq:20200707-22} and the vanishing degrees of freedom~\eqref{Hdivdivfem3ddof5} that $\div\div\bs\tau=0$. Applying the Green's identity~\eqref{eq:greenidentitydivdiv} and ~\eqref{eq:20200707-1}-\eqref{eq:20200707-22}, it follows
\[
(\boldsymbol \tau, \nabla^2v)_K=0\quad\forall~v\in H^2(K),
\]
which together with~\eqref{Hdivdivfem3ddof55} and the decomposition~\eqref{eq:hesspolyspacedecomp2} yields~\eqref{eq:20200707-3}.
\end{proof}

With previous preparations, 
we prove the unisolvence as follows. 
For any $\bs\tau\in\boldsymbol \Sigma_{\ell,k}(K)$ satisfying $\div\div\boldsymbol \tau=0$, we have $\bs\tau\in\mathbb P_{\ell}(K;\mathbb S)$ as no contribution from $\mathbb C_k^{\oplus}(K;\mathbb S)$. By~\eqref{eq:20200707-3} the volume moments can only determine the polynomial of degree up to $\ell-1$. 

We then use the vanished trace. Similarly as the RT and BDM elements~\cite{BoffiBrezziFortin2013},  the vanishing normal-normal trace~\eqref{eq:20200707-21} implies the normal-normal part of $\bs\tau$ is zero. To determine the normal-tangential terms, further degrees of freedoms are needed. 

Unlike the traditional approach by transforming back to the reference element, we will chose an intrinsic coordinate. For ease of presentation, denote the four faces in $\mathcal F(K)$ by $F_i$, which is opposite to the $i$th vertex of $K$, and by $\boldsymbol n_i$ the outward unit normal vector of $F_i$ for $i=1,2,3,4$. 
Let $\boldsymbol t_i$ be the unit tangential vector of the edge from vertex $4$ to vertex $i$; see Fig.~\ref{fig:localcoor}. 
The set of three vectors $\{\boldsymbol t_1, \boldsymbol t_2, \boldsymbol t_3\}$ forms a basis of $\mathbb R^3$ although they may not be orthogonal in general. Consequently $\{\boldsymbol t_i\boldsymbol t_j^{\intercal}\}_{i,j=1}^3$ forms a basis of the second order tensor and $(\boldsymbol t_i, \boldsymbol n_i) \neq 0$ for $i=1,2,3$.
\begin{figure}[htbp]
\begin{center}
\includegraphics[width=6cm]{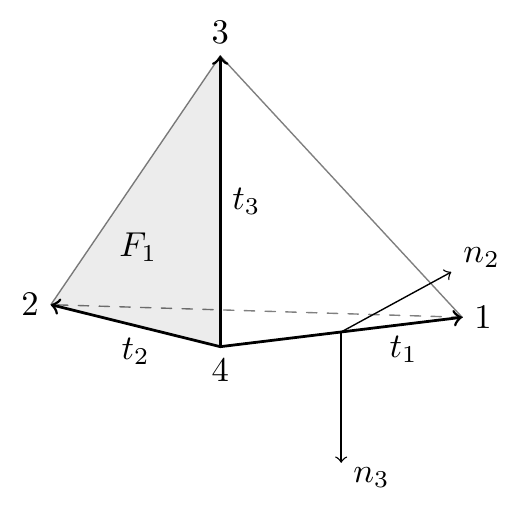}
\caption{Local coordinate formed by three edge vectors.}
\label{fig:localcoor}
\end{center}
\end{figure}
Let $\lambda_i(\boldsymbol x)$ be the $i$th barycentric coordinate with respect to the tetrahedron $K$ for $i=1,2,3,4$. Then $\lambda_i|_{F_i} = 0$ and $\nabla \lambda_i = - c_i \boldsymbol n_i$ for some $c_i>0$.

\begin{theorem}\label{lem:unisovlenHdivdivfem}
The degrees of freedom~\eqref{Hdivdivfem3ddof1}-\eqref{Hdivdivfem3ddof6} are unisolvent for $\boldsymbol \Sigma_{\ell,k}(K)$.
\end{theorem}
\begin{proof}
We first count the number of the degrees of freedom~\eqref{Hdivdivfem3ddof1}-\eqref{Hdivdivfem3ddof6}. Calculation of d.o.f.~\eqref{Hdivdivfem3ddof55} can be found in~\eqref{eq:dimsymx}.
The number of d.o.f.~\eqref{Hdivdivfem3ddof1}-\eqref{Hdivdivfem3ddof6} is
\begin{align*}
&24+18(\ell-1)+2[(\ell-1)(\ell-2)+(\ell+1)\ell] \\
&\quad+\frac{1}{6}(k^3-k)-4+\frac{1}{6}\ell(\ell-1)(5\ell+14) +\frac{1}{2}\ell(\ell-1) \\
=&\frac{1}{6}(5\ell^3+36\ell^2+67\ell+36)+\frac{1}{6}(k^3-k),
\end{align*}
which is same as $\dim\boldsymbol \Sigma_{\ell,k}(K)$ cf.~\eqref{eq:dimC}. 


Take any $\boldsymbol \tau\in\boldsymbol \Sigma_{\ell,k}(K)$ and
suppose all the degrees of freedom~\eqref{Hdivdivfem3ddof1}-\eqref{Hdivdivfem3ddof6} vanish. We are going to prove the function $\boldsymbol \tau = 0$. 
Using the local coordinate sketched in Fig.~\ref{fig:localcoor}, we can expand $\boldsymbol \tau$ as
\[
\bs\tau=\sum_{i,j=1}^3\tau_{ij}\boldsymbol t_i\boldsymbol t_j^{\intercal} \quad\textrm{with}\quad \tau_{ij}=\frac{\boldsymbol n_i^{\intercal}\bs\tau\boldsymbol n_j}{(\boldsymbol t_i^{\intercal}\boldsymbol n_i)(\boldsymbol t_j^{\intercal}\boldsymbol n_j)}.
\]
As $\boldsymbol \tau$ is symmetric, $\tau_{ij} = \tau_{ji}$. 
By~\eqref{eq:20200707-21}, it follows
\[
\tau_{ii}|_{F_i}=\boldsymbol n_i^{\intercal}\bs\tau\boldsymbol n_i|_{F_i}=0, \quad i = 1,2,3.
\]
Thus there exists $q_{\ell-1}\in \mathbb P_{\ell-1}(K)$ satisfying $\tau_{ii}=\lambda_iq_{\ell-1}$ for $i=1,2,3$.
Taking $\boldsymbol \varsigma = q_{\ell-1}\boldsymbol n_i\boldsymbol n_i^{\intercal}$ in~\eqref{eq:20200707-3} will produce 
\begin{equation}\label{eq:diagonal}
\tau_{ii}=0, \quad i = 1,2,3.
\end{equation}
Namely the diagonal of $\boldsymbol \tau$ is zero.
So far, in the chosen coordinate, $\boldsymbol n_4^{\intercal}\bs\tau\boldsymbol n_4 = 0$ has no simple formulation and will be used later on.

On the other hand, from~\eqref{eq:20200707-1} we have $\Pi_{F_1}(\bs\tau\boldsymbol n_1)\in H_0(\div_{F_1}, F_1) $.
As $\boldsymbol n_1^{\intercal}\boldsymbol \tau \boldsymbol n_1 = \tau_{11} = 0$ in $K$ cf.~\eqref{eq:diagonal}, it follows $\partial_{n_1}(\boldsymbol n_1^{\intercal}\boldsymbol \tau \boldsymbol n_1)|_{F_1}=0$. Therefore~\eqref{eq:20200707-22} becomes
\[
2\div_{F_1}(\boldsymbol\tau \boldsymbol n_1)|_{F_1}=0.
\]
Hence there exists $q_{\ell-2}\in \mathbb P_{\ell-2}(F_1)$ such that $\boldsymbol n_1\times(\boldsymbol\tau \boldsymbol n_1)=\nabla_{F_1}(b_{F_1}q_{\ell-2})$, 
where $b_{F_1}$ is the cubic bubble function on face $F_1$.
Together with~\eqref{Hdivdivfem3ddof6} and the fact $\div_{F_1}(\boldsymbol x\mathbb P_{\ell-2}(F_1))=\mathbb P_{\ell-2}(F_1)$, 
we get $(\boldsymbol n_1\times(\boldsymbol\tau \boldsymbol n_1))|_{F_1}=\boldsymbol 0$.
Thus $(\boldsymbol\tau \boldsymbol n_1)|_{F_1}=\boldsymbol 0$.
Then there exists $\boldsymbol q_{\ell-1}\in \mathbb P_{\ell-1}(K;\mathbb R^3)$ such that $\boldsymbol\tau \boldsymbol n_1=\lambda_1\boldsymbol q_{\ell-1}$,  combined with~\eqref{eq:20200707-3} yields $\boldsymbol\tau \boldsymbol n_1=\boldsymbol 0$. That is the first row of $\boldsymbol \tau$ is zero, i.e. $\tau_{11}=\tau_{12}=\tau_{13}=0$. 

By the symmetry, now
$
\bs\tau=2\tau_{23}\sym(\boldsymbol t_2\boldsymbol t_3^{\intercal})
$. Multiplying $\bs\tau$ by $\boldsymbol n_4$ from both sides and restricting to $F_4$, we have
\[
\tau_{23}|_{F_4}=\frac{1}{2}\frac{\boldsymbol n_4^{\intercal}\bs\tau\boldsymbol n_4}{(\boldsymbol t_2^{\intercal}\boldsymbol n_4)(\boldsymbol t_3^{\intercal}\boldsymbol n_4)}|_{F_4}=0.
\]
The denominator is non-zero as $\boldsymbol t_2, \boldsymbol t_3$ are non tangential vectors of face $F_4$. 
Again there exists $q_{\ell-1}\in \mathbb P_{\ell-1}(K)$ satisfying $\tau_{23}=\lambda_4q_{\ell-1}$.
Taking $\boldsymbol \varsigma=\sym(\boldsymbol t_2\boldsymbol t_3^{\intercal})q_{\ell-1}$ in~\eqref{eq:20200707-3} gives $\tau_{23}=0$. We thus have proved $\boldsymbol \tau = 0$ and consequently the unisolvence. 
\end{proof}

Due to~\eqref{Hdivdivfem3ddof4}, it is arduous to figure out the explicit basis functions of  $\boldsymbol \Sigma_{\ell,k}(K)$, which are dual to the degrees of freedom~\eqref{Hdivdivfem3ddof1}-\eqref{Hdivdivfem3ddof6}.
Alternatively we can hybridize the degrees of freedom~\eqref{Hdivdivfem3ddof4}, and use the basis functions of the standard Lagrange element~\cite{ChenHuang2020}. 

\subsection{Polynomial bubble function spaces and the bubble complex}\label{sec:bubble}
Let
\begin{align*}
\mathring{\boldsymbol \Sigma}_{\ell,k}(K)&:=\{\boldsymbol  \tau\in\boldsymbol \Sigma_{\ell,k}(K): \textrm{all degrees of freedom~\eqref{Hdivdivfem3ddof1}-\eqref{Hdivdivfem3ddof4} vanish}\}.
\end{align*}
Together with vanishing~\eqref{Hdivdivfem3ddof5}, we can  conclude that $\div\div \boldsymbol \tau = 0$. In view of Fig.~\ref{fig:femdec} and Lemma~\ref{lem:vanishdof}, the last two set of d.o.f.~\eqref{Hdivdivfem3ddof55}-\eqref{Hdivdivfem3ddof6} can be replaced by 
\begin{equation*}
(\boldsymbol \tau, \boldsymbol \varsigma)_K \quad \forall~\boldsymbol \varsigma\in \mathring{\boldsymbol \Sigma}_{\ell,k}(K) \cap \ker(\div\div),
\end{equation*}
Next we give characterization of $\mathring{\boldsymbol \Sigma}_{\ell,k}(K) \cap \ker(\div\div)$.

By the exactness of divdiv complex, if $\div\div \boldsymbol \tau = 0$ and $\tr(\boldsymbol \tau) = 0$, it is possible that $\boldsymbol \tau = \sym \curl \boldsymbol \sigma$ for some $\boldsymbol \sigma \in \boldsymbol B_{\ell+1}(\sym\curl, K;\mathbb T) : = \boldsymbol H_0(\sym \curl,K; \mathbb T)\cap \mathbb P_{\ell +1}(K;\mathbb T)$. 
We will give an explicit characterization of $\boldsymbol B_{\ell+1}(\sym\curl, K;\mathbb T)$, show $\mathring{\boldsymbol \Sigma}_{\ell,k}(K) \cap \ker(\div\div)= \sym \curl \boldsymbol B_{\ell+1}(\sym\curl, K;\mathbb T)$, and consequently get a set of computable and symmetric d.o.f..

We begin with a characterization of the trace of functions in $\boldsymbol H(\sym \curl,K; \mathbb T)$.
\begin{lemma} [Green's identity]\label{lm:Green}
Let $K$ be a polyhedron, and let $\boldsymbol  \tau\in \boldsymbol H^1(K; \mathbb M)$ and $\bs\sigma \in \boldsymbol H^1(K; \mathbb S)$. Then we have
\begin{align}
(\sym\curl\boldsymbol \tau, \boldsymbol \sigma)_K=(\boldsymbol \tau, \curl\boldsymbol \sigma)_K & - \sum_{F\in\mathcal F(K)}(\sym \Pi_F (\boldsymbol \tau\times\boldsymbol n)\Pi_F, \Pi_F \boldsymbol\sigma \Pi_F)_F \notag\\
& - \sum_{F\in\mathcal F(K)}(\boldsymbol n\cdot \boldsymbol \tau\times\boldsymbol n, \boldsymbol  n\cdot \boldsymbol\sigma \Pi_F)_F. \notag
\end{align}
\end{lemma}
\begin{proof}
As $\boldsymbol \sigma$ is symmetric,
\begin{align*}
(\sym\curl\boldsymbol \tau, \boldsymbol \sigma)_K&= (\curl\boldsymbol \tau, \boldsymbol \sigma)_K = (\boldsymbol \tau, \curl \boldsymbol \sigma)_K - (\boldsymbol \tau \times \boldsymbol n, \boldsymbol \sigma)_{\partial K}.
\end{align*}
On each face, we expand the boundary term
\begin{align*}
(\boldsymbol \tau\times \boldsymbol n, \boldsymbol \sigma)_{F} = (\Pi_F (\boldsymbol \tau\times \boldsymbol n)\Pi_F, \Pi_F \boldsymbol\sigma \Pi_F)_F + (\boldsymbol n\cdot \boldsymbol \tau\times\boldsymbol n , \boldsymbol  n\cdot \boldsymbol\sigma \Pi_F)_F.
\end{align*}
Then we use the fact $\Pi_F \boldsymbol\sigma \Pi_F$ is symmetric to arrive the desired identity. 
\end{proof}

Based on the Green's identity, we introduce the following trace operators for $\boldsymbol H(\sym\curl)$ space
\begin{enumerate}
\item $\tr_1(\boldsymbol \tau) := \Pi_F \sym (\boldsymbol\tau\times\boldsymbol n)\Pi_F$,

 \item $\tr_1^{\bot}(\boldsymbol \tau) := \boldsymbol n \times \sym (\boldsymbol\tau\times\boldsymbol n)\times \boldsymbol n$,

 \item $\tr_2(\boldsymbol \tau) := \boldsymbol n\cdot \boldsymbol \tau\times\boldsymbol n$.
\end{enumerate}
Both $\tr_1(\boldsymbol \tau)$ and $\tr_1^{\bot}(\boldsymbol \tau)$ are symmetric tensors on each face and $\tr_2(\boldsymbol \tau)$ is a vector function. Obviously $\tr_1(\boldsymbol \tau) = 0$ if and only if  $\tr_1^{\bot}(\boldsymbol \tau) = 0$  as   $\tr_1^{\bot}(\boldsymbol \tau)$ is just a rotation of $\tr_1(\boldsymbol \tau)$.
Using the trace operators, $\boldsymbol H(\sym\curl)$ polynomial bubble function space can be defined as
\begin{align*}
\boldsymbol B_{\ell+1}(\sym\curl, K;\mathbb T):=\{&  
\bs\tau\in\mathbb P_{\ell+1}(K;\mathbb T):(\boldsymbol n\cdot \boldsymbol \tau\times\boldsymbol n)|_F=\bs0, \\
&(\boldsymbol n\times \sym(\boldsymbol\tau\times\boldsymbol n)\times\boldsymbol n)|_F=\bs0\quad\forall~F\in\mathcal F(K)
\}.
\end{align*}

We shall give an explicit characterization of $\boldsymbol B_{\ell+1}(\sym\curl, K;\mathbb T)$. 
\begin{lemma}\label{lem:symcurlbubbleedgevanish}
Let $\bs\tau\in \boldsymbol B_{\ell+1}(\sym\curl, K;\mathbb T)$. It holds
\begin{equation}\label{eq:edge}
\bs\tau|_e=\boldsymbol 0\quad\forall~e\in\mathcal E(K).
\end{equation}
\end{lemma}
\begin{proof}
It is straightforward to verify~\eqref{eq:edge} on the reference tetrahedron for which $\boldsymbol e = (1,0,0)$ and two normal vectors of the face containing $\boldsymbol e$ is $\boldsymbol n_1= (1,0,0)$ and $\boldsymbol n_2=(0,0,1)$. To avoid complicated transformation of trace operators, we provide a proof using an intrinsic basis of $\mathbb T$ on $K$. 

Take any edge $e\in\mathcal E(K)$ with the tangential vector $\boldsymbol t$. Let $\boldsymbol n_1$ and $\boldsymbol n_2$ be the unit outward normal vectors of two faces sharing edge $e$. Set $\boldsymbol s_i:=\boldsymbol t\times\boldsymbol n_i$ for $i=1,2$.
By direction computation, we get on edge $e$ for $i=1,2$ that
$$
\boldsymbol n_i^{\intercal}\bs\tau\boldsymbol t=(\boldsymbol n_i\cdot\bs\tau\times\boldsymbol n_i)\cdot\boldsymbol s_i=0,
$$
$$
\boldsymbol n_i^{\intercal}\bs\tau\boldsymbol s_i=-(\boldsymbol n_i\cdot\bs\tau\times\boldsymbol n_i)\cdot\boldsymbol t=0,
$$
$$
\boldsymbol t^{\intercal}\bs\tau\boldsymbol t - \boldsymbol s_i^{\intercal}\bs\tau\boldsymbol s_i=2\boldsymbol t\cdot \sym(\boldsymbol\tau\times\boldsymbol n_i)\cdot\boldsymbol s_i=2\boldsymbol s_i\cdot(\boldsymbol n_i\times \sym(\boldsymbol\tau\times\boldsymbol n_i)\times\boldsymbol n_i)\cdot\boldsymbol t=0,
$$
$$
\boldsymbol t^{\intercal}\bs\tau\boldsymbol s_i=-\boldsymbol t\cdot \sym(\boldsymbol\tau\times\boldsymbol n_i)\cdot\boldsymbol t=\boldsymbol s_i\cdot(\boldsymbol n_i\times \sym(\boldsymbol\tau\times\boldsymbol n_i)\times\boldsymbol n_i)\cdot\boldsymbol s_i=0.
$$
Both $\textrm{span}\{\boldsymbol s_1, \boldsymbol s_2\}$ and $\textrm{span}\{\boldsymbol n_1, \boldsymbol n_2\}$ form the same normal vector space of edge $e$, then the last identity implies
$$
\boldsymbol t^{\intercal}\bs\tau\boldsymbol n_i=0.
$$
Then it is sufficient to prove the eight trace-free tensors 
\begin{equation}\label{eq:nst}
\boldsymbol n_1\boldsymbol t^{\intercal},\; \boldsymbol n_2\boldsymbol t^{\intercal},\; \boldsymbol n_1\boldsymbol s_1^{\intercal},\; \boldsymbol n_2\boldsymbol s_2^{\intercal},\; \boldsymbol t\,\boldsymbol n_1^{\intercal},\; \boldsymbol t\,\boldsymbol n_2^{\intercal},\; \boldsymbol t\,\boldsymbol t^{\intercal}-\boldsymbol s_1\boldsymbol s_1^{\intercal},\; \boldsymbol t\,\boldsymbol t^{\intercal}-\boldsymbol s_2\boldsymbol s_2^{\intercal}
\end{equation}
are linear independent. Assume there exist $c_i\in\mathbb R$, $i=1,\cdots,8$ such that 
\begin{align*}
c_1\boldsymbol n_1\boldsymbol t^{\intercal} + c_2\boldsymbol n_2\boldsymbol t^{\intercal} + c_3\boldsymbol n_1\boldsymbol s_1^{\intercal} + c_4\boldsymbol n_2\boldsymbol s_2^{\intercal} + c_5\boldsymbol t\,\boldsymbol n_1^{\intercal} + c_6\boldsymbol t\,\boldsymbol n_2^{\intercal}&\\+c_7(\boldsymbol t\,\boldsymbol t^{\intercal}-\boldsymbol s_1\boldsymbol s_1^{\intercal}) + c_8(\boldsymbol t\,\boldsymbol t^{\intercal}-\boldsymbol s_2\boldsymbol s_2^{\intercal})&=\bs0.
\end{align*}
Multiplying the last equation by $\boldsymbol t$ from the right and left respectively, we obtain
$$
c_1\boldsymbol n_1 + c_2\boldsymbol n_2 + (c_7+c_8)\boldsymbol t=\bs0, \quad c_5\boldsymbol n_1^{\intercal} + c_6\boldsymbol n_2^{\intercal} + (c_7+c_8)\boldsymbol t^{\intercal}=\bs0.
$$
Hence $c_1=c_2=c_5=c_6=c_7+c_8=0$, which yields
$$
c_3\boldsymbol n_1\boldsymbol s_1^{\intercal} + c_4\boldsymbol n_2\boldsymbol s_2^{\intercal} + c_7(\boldsymbol s_2\boldsymbol s_2^{\intercal}-\boldsymbol s_1\boldsymbol s_1^{\intercal})=\bs0.
$$
Multiplying the last equation by $\boldsymbol n_1$ from the right, it follows
$$
(\boldsymbol s_2\cdot\boldsymbol n_1)(c_4\boldsymbol n_2 + c_7\boldsymbol s_2)=\bs0.
$$
As a result $c_4=c_7=0$, and then $c_3=0$.
\end{proof}

We write $\mathbb P_{\ell+1}(K;\mathbb T)$ as $\mathbb P_{\ell+1}(K)\otimes \mathbb T$ and use the barycentric coordinate representation of a polynomial. That is a polynomial $p\in \mathbb P_{\ell+1}(K)$ has a unique representation in terms of
\begin{equation}\label{eq:bary}
 p = \lambda_{1}^{\alpha_1}\lambda_{2}^{\alpha_2}\lambda_{3}^{\alpha_3}\lambda_{4}^{\alpha_4}, \quad \sum_{i=1}^4\alpha_i = \ell+1, \alpha_i \in \mathbb N.
\end{equation}
Lemma~\ref{lem:symcurlbubbleedgevanish} implies that $p$ must contain a face bubble $b_F = \lambda_{i}\lambda_{j}\lambda_{k}$ where $(i,j,k)$ are three vertices of $F$. Otherwise, if $p = \lambda_i^{\alpha_i}\lambda_j^{\alpha_j}, \alpha_i + \alpha_j = \ell+1$, then $p$ is not zero on the edge $(i,j)$. 

We consider the subspace $b_F \mathbb P_{\ell-2}(K)\otimes \mathbb T$ and identify its intersection with $\ker(\tr)$. Due to the face bubble $b_F$, the polynomial is zero on the other faces. So we only need to consider the trace on face $F$. Without loss of generality, we can chose the coordinate s.t. $n_F = (0,0,1)$. Chose the canonical basis of $\mathbb T$ associated to this coordinate. Then by direct calculation to find out $\ker(\tr)\cap \mathbb T$ consists of
$$
\begin{pmatrix}
0 & 0 & 1 \\
0 & 0 & 0 \\
0 & 0 & 0
\end{pmatrix},
\begin{pmatrix}
0 & 0 & 0 \\
0 & 0 & 1 \\
0 & 0 & 0
\end{pmatrix},
\text{ and }
\begin{pmatrix}
1 & 0 & 0 \\
0 & 1 & 0 \\
0 & 0 & -2
\end{pmatrix}.
$$
Switch to an intrinsic basis, we obtain the following explicit characterization of $\boldsymbol B_{\ell + 1}(\sym\curl, K;\mathbb T)$.
\begin{lemma}
For each face $F$, we chose two unit tangent vectors $\boldsymbol t_1, \boldsymbol t_2$ s.t. $(\boldsymbol t_1, \boldsymbol t_2, \boldsymbol n_F)$ forms an orthonormal basis of $\mathbb R^3$. Then
\begin{equation}\label{eq:span1}
\boldsymbol B_{\ell+1}(\sym\curl, K;\mathbb T)  = {\rm span}\{ p b_F \psi_i^F, p\in \mathbb P_{\ell-2}(K), F\in \mathcal F(K), i = 1,2,3\},
\end{equation}
where the three bubble functions are:
$$
\psi_1^F = \boldsymbol t_1 \boldsymbol n_F^{\intercal}, \quad \psi_2^F = \boldsymbol t_2 \boldsymbol n_F^{\intercal}, \quad  \psi_3^F = \boldsymbol t_1  \boldsymbol t_1^{\intercal} +  \boldsymbol t_2  \boldsymbol t_2^{\intercal} - 2 \boldsymbol n_F  \boldsymbol n_F^{\intercal}.
$$
\end{lemma}
\begin{proof}
Using the formulae~\eqref{eq:xuv}-\eqref{eq:xtimesuv}, by the direct calculation, we can easily show $\psi_i^F\in \ker(\tr)\cap \mathbb T$ for each face $F$ and $i=1,2,3$. As $\dim \ker(\tr)\cap \mathbb T = 3$, we conclude that
$$ \ker (\tr) \cap (b_F \mathbb P_{\ell-2}(K)\otimes \mathbb T) = {\rm span} \{ p b_F \psi_i^F, p\in \mathbb P_{\ell-2}(K), i = 1,2,3\}.$$
By Lemma~\ref{lem:symcurlbubbleedgevanish}  we know that 
$$
\ker(\tr)\cap (\mathbb P_{\ell+1}\otimes \mathbb T) =  \cup_F \ker(\tr)\cap (b_F \mathbb P_{\ell-2}(K)\otimes \mathbb T) 
$$
and thus~\eqref{eq:span1} follows. 
\end{proof}

We only give a generating set of the bubble function space as the $12$ constant matrices $\{ \psi_1^F, \psi_2^F, \psi_3^F, F\in \mathcal F(K)\}, $ are not linear independent. Next we find out a basis from this generating set. 
\begin{lemma}
Let $(i,j,k)$ be three vertices of face $F$ and $\mathbb P_{\ell-2}(F) = \{ \lambda_{i}^{\alpha_1}\lambda_{j}^{\alpha_2}\lambda_{k}^{\alpha_3}, \alpha_1+\alpha_2+\alpha_3 = \ell-2, \alpha_i\in \mathbb N, i=1,2,3\}$.
 Define $\boldsymbol B_{F,\ell + 1}: = b_F\mathbb P_{\ell-2}(F)\otimes {\rm span}\{ \psi_1^F, \psi_2^F, \psi_3^F\}$ and $\boldsymbol B_{K, \ell + 1} = b_K\mathbb P_{\ell-3}(K)\otimes {\rm span}\{ \psi_1^F, \psi_2^F, F\in \mathcal F(K)\}$. Then
\begin{equation}\label{eq:basisofB}
\boldsymbol B_{\ell+1}(\sym\curl, K;\mathbb T)  = \oplus_{F\in \mathcal F(K)} \boldsymbol B_{F,\ell + 1} \oplus \boldsymbol B_{K,\ell + 1},
\end{equation} 
and consequently
$$\dim \boldsymbol B_{\ell+1}(\sym\curl, K;\mathbb T)=\frac{2}{3}\ell(\ell-1)(2\ell+5)=\frac{1}{3}(4\ell^3+6\ell^2-10\ell).$$ 
\end{lemma}
\begin{proof}
The $12$ constant matrices $\{ \psi_1^F, \psi_2^F, \psi_3^F, F\in \mathcal F(K)\}$ are not linear independent as $\dim \mathbb T = 8$. Among them, $\{ \psi_1^F, \psi_2^F, F\in \mathcal F(K)\}$ forms a basis of $\mathbb T$ which can be proved as verifying the linear independence of~\eqref{eq:nst} in Lemma~\ref{lem:symcurlbubbleedgevanish} or see~\cite{HuLiang2020}. 

For each $p b_F$, with $p\in \mathbb P_{\ell - 2}(K)$, we can group into either $b_K\mathbb P_{\ell-3}(K)$ or $b_F\mathbb P_{\ell-2}(F)$  depending on if the polynomial $p|_F$ is zero or not, respectively. That is, for one fixed face $F$:
$$
b_F \mathbb P_{\ell - 2}(K) = b_F\mathbb P_{\ell-2}(F) \oplus b_K\mathbb P_{\ell-3}(K). 
$$
The sum is direct in view of the barycentric representation~\eqref{eq:bary} of a polynomial.
Then coupled with $\{\psi_i^F\}$, we get the basis~\eqref{eq:basisofB} of the bubble function space.

The dimension of $\boldsymbol B_{\ell+1}(\sym\curl, K;\mathbb T) $ is 
\begin{align*}
4\cdot 3 \cdot \dim\mathbb P_{\ell-2}(F) + 8 \dim\mathbb P_{\ell-3}(K) = \frac{1}{3}(4\ell^3+6\ell^2-10\ell),
\end{align*}
as required.
\end{proof}



We then verify $\sym \curl \boldsymbol B_{\ell+1}(\sym\curl, K;\mathbb T) \subset \mathring{\boldsymbol \Sigma}_{\ell,k}(K)$ by verifying all boundary d.o.f. are vanished. 
\begin{lemma}\label{lem:symcurledgeprop}
Let $\bs\tau \in \boldsymbol B_{\ell+1}(\sym\curl, K;\mathbb T)$. Assume edge $e\in\mathcal E(K)$ is shared by faces $F_i$ and $F_j$. It holds $\boldsymbol  n_i^{\intercal}(\sym\curl\boldsymbol \tau )\boldsymbol n_j|_e=0$.
\end{lemma}
\begin{proof}
For the ease of notation, let $\boldsymbol \sigma=\sym\curl\boldsymbol \tau$. Suppose $\bs\tau=\sum\limits_{F\in\mathcal F(K)}\sum\limits_{l=1}^3q_{F,l} b_F \psi_l^F$ with $q_{F,l}\in \mathbb P_{\ell-2}(K)$.
By $b_F|_e=0$, we get
$$
\boldsymbol n_i^{\intercal}\boldsymbol \sigma\boldsymbol n_j|_e=\sum_{F\in\mathcal F(K)}\sum_{l=1}^3q_{F,l}|_e(\boldsymbol n_i^{\intercal} \sym\curl(b_F \psi_l^F)\boldsymbol n_j)|_e.
$$
Since $\lambda_i|_e=\lambda_j|_e=0$, we can see that $(\boldsymbol n_i\times\boldsymbol n_F\cdot\nabla b_F)|_e=(\boldsymbol n_j\times\boldsymbol n_F\cdot\nabla b_F)|_e=0$. Thus for $l=1,2$,
\begin{align*}
&\quad\;(\boldsymbol n_i^{\intercal} \sym\curl(b_F \psi_l^F)\boldsymbol n_j)|_e\\
&=-(\boldsymbol n_i\cdot(b_F \boldsymbol t_l\boldsymbol n_F)\times\nabla\cdot\boldsymbol n_j)|_e - (\boldsymbol n_j\cdot(b_F\boldsymbol t_l\boldsymbol n_F)\times\nabla\cdot\boldsymbol n_i)|_e \\
&=\boldsymbol n_i\cdot\boldsymbol t_l(\boldsymbol n_F\times\boldsymbol n_j\cdot\nabla b_F)|_e + \boldsymbol n_j\cdot\boldsymbol t_l(\boldsymbol n_F\times\boldsymbol n_i\cdot\nabla b_F)|_e\\
&=0.
\end{align*}

Next consider $l=3$. When $F\neq F_j$, the face bubble $b_F$ has a factor $\lambda_j$, which implies $(\boldsymbol n_j\times\nabla b_F)|_e=\bs0$. Thus
$$
(\boldsymbol n_i^{\intercal}\curl(b_F \psi_3^F)\boldsymbol n_j)|_e =-(\boldsymbol n_i\cdot(b_F \psi_3^F)\times\nabla\cdot\boldsymbol n_j)|_e=(\boldsymbol n_i\cdot\psi_3^F\cdot(\boldsymbol n_j\times\nabla b_F))|_e=0.
$$
When $F= F_j$, the face bubble $b_F$ has a factor $\lambda_i$. 
By the fact that $(\boldsymbol t_1, \boldsymbol t_2, \boldsymbol n_j)$ forms an orthonormal basis of $\mathbb R^3$,
\begin{align*}
\boldsymbol n_i\cdot\boldsymbol t_2(\boldsymbol t_2\times\boldsymbol n_j\cdot\nabla\lambda_i)&=\boldsymbol n_i\cdot(\boldsymbol n_j\times\boldsymbol t_1)(\boldsymbol t_1\cdot\nabla\lambda_i)=-(\boldsymbol t_1\cdot\nabla\lambda_i)(\boldsymbol n_j\times\boldsymbol n_i\cdot\boldsymbol t_1) \\
&=-\boldsymbol n_i\cdot\boldsymbol t_1(\boldsymbol n_j\times\nabla\lambda_i\cdot\boldsymbol t_1),
\end{align*}
which implies
$$
\boldsymbol n_i\cdot\boldsymbol t_1(\boldsymbol n_j\times\nabla\lambda_i\cdot\boldsymbol t_1) + \boldsymbol n_i\cdot\boldsymbol t_2(\boldsymbol n_j\times\nabla\lambda_i\cdot\boldsymbol t_2)=0.
$$
As a result,
$$
(\boldsymbol n_i^{\intercal}\curl(b_F \psi_3^F)\boldsymbol n_j)|_e = \boldsymbol n_i\cdot\boldsymbol t_1(\boldsymbol n_j\times\nabla b_F\cdot\boldsymbol t_1)|_e + \boldsymbol n_i\cdot\boldsymbol t_2(\boldsymbol n_j\times\nabla b_F\cdot\boldsymbol t_2)|_e=0.
$$
Similarly $(\boldsymbol n_j^{\intercal}\curl(b_F \psi_3^F)\boldsymbol n_i)|_e=0$ holds. Hence $(\boldsymbol n_i^{\intercal}\sym\curl(b_F \psi_3^F)\boldsymbol n_j)|_e=0$.

Therefore $\boldsymbol  n_i^{\intercal}\boldsymbol \sigma\boldsymbol n_j|_e=0$.
\end{proof}

Next we show the two traces $\tr_2(\bs \tau)$ is in $H(\div_F)$ and $\tr_1(\bs \tau)$ in $H(\div_F\div_F)$. 
\begin{lemma}\label{lm:tracerelation}
When $\boldsymbol \sigma = \sym\curl \, \boldsymbol\tau$ with $\boldsymbol\tau\in\boldsymbol H^2(K;\mathbb M)$, we can express the trace in terms of the differential operators on surface $F$ of $K$
\begin{align}
\label{eq:trace1} \boldsymbol n^{\intercal}\boldsymbol \sigma\boldsymbol n &=\div_F(\boldsymbol n\cdot \boldsymbol \tau\times\boldsymbol n),\\
\label{eq:trace2} \nabla_F^{\bot}\cdot(\boldsymbol n\times\boldsymbol \sigma\cdot\boldsymbol n)+\boldsymbol  n^{\intercal} \div\boldsymbol \sigma  &=-\mathrm{rot}_F\mathrm{rot}_F(\boldsymbol n\times \sym(\boldsymbol\tau\times\boldsymbol n)\times\boldsymbol n)\\
& = \div_F\div_F(\Pi_F \sym(\boldsymbol\tau\times\boldsymbol n) \Pi_F). \notag
\end{align}
\end{lemma}
\begin{proof}
By
\begin{align*}
\boldsymbol n^{\intercal}\boldsymbol\sigma\boldsymbol n&=\frac{1}{2}\boldsymbol n\cdot(\nabla\times(\boldsymbol\tau^{\intercal})-\boldsymbol\tau\times\nabla)\cdot\boldsymbol n=\frac{1}{2}\nabla_F^{\bot}\cdot(\boldsymbol\tau^{\intercal})\cdot\boldsymbol n + \frac{1}{2}\boldsymbol n\cdot\boldsymbol\tau\cdot\nabla_F^{\bot}
\end{align*}
and the fact $\nabla_F^{\bot}\cdot(\boldsymbol\tau^{\intercal})\cdot\boldsymbol n=\boldsymbol n\cdot\boldsymbol\tau\cdot\nabla_F^{\bot}$, we get
$$
\boldsymbol n^{\intercal}\boldsymbol\sigma\boldsymbol n=\boldsymbol n\cdot\boldsymbol\tau\cdot\nabla_F^{\bot}=\mathrm{rot}_F (\boldsymbol n\cdot\boldsymbol \tau \Pi_F).
$$
Then the identity~\eqref{eq:trace1} holds from~\eqref{eq:rotFdivF}.

Next we prove~\eqref{eq:trace2}. Employing~\eqref{eq:tangentialtrace} with $\boldsymbol v=\boldsymbol\tau^{\intercal}\cdot\boldsymbol n$,
\begin{align*}
\nabla_F^{\bot}\cdot(\boldsymbol n\times\boldsymbol \sigma\cdot\boldsymbol n) &=\frac{1}{2}\nabla_F^{\bot}\cdot\big(\boldsymbol n\times(\nabla\times(\boldsymbol\tau^{\intercal})-\boldsymbol\tau\times\nabla)\cdot\boldsymbol n\big) \\
 &=\frac{1}{2}\nabla_F^{\bot}\cdot\big(\boldsymbol n\times(\nabla\times(\boldsymbol\tau^{\intercal}\cdot\boldsymbol n))\big) + \frac{1}{2}\nabla_F^{\bot}\cdot(\boldsymbol n\times\boldsymbol\tau)\cdot\nabla_F^{\bot} \\
 &=\frac{1}{2}\nabla_F^{\bot}\cdot\big(\nabla(\boldsymbol n\cdot\boldsymbol\tau^{\intercal}\cdot\boldsymbol n) - \partial_n(\boldsymbol\tau^{\intercal}\cdot\boldsymbol n)\big) + \frac{1}{2}\nabla_F^{\bot}\cdot(\boldsymbol n\times\boldsymbol\tau)\cdot\nabla_F^{\bot} \\
 &=-\frac{1}{2}\nabla_F^{\bot}\cdot\big(\partial_n(\boldsymbol\tau^{\intercal}\cdot\boldsymbol n)\big) + \frac{1}{2}\nabla_F^{\bot}\cdot(\boldsymbol n\times\boldsymbol\tau)\cdot\nabla_F^{\bot}.
\end{align*}
On the other side, we have
\begin{align*}
\boldsymbol n\cdot \div\boldsymbol \sigma &=\boldsymbol  n\cdot\boldsymbol\sigma\cdot\nabla=\frac{1}{2}\boldsymbol n\cdot(\nabla\times(\boldsymbol\tau^{\intercal}))\cdot\nabla=\frac{1}{2}\nabla_F^{\bot}\cdot(\boldsymbol\tau^{\intercal})\cdot\nabla \\
&=\frac{1}{2}\nabla_F^{\bot}\cdot(\boldsymbol\tau^{\intercal})\cdot(\boldsymbol n\partial_n+\nabla_F)=\frac{1}{2}\nabla_F^{\bot}\cdot\big(\partial_n(\boldsymbol\tau^{\intercal}\cdot\boldsymbol n)\big) + \frac{1}{2}\nabla_F^{\bot}\cdot(\boldsymbol\tau^{\intercal})\cdot\nabla_F \\
&=\frac{1}{2}\nabla_F^{\bot}\cdot\big(\partial_n(\boldsymbol\tau^{\intercal}\cdot\boldsymbol n)\big) - \frac{1}{2}\nabla_F^{\bot}\cdot(\boldsymbol\tau^{\intercal}\times\boldsymbol n)\cdot\nabla_F^{\bot}.
\end{align*}
The sum of the last two identities gives
\begin{align*}
\nabla_F^{\bot}\cdot(\boldsymbol n\times\boldsymbol \sigma\cdot\boldsymbol n)+\boldsymbol  n\cdot \div\boldsymbol \sigma &=\nabla_F^{\bot}\cdot\sym(\boldsymbol n\times\boldsymbol\tau\Pi_F)\cdot\nabla_F^{\bot}.
\end{align*}
Therefore~\eqref{eq:trace2} follows from $\sym(\boldsymbol n\times\boldsymbol\tau\Pi_F)=-\boldsymbol  n\times \sym(\boldsymbol\tau\times\boldsymbol n)\times\boldsymbol n$.
\end{proof}

Note that $ \nabla_F^{\bot}\cdot(\boldsymbol n\times\boldsymbol \sigma\cdot\boldsymbol n)+\boldsymbol  n^{\intercal} \div\boldsymbol \sigma$ is an equivalent formulation of the second trace of $\boldsymbol \sigma$. 
%
Combining Lemmas~\ref{lem:symcurledgeprop} and~\ref{lm:tracerelation} gives the following result.
\begin{lemma}\label{lm:symbubble}
It holds
\begin{equation}\label{eq:symcurlindivdivbubble}
\sym\curl \boldsymbol B_{\ell+1}(\sym\curl, K;\mathbb T) \subseteq (\mathring{\boldsymbol \Sigma}_{\ell,k}(K)\cap\ker(\div\div)).
\end{equation}
\end{lemma}
\begin{proof}
For $\boldsymbol \tau \in \boldsymbol B_{\ell+1}(\sym\curl, K;\mathbb T)$, by construction, $\boldsymbol n\cdot \boldsymbol \tau\times\boldsymbol n = 0$ and $\boldsymbol n\times \sym(\boldsymbol\tau\times\boldsymbol n)\times\boldsymbol n = 0$ on $\partial K$. Let $\boldsymbol \sigma = \sym \curl \boldsymbol \tau$. Then by Lemma~\ref{lm:tracerelation}, d.o.f.~\eqref{Hdivdivfem3ddof3}-\eqref{Hdivdivfem3ddof4} vanish. By Lemma~\ref{lem:symcurledgeprop},~\eqref{Hdivdivfem3ddof2} vanish. As $\boldsymbol \tau$ contains a face bubble, $\boldsymbol \sigma$  will have an  edge bubble function which means $\boldsymbol \sigma(\delta) = 0$ for all $\delta \in \mathcal V(K)$. Therefore  $\sym\curl \boldsymbol B_{\ell+1}(\sym\curl, K;\mathbb T) \subseteq \mathring{\boldsymbol \Sigma}_{\ell,k}(K)$. The property $\div\div (\sym \curl \boldsymbol \tau) = 0$ is from the divdiv complex.
\end{proof}

Indeed the $``\subseteq"$  in~\eqref{eq:symcurlindivdivbubble} can be changed to ``=". This will be clean after we present a bubble complex. 

\begin{lemma}\label{lem:divdivbubble}
For each $K\in\mathcal T_h$, it holds
\begin{equation}\label{eq:divdivL2local0}
\div\div\mathring{\boldsymbol \Sigma}_{\ell,k}(K)=\mathbb P_{k-2}(K)\cap \mathbb P_1^{\bot}(K),
\end{equation}
where $\mathbb P_1^{\bot}(K)$ is a subspace of $L^2(K)$ being orthogonal to $\mathbb P_1(K)$ with respect to the $L^2$-inner product $(\cdot, \cdot)_K$. Consequently
\begin{equation}\label{eq:dimdivdivbubble}
\dim(\mathring{\boldsymbol \Sigma}_{\ell,k}(K)\cap\ker(\div\div))=\frac{1}{6}\ell(\ell-1)(5\ell+17). 
\end{equation}
\end{lemma}
\begin{proof}
From the integration by parts, it is obviously true that $\div\div\mathring{\boldsymbol \Sigma}_{\ell,k}(K)\subseteq\mathbb P_{k-2}(K)\cap \mathbb P_1^{\bot}(K)$.
On the other side, for any $v\in\mathbb P_{k-2}(K)\cap \mathbb P_1^{\bot}(K)$, due to the fact that $\div\div \boldsymbol H_0^2(K;\mathbb S)=L^2(K)\cap \mathbb P_1^{\bot}(K)$~\cite{CostabelMcIntosh2010}, 
there exists $\widetilde{\bs\tau}\in\boldsymbol H_0^2(K;\mathbb S)$ such that
\[
\div\div\widetilde{\bs\tau}=v. 
\]
Then take $\boldsymbol \tau \in\mathring{\boldsymbol \Sigma}_{\ell,k}(K)$ with the rest d.o.f. 
\begin{align*}
(\boldsymbol \tau-\widetilde{\bs\tau}, \boldsymbol \varsigma)_K=0&  \quad\forall~\boldsymbol \varsigma\in\nabla^2\mathbb P_{k-2}(K)\oplus\sym(\mathbb P_{\ell-2}(K; \mathbb T)\times\boldsymbol x), \\
((\boldsymbol \tau-\widetilde{\bs\tau})\boldsymbol n, \boldsymbol  n\times \boldsymbol x q)_{F_1}=0 &\quad\forall~q\in\mathbb P_{\ell-2}(F_1).
\end{align*}
Applying the Green's identity~\eqref{eq:greenidentitydivdiv}, we get
\[
(\div\div(\boldsymbol \tau-\widetilde{\bs\tau}), q)_K=0\quad\forall~q\in\mathbb P_{k-2}(K).
\]
This implies $\div\div\boldsymbol \tau=\div\div\widetilde{\bs\tau}=v$. Namely~\eqref{eq:divdivL2local0} holds.

An immediate result of~\eqref{eq:divdivL2local0} is
\begin{align*}
\dim(\mathring{\boldsymbol \Sigma}_{\ell,k}(K)\cap\ker(\div\div))& =\dim\mathring{\boldsymbol \Sigma}_{\ell,k}(K)-\dim\mathbb P_{k-2}(K)+4 \\
&=\frac{1}{6}\ell(\ell-1)(5\ell+14) +\frac{1}{2}\ell(\ell-1)\\
&=\frac{1}{6}\ell(\ell-1)(5\ell+17).
\end{align*}
\end{proof}

Define
$$
\boldsymbol B_{\ell+2}(K;\mathbb R^3):=\{\boldsymbol v\in\mathbb P_{\ell+2}(K;\mathbb R^3): \boldsymbol v|_{\partial K}=\bs0\}=b_K\mathbb P_{\ell-2}(K;\mathbb R^3).
$$
Now we are in the position to present the so-called bubble complex. 

\begin{theorem}\label{th:bubblecomplex}
The bubble function spaces for the divdiv complex
\begin{equation}\label{eq:divdivcomplex3dfembubble}
\resizebox{.92\hsize}{!}{$
0\xrightarrow{} \boldsymbol B_{\ell+2}(K;\mathbb R^3)\xrightarrow{\dev\grad} \boldsymbol B_{\ell+1}(\sym\curl, K;\mathbb T) \xrightarrow{\sym\curl} \mathring{\boldsymbol \Sigma}_{\ell,k}(K) \xrightarrow{\div{\div}} \mathbb P_{k-2}(K)\cap \mathbb P_1^{\bot}(K)\xrightarrow{}0
$}
\end{equation}
form an exact Hilbert complex. 
\end{theorem}
\begin{proof}
Take any $\boldsymbol v\in \boldsymbol B_{\ell+2}(K;\mathbb R^3)$ with $\boldsymbol v|_{\partial K}=\bs0$.
We have on each face $F\in\mathcal F(K)$,
\begin{equation}\label{eq:devgradtr1}
\boldsymbol n\cdot(\dev\grad\boldsymbol v)\times\boldsymbol n=\boldsymbol n\cdot(\grad\boldsymbol v)\times\boldsymbol n=-(\boldsymbol n\times\nabla)(\boldsymbol v\cdot\boldsymbol n)=\bs0,
\end{equation}
and
\begin{align}
\boldsymbol n\times \sym((\dev\grad\boldsymbol v)\times\boldsymbol n)\times\boldsymbol n&=\boldsymbol n\times \sym((\grad\boldsymbol v)\times\boldsymbol n)\times\boldsymbol n \notag\\
&=-\boldsymbol n\times \sym(\boldsymbol v\nabla_F^{\bot})\times\boldsymbol n \notag\\
&=-\boldsymbol n\times \sym((\Pi_F\boldsymbol v)\nabla_F^{\bot})\times\boldsymbol n=\bs0. \label{eq:devgradtr2}
\end{align}
Hence $\dev\grad\boldsymbol B_{\ell+2}(K;\mathbb R^3)\subseteq\boldsymbol B_{\ell+1}(\sym\curl, K;\mathbb T)\cap\ker(\sym\curl)$.
Thanks to Lemma~\ref{lm:symbubble} and~\eqref{eq:divdivL2local0}, we conclude that ~\eqref{eq:divdivcomplex3dfembubble} is a complex.

We then verify the exactness from the left to the right.

\medskip
\noindent {\em 1. $\boldsymbol B_{\ell+1}(\sym\curl, K;\mathbb T)\cap\ker(\sym\curl)=\dev\grad\boldsymbol B_{\ell+2}(K;\mathbb R^3)$, i.e. if $\sym\curl\boldsymbol \tau = \boldsymbol 0$ and $\bs\tau\in \boldsymbol B_{\ell+1}(\sym\curl, K;\mathbb T)$, then there exists a $\boldsymbol v\in\boldsymbol B_{\ell+2}(K;\mathbb R^3)$, s.t. $\boldsymbol \tau = \dev\grad\boldsymbol v$}. 

Firstly, by the exactness of the polynomial divdiv complex~\eqref{eq:divdivcomplex3dPoly}, there exists $\boldsymbol v\in\mathbb P_{\ell+2}(K;\mathbb R^3)$ such that $\bs\tau=\dev\grad\boldsymbol v$. As $\bs{RT}=\ker(\dev\grad)$, we  can further impose constraint $\int_F \boldsymbol v\cdot\boldsymbol n=0$ for each $F\in\mathcal F(K)$.
By~\eqref{eq:devgradtr1}, we get $ \boldsymbol v\cdot\boldsymbol n \mid _F\in\mathbb P_0(F)$. Hence $\boldsymbol v\cdot\boldsymbol n|_F=0$, which indicates $\boldsymbol v(\delta)=\bs0$ for each vertex $\delta\in\mathcal V(K)$.
By~\eqref{eq:devgradtr2}, we obtain $\sym((\Pi_F\boldsymbol v)\nabla_F^{\bot})=\bs0$, i.e. $\Pi_F\boldsymbol v\in\mathbb P_0(F;\mathbb R^2)+(\Pi_F\boldsymbol x)\mathbb P_0(F)$. This combined with $\boldsymbol v(\delta)=\bs0$ for each vertex $\delta\in\mathcal V(F)$ means $\Pi_F\boldsymbol v=\bs0$, and then $\boldsymbol v|_F=\bs0$ for each $F\in \mathcal F(K)$. Thus $\boldsymbol v\in \boldsymbol B_{\ell+2}(K;\mathbb R^3)$.

\medskip
\noindent {\em 2. $\sym\curl\boldsymbol B_{\ell+1}(\sym\curl, K;\mathbb T)=\mathring{\boldsymbol \Sigma}_{\ell,k}(K)\cap\ker(\div\div)$}. 

By step 1, we acquire
\begin{align}
&\quad\,\dim\sym\curl\boldsymbol B_{\ell+1}(\sym\curl, K;\mathbb T) \notag \\
&=\dim\boldsymbol B_{\ell+1}(\sym\curl, K;\mathbb T)-\dim\boldsymbol B_{\ell+2}(K;\mathbb R^3) \notag  \\
&=\dim\boldsymbol B_{\ell+1}(\sym\curl, K;\mathbb T)-\dim\mathbb P_{\ell-2}(K;\mathbb R^3) \notag \\
&=\frac{1}{6}\ell(\ell-1)(5\ell+17) \label{eq:dimsymcurlB},
\end{align}
which together with~\eqref{eq:dimdivdivbubble} indicates
$$
\dim\sym\curl\boldsymbol B_{\ell+1}(\sym\curl, K;\mathbb T)=\dim(\mathring{\boldsymbol \Sigma}_{\ell,k}(K)\cap\ker(\div\div)).$$ 
Together with~\eqref{eq:symcurlindivdivbubble} implies $\sym\curl\boldsymbol B_{\ell+1}(\sym\curl, K;\mathbb T)=\mathring{\boldsymbol \Sigma}_{\ell,k}(K)\cap\ker(\div\div)$.

\medskip
\noindent {\em 3. $\div\div\mathring{\boldsymbol \Sigma}_{\ell,k}(K)=\mathbb P_{k-2}(K)\cap \mathbb P_1^{\bot}(K)$.}  This is~\eqref{eq:divdivL2local0} proved in Lemma~\ref{lem:divdivbubble}.

\medskip

Therefore complex~\eqref{eq:divdivcomplex3dfembubble} is exact.
\end{proof}

As a result of complex~\eqref{eq:divdivcomplex3dfembubble}, we can replace the degrees of freedom~\eqref{Hdivdivfem3ddof55}-\eqref{Hdivdivfem3ddof6} by
\begin{align}\label{Hdivdivfem3ddofbubble}
(\boldsymbol \tau, \boldsymbol \varsigma)_K & \quad\forall~\boldsymbol \varsigma\in\sym\curl\boldsymbol B_{\ell+1}(\sym\curl, K;\mathbb T).
\end{align}
The dimension of~\eqref{Hdivdivfem3ddofbubble} is counted in~\eqref{eq:dimsymcurlB}, which also matches the sum of~\eqref{Hdivdivfem3ddof55}-\eqref{Hdivdivfem3ddof6}. 

We summarize the unisolvence below.

\begin{corollary}\label{cor:unisovlenHdivdivfem2}
The degrees of freedom~\eqref{Hdivdivfem3ddof1}-\eqref{Hdivdivfem3ddof5} and~\eqref{Hdivdivfem3ddofbubble} are unisolvent for $\boldsymbol \Sigma_{\ell,k}(K)$.
\end{corollary}

Notice that although $\boldsymbol B_{\ell+1}(\sym\curl, K;\mathbb T)$ is in a symmetric form, cf.~\eqref{eq:basisofB}, the degrees of freedom~\eqref{Hdivdivfem3ddofbubble} is indeed not simpler than~\eqref{Hdivdivfem3ddof55}-\eqref{Hdivdivfem3ddof6} in computation as $\sym\curl\boldsymbol B_{\ell+1}(\sym\curl, K;\mathbb T)$ is much more complicated than polynomials on a face. 

\subsection{Two dimensional divdiv conforming finite elements}\label{sec:2D}
Recently we have constructed divdiv conforming finite elements in two dimensions in~\cite{ChenHuang2020}. Here we briefly review the results and compare to the three dimensional case. 

Let $F$ be a triangle. Take the space of shape functions
\begin{equation}\label{eq:2DSigma}
\boldsymbol \Sigma_{\ell,k}(F):= \mathbb C_{\ell}(F;\mathbb S)\oplus\mathbb C_k^{\oplus}(F;\mathbb S)
\end{equation}
with $k\geq 3$ and $\ell\geq \max\{k-1, 3\}$ and 
\[
\mathbb C_{\ell}(F; \mathbb S)=\sym \curl_F \, \mathbb  P_{\ell +1}(F; \mathbb R^2),\quad \mathbb C_k^{\oplus}(F; \mathbb S)=\boldsymbol  x\boldsymbol  x^{\intercal}\mathbb P_{k-2}(F).
\]
Here the polynomial space for $\boldsymbol H(\sym \curl, F; \mathbb R^2)$ is the vector space not a tensor space, which simplifies the construction significantly. 

The degrees of freedom are given by
\begin{align}
\boldsymbol \tau (\delta) & \quad\forall~\delta\in \mathcal V(F), \label{Hdivdivfemdof1}\\
(\boldsymbol  n_{e}^{\intercal}\boldsymbol \tau\boldsymbol  n_{e}, q)_e & \quad\forall~q\in\mathbb P_{\ell-2}(e),  e\in\mathcal E(F),\label{Hdivdivfemdof2}\\
(\partial_{t}(\boldsymbol  t^{\intercal}\boldsymbol \tau\boldsymbol  n_{e})+\boldsymbol  n_{e}^{\intercal}\div_F\boldsymbol \tau, q)_e & \quad\forall~q\in\mathbb P_{\ell-1}(e),  e\in\mathcal E(F),\label{Hdivdivfemdof3}\\
(\boldsymbol \tau, \boldsymbol \varsigma)_F & \quad\forall~\boldsymbol \varsigma\in\nabla_F^2\mathbb P_{k-2}(F), \label{Hdivdivfemdof4}\\
(\boldsymbol \tau, \boldsymbol \varsigma)_F & \quad\forall~\boldsymbol \varsigma\in \sym (\boldsymbol x^{\bot} \mathbb P_{\ell-2}(F;\mathbb R^2)). \label{Hdivdivfemdof5}
\end{align}
Here to avoid confusion with three dimensional version, we use $\bs n_e$ to emphasize it is a normal vector of edge vector $e$. 

The unisolvence is again better understood with the help of Fig.~\ref{fig:femdec}. 
By the vanishing degrees of freedom~\eqref{Hdivdivfemdof1}-\eqref{Hdivdivfemdof3}, the trace is vanished. Then together with the vanishing ~\eqref{Hdivdivfemdof4}, $\div \div \boldsymbol \tau = 0$. The rest is to identify the intersection of the bubble space and the kernel of $\div\div$. Define
\begin{align*}
\mathring{\boldsymbol \Sigma}_{\ell,k}(F)&:=\{\boldsymbol  \tau\in\boldsymbol \Sigma_{\ell,k}(F): \textrm{all degrees of freedom~\eqref{Hdivdivfemdof1}-\eqref{Hdivdivfemdof3} vanish}\}.
\end{align*}
It turns out the space $\mathring{\boldsymbol \Sigma}_{\ell,k}(F)\cap \ker(\div_F\div_F)$ is much simpler in two dimensions. 

The key is the following formulae on the trace $\tr_2$. 
\begin{lemma}\label{lm:tauv}
When $\boldsymbol  \tau = \sym \curl_F \, \boldsymbol  v$, we have
\begin{align}
\label{eq:trace2-2d} \partial_{t}(\boldsymbol  t^{\intercal}\boldsymbol \tau\boldsymbol  n_{e})+\boldsymbol  n_{e}^{\intercal}\div_F\boldsymbol \tau & =  \partial_t(\boldsymbol  t^{\intercal}\partial_t\boldsymbol  v).
\end{align}
\end{lemma}
\begin{proof}
Since $\div_F \curl_F \, \boldsymbol  v = 0$, we have
$$\boldsymbol  n^{\intercal}\div_F\boldsymbol \tau  = \frac{1}{2}\boldsymbol  n_{e}^{\intercal} \div_F ( \curl_F \, \boldsymbol  v)^{\intercal} = \frac{1}{2}\boldsymbol  n_{e}^{\intercal} \curl_F \div_F \boldsymbol  v =  \frac{1}{2} \partial_t \div_F \boldsymbol  v.$$
As $\div_F\boldsymbol  v = {\rm trace} (\nabla_F\boldsymbol  v)$ is invariant to the rotation, we can write it as
$$
\div_F \boldsymbol  v = \boldsymbol  t^{\intercal}\nabla_F\boldsymbol  v \boldsymbol  t + \boldsymbol  n_{e}^{\intercal}\nabla_F\boldsymbol  v \boldsymbol  n_{e} =  \boldsymbol  t^{\intercal}\partial_t\boldsymbol  v + \boldsymbol  n_{e}^{\intercal}\partial_n\boldsymbol  v.
$$
Then
$$
 \partial_{t}(\boldsymbol  t^{\intercal}\boldsymbol \tau\boldsymbol  n_{e})+\boldsymbol  n_{e}^{\intercal}\div_F\boldsymbol \tau  =\frac{1}{2}\partial_t[  \boldsymbol  t^{\intercal}\partial_t\boldsymbol  v -  \boldsymbol  n_{e}^{\intercal}\partial_n\boldsymbol  v + \div_F\boldsymbol  v]  =  \partial_t  (\boldsymbol  t^{\intercal}\partial_t\boldsymbol  v),
 $$
 i.e.~\eqref{eq:trace2-2d} holds.
\end{proof}

\begin{lemma}
The following bubble complex: 
\begin{equation*}
\bs0\xrightarrow{\subset} b_F\mathbb P_{\ell-2}(F;\mathbb R^2)\xrightarrow{\sym\curl_F} \mathring{\boldsymbol \Sigma}_{\ell,k}(F)\xrightarrow{\div_F\div_F} \mathbb P_{k-2}(F)\cap\mathbb P_1^{\perp}(F)\xrightarrow{}0
\end{equation*}
is exact.
\end{lemma}
\begin{proof}
The fact $\div_F\div_F: \mathring{\boldsymbol \Sigma}_{\ell,k}(F) \to \mathbb P_{k-2}(F)\cap\mathbb P_1^{\perp}(F)$ is surjective can be proved similarly to Lemma~\ref{lem:divdivbubble}. 

For $\boldsymbol \tau \in \mathring{\boldsymbol \Sigma}_{\ell,k}(F)\cap\ker(\div_F\div_F)$, from the complex~\eqref{eq:divdivcomplexPolydouble2D}, we can find $\boldsymbol v\in \mathbb P_{\ell+1}(F)$ s.t. $\sym \curl_F \boldsymbol v = \boldsymbol \tau$.  We will prove $\boldsymbol v|_{\partial F} = \boldsymbol 0$. 

Since $\bs{RT}=\ker(\sym\curl_F)$, we can further impose constraint $\int_e \boldsymbol v\cdot\boldsymbol n_{e}=0$ for each $e\in\mathcal E(F)$. The fact $(\boldsymbol n_{e}^{\intercal}\boldsymbol \tau\boldsymbol  n_{e})|_{\partial F}=0$ implies
\[
\partial_t(\boldsymbol  n_{e}^{\intercal}\boldsymbol  v)|_{\partial F}= (\boldsymbol  n_{e}^{\intercal}\boldsymbol \tau\boldsymbol  n_{e})|_{\partial F}=0.
\]
Hence $\boldsymbol  n_{e}^{\intercal}\boldsymbol  v|_{\partial F}=0$.
This also means $\boldsymbol  v(\delta)=\boldsymbol 0$ for each $\delta\in\mathcal V(F)$.

By Lemma~\ref{lm:tauv}, since
\begin{equation*}
\partial_{t}(\boldsymbol  t^{\intercal}\boldsymbol \tau\boldsymbol  n_{e})+\boldsymbol  n_{e}^{\intercal}\div_F\boldsymbol \tau= \partial_t(\boldsymbol  t^{\intercal}\partial_t\boldsymbol  v)
\end{equation*}
and $(\partial_{t}(\boldsymbol  t^{\intercal}\boldsymbol \tau\boldsymbol  n_{e})+\boldsymbol  n_{e}^{\intercal}\div_F\boldsymbol \tau)|_{\partial F}=0$, we acquire
\[
\partial_{tt}(\boldsymbol  t^{\intercal}\boldsymbol  v)|_{\partial F}=0.
\]
That is $\boldsymbol  t^{\intercal}\boldsymbol  v|_e\in\mathbb P_1(e)$ on each edge $e\in\mathcal E(F)$. Noting that $\boldsymbol  v(\delta)=\boldsymbol 0$ for each $\delta\in\mathcal V(F)$, we get $\boldsymbol t^{\intercal}\boldsymbol  v|_{\partial F}=0$ and consequently $\boldsymbol  v|_{\partial F}=\boldsymbol 0$, i.e., 
$$
\boldsymbol v = b_F\psi_{\ell-2},\quad \text{for some } \psi_{\ell-2} \in \mathbb P_{\ell-2}(F;\mathbb R^2).
$$
\end{proof}

We now prove the unisolvence as follows.

\begin{theorem}\label{thm:unisovlenHdivdivfem2D}
The degrees of freedom~\eqref{Hdivdivfemdof1}-\eqref{Hdivdivfemdof4} are unisolvent for $\boldsymbol \Sigma_{\ell,k}(F)$~\eqref{eq:2DSigma}.
\end{theorem}
\begin{proof}
We first count the number of the degrees of freedom~\eqref{Hdivdivfemdof1}-\eqref{Hdivdivfemdof4}  and the dimension of the space, i.e., $\dim\boldsymbol \Sigma_{\ell,k}(K)$.
Both of them are $$\ell^2+5\ell+3+\frac{1}{2}k(k-1).$$
Then suppose all the degrees of freedom~\eqref{Hdivdivfemdof1}-\eqref{Hdivdivfemdof4} applied to $\boldsymbol \tau$ vanish. We are going to prove the function $\boldsymbol \tau = 0$. 

By the vanishing degrees of freedom~\eqref{Hdivdivfemdof1}-\eqref{Hdivdivfemdof3}, the two traces are vanished. Together with~\eqref{Hdivdivfemdof4}, the Green's identity implies $\div_F\div_F \boldsymbol \tau = 0$. Then $$
\boldsymbol \tau = \sym \curl_{F}  (b_F\psi_{\ell-2}),\quad \text{for some } \psi_{\ell-2} \in \mathbb P_{\ell-2}(F;\mathbb R^2).
$$
We then use the fact $\rot_F: \sym(\boldsymbol x^{\perp}\mathbb P_{\ell-2}(F;\mathbb R^2)) \to  \mathbb P_{\ell-2}(F;\mathbb R^2)$ is bijection, cf. the complex~\eqref{eq:hessiancomplexPolydouble2D}, to find $\phi_{\ell-2}$ s.t. $\rot_F (\sym(\boldsymbol x^{\perp}\phi_{\ell-2})) = \psi_{\ell-2}$.
Finally we finish the unisolvence proof by choosing $\boldsymbol \varsigma = \sym (\boldsymbol x^{\perp} \phi_{\ell -2})$ in~\eqref{Hdivdivfemdof5}. The fact 
$$
(\boldsymbol \tau, \boldsymbol \varsigma)_F = (\sym\curl_F(b_F\psi_{\ell-2}),  \sym (\boldsymbol x^{\perp} \phi_{\ell -2}))_F = ( b_F\psi_{\ell-2}, \psi_{\ell-2} )_F = 0
$$
will imply $\psi_{\ell-2} = 0$ and consequently $\boldsymbol \tau = 0$.  
\end{proof}
As finite element spaces for $\boldsymbol  H^1$ are relatively mature and the bubble function space of $\mathbb P_{\ell+1}(F;\mathbb R^2)\cap \boldsymbol H_0^1(F;\mathbb R^2) = b_F\mathbb P_{\ell-2}(F;\mathbb R^2)$, the design of divdiv conforming finite elements in two dimensions is relatively easy. By rotation, we can also get finite elements for the strain space $\boldsymbol{H}(\rot_F\rot_F,F; \mathbb{S})$; see~\cite[Section 3.4]{ChenHuang2020}.

\section{Finite Elements for Sym Curl-Conforming Trace-Free Tensors}\label{sec:tracefreefem}
In this section we construct conforming finite element spaces for $\boldsymbol{H}(\sym\curl,\Omega; \mathbb{T})$.

\subsection{A finite element space}
Let $K$ be a tetrahedron. For each edge $e$, we have a direction vector $\bs t$ and then chose two orthonormal vectors $\boldsymbol n_1$ and $\boldsymbol n_2$ being orthogonal to $e$ such that $\boldsymbol n_2=\boldsymbol t\times\boldsymbol n_1$ and $\boldsymbol n_1=\boldsymbol n_2\times\boldsymbol t$. 
Take the space of shape functions as $\mathbb P_{\ell+1}(K;\mathbb T)$.
The degrees of freedom $\mathcal N_{\ell+1}(K)$ are given by
{\small
\begin{align}
\boldsymbol \tau(\delta) & \quad\forall~\delta\in \mathcal V(K), \label{Hsymcurlfem3ddof1}\\
(\sym\curl\boldsymbol \tau )(\delta) & \quad\forall~\delta\in \mathcal V(K), \label{Hsymcurlfem3ddof2}\\
(\boldsymbol  n_i^{\intercal}(\sym\curl\boldsymbol \tau )\boldsymbol n_j, q)_e & \quad\forall~q\in\mathbb P_{\ell-2}(e),  e\in\mathcal E(K), i,j=1,2, \label{Hsymcurlfem3ddof3}\\
(\boldsymbol  n_i^{\intercal}\boldsymbol \tau\boldsymbol t, q)_e & \quad\forall~q\in\mathbb P_{\ell-1}(e),  e\in\mathcal E(K), i=1,2,\label{Hsymcurlfem3ddof4}\\
(\boldsymbol n_{2}^{\intercal}(\curl\boldsymbol \tau)\boldsymbol n_1 + \partial_{t}(\boldsymbol t^{\intercal}\bs\tau\boldsymbol t), q)_e & \quad\forall~q\in\mathbb P_{\ell}(e),  e\in\mathcal E(K),\label{Hsymcurlfem3ddof5}\\
(\boldsymbol n\times\sym(\boldsymbol\tau\times\boldsymbol n)\times\boldsymbol n, \boldsymbol \varsigma)_F &\quad\forall~\boldsymbol \varsigma\in (\nabla_F^{\bot})^2\, \mathbb P_{\ell-1}(F)\oplus
 \sym (\boldsymbol x\otimes \mathbb P_{\ell-1}(F;\mathbb R^2)), \label{Hsymcurlfem3ddof6}\\
(\boldsymbol n\cdot \boldsymbol\tau\times\boldsymbol n, \boldsymbol q)_F & \quad\forall~\boldsymbol q\in\nabla_F\mathbb P_{\ell-3}(F)\oplus\boldsymbol x^{\perp}\mathbb P_{\ell-1}(F),  F\in\mathcal F(K),\label{Hsymcurlfem3ddof7}\\
(\boldsymbol \tau, \boldsymbol q)_K & \quad\forall~\boldsymbol q\in\boldsymbol B_{\ell+1}(\sym\curl, K;\mathbb T). \label{Hsymcurlfem3ddof8} 
\end{align}}

The degree of freedom~\eqref{Hsymcurlfem3ddof2},~\eqref{Hsymcurlfem3ddof3}, and~\eqref{Hsymcurlfem3ddof8} are motivated by~\eqref{Hdivdivfem3ddof1},~\eqref{Hdivdivfem3ddof2}, and~\eqref{Hdivdivfem3ddofbubble}, respectively, as $\sym \curl \boldsymbol \tau \in \boldsymbol H(\div\div, K;\mathbb S)$. Recall that $\tr_2(\boldsymbol \tau)\in H(\div_F)$ and $\tr_1(\boldsymbol \tau) \in H(\div_F\div_F)$, cf. Lemma \ref{lm:tracerelation}. Let $\bs n_{F,e} = \bs t\times \bs n$ be the norm vector of $e$ sitting on the face $F$. 
For $\div_F$ elements on face $F$, the normal trace becomes
\begin{equation*}
(\boldsymbol n\cdot \boldsymbol \tau\times\boldsymbol n)\cdot \boldsymbol n_{F,e}=\boldsymbol n^{\intercal}\boldsymbol\tau\boldsymbol t,
\end{equation*}
which motivates d.o.f.~\eqref{Hsymcurlfem3ddof4}. Together with d.o.f.~\eqref{Hsymcurlfem3ddof7}, $\boldsymbol n\cdot \boldsymbol \tau\times\boldsymbol n$ can be determined. 
For the $\div_F\div_F$ element, the normal-normal trace becomes
\begin{equation}\label{eq:symcurltracepart1}
\boldsymbol n_{F,e}^{\intercal}(\Pi_F\sym(\boldsymbol\tau\times\boldsymbol n)\Pi_F)\boldsymbol n_{F,e}=\boldsymbol n_{F,e}^{\intercal}\sym(\boldsymbol\tau\times\boldsymbol n)\boldsymbol n_{F,e}= \boldsymbol n_{F,e}^{\intercal}\boldsymbol\tau\boldsymbol t,
\end{equation}
which can be also determined by~\eqref{Hsymcurlfem3ddof4}. Notice that for each edge $e$, there are two $\boldsymbol n_{F,e}$ inside one tetrahedron. In~\eqref{Hsymcurlfem3ddof4}, the two normal vectors $\bs n_1, \bs n_2$ are chosen independent of elements and~\eqref{Hsymcurlfem3ddof4} can determine the projection of vector $\bs \tau\bs t$ to the plane orthogonal to edge $e$ including $\boldsymbol n_{F,e}^{\intercal}\boldsymbol\tau\boldsymbol t$. 

The other trace of a $\div_F\div_F$ element will be determined by~\eqref{Hsymcurlfem3ddof3} and~\eqref{Hsymcurlfem3ddof5}, which is less obvious. The following lemma is borrowed from~\cite[Lemma 9 and Remark 8]{Hu;Liang;Ma:2021Finite}. 

\begin{lemma}\label{lem:symcurlkey}
Let $F\in\mathcal F(K)$ with a normal vector $\bs n_F$. For an edge $e\in\mathcal E(F)$, we fix a direction vector $\bs t$ for $e$ and chose two orthonormal vectors $\boldsymbol n_1$ and $\boldsymbol n_2$ being orthogonal to $e$ such that $\boldsymbol n_2=\boldsymbol t\times\boldsymbol n_1$ and $\boldsymbol n_1=\boldsymbol n_2\times\boldsymbol t$. Let $\boldsymbol n_{F,e}=\boldsymbol t \times\boldsymbol n_{F}$.
For any sufficiently smooth tensor $\bs\tau$, we have
\begin{align}
\boldsymbol n_{F,e}^{\intercal}(\curl\boldsymbol \tau)\boldsymbol n_F
&=(\boldsymbol n_F\cdot\boldsymbol n_1) (\boldsymbol n_F\cdot\boldsymbol n_2) \left [\boldsymbol n_{2}^{\intercal}(\sym\curl\boldsymbol \tau)\boldsymbol n_2-\boldsymbol n_{1}^{\intercal}(\sym\curl\boldsymbol \tau)\boldsymbol n_1\right ] \notag\\
&\quad - 2(\boldsymbol n_F\cdot\boldsymbol n_2)^2\boldsymbol n_{1}^{\intercal}(\sym\curl\boldsymbol \tau)\boldsymbol n_2 + \boldsymbol n_{2}^{\intercal}(\curl\boldsymbol \tau)\boldsymbol n_1, \label{eq:edgedofprop1}
\end{align}
For ${\rm tr}_1(\boldsymbol \tau)=\Pi_F\sym(\bs\tau\times\boldsymbol n_F)\Pi_F$, we have
\begin{equation}\label{eq:edgedofprop2}
\partial_{t}(\boldsymbol  t^{\intercal}{\rm tr}_1(\boldsymbol \tau)\boldsymbol  n_{F,e}) +\boldsymbol  n_{F,e}^{\intercal}\div_F({\rm tr}_1(\boldsymbol \tau)) = \boldsymbol n_{F,e}^{\intercal}(\curl\boldsymbol \tau)\boldsymbol n_F + \partial_{t}(\boldsymbol t^{\intercal}\bs\tau\boldsymbol t).
\end{equation}
Consequently it can be determined by d.o.f.~\eqref{Hsymcurlfem3ddof3} and~\eqref{Hsymcurlfem3ddof5}. 
\end{lemma}
\begin{proof}
On the plane orthogonal to $e$, the vectors $\bs n_1$ and $\bs n_2$ form an orthonormal basis. We expand $\boldsymbol n_F= c_1 \boldsymbol n_1+c_2\boldsymbol n_2$ in this coordinate, with $c_i = \boldsymbol n_F\cdot\boldsymbol n_i$ for $i=1,2$. Then $
\boldsymbol n_{F,e}=\boldsymbol t \times\boldsymbol n_{F}=c_1\boldsymbol n_2 - c_2 \boldsymbol n_1.$
Then in this coordinate
\begin{align*}
\boldsymbol n_{F,e}^{\intercal}(\curl\boldsymbol \tau)\boldsymbol n_F&=(c_1\boldsymbol n_2-c_2\boldsymbol n_1)^{\intercal}(\curl\boldsymbol \tau)(c_1\boldsymbol n_1+c_2\boldsymbol n_2) \\
&=c_1c_2(\boldsymbol n_{2}^{\intercal}(\curl\boldsymbol \tau)\boldsymbol n_2-\boldsymbol n_{1}^{\intercal}(\curl\boldsymbol \tau)\boldsymbol n_1) \\
&\quad+ c_1^2\boldsymbol n_{2}^{\intercal}(\curl\boldsymbol \tau)\boldsymbol n_1 - c_2^2\boldsymbol n_{1}^{\intercal}(\curl\boldsymbol \tau)\boldsymbol n_2.
\end{align*}
Thus we acquire~\eqref{eq:edgedofprop1} from the fact $c_1^2+c_2^2=1$.

On the other hand, by the fact $\nabla_F=\boldsymbol t\partial_{t}+\boldsymbol n_{F,e}\partial_{n_{F,e}}$, we obtain
\begin{align*}
&\partial_{t}(\boldsymbol t^{\intercal}{\rm tr}_1(\boldsymbol \tau)\boldsymbol  n_{F,e}) +\boldsymbol  n_{F,e}^{\intercal}\div_F({\rm tr}_1(\boldsymbol \tau)) \\
=&2\partial_{t}(\boldsymbol t^{\intercal}{\rm tr}_1(\boldsymbol \tau)\boldsymbol  n_{F,e}) + \partial_{n_{F,e}}(\boldsymbol  n_{F,e}^{\intercal} {\rm tr}_1(\boldsymbol \tau)\boldsymbol n_{F,e}) \\
=&2\partial_{t}(\boldsymbol  t^{\intercal} \sym(\bs\tau\times\boldsymbol n_F)\boldsymbol  n_{F,e}) + \partial_{n_{F,e}}(\boldsymbol  n_{F,e}^{\intercal} \sym(\bs\tau\times\boldsymbol n_F)\boldsymbol  n_{F,e}) \\
=&\partial_{t}(\boldsymbol t^{\intercal}\boldsymbol\tau\boldsymbol  t - \boldsymbol  n_{F,e}^{\intercal}\boldsymbol\tau\boldsymbol  n_{F,e}) + \partial_{n_{F,e}}(\boldsymbol n_{F,e}^{\intercal}\boldsymbol\tau\boldsymbol t),
\end{align*}
and
\begin{align*}
\boldsymbol n_{F,e}^{\intercal}(\curl\boldsymbol \tau)\boldsymbol n_F&=(\boldsymbol n_F\times\nabla)\cdot(\boldsymbol n_{F,e}^{\intercal}\bs\tau)=(\boldsymbol n_F\times\nabla)\cdot(\boldsymbol n_{F,e}^{\intercal}\bs\tau\boldsymbol t\boldsymbol t+\boldsymbol n_{F,e}^{\intercal}\bs\tau\boldsymbol n_{F,e}\boldsymbol n_{F,e}) \\
&=\partial_{n_{F,e}}(\boldsymbol n_{F,e}^{\intercal}\bs\tau\boldsymbol t) - \partial_{t}(\boldsymbol n_{F,e}^{\intercal}\bs\tau\boldsymbol n_{F,e}).
\end{align*}
Therefore~\eqref{eq:edgedofprop2} is true.
\end{proof}

The trace $\partial_{t}(\boldsymbol  t^{\intercal}{\rm tr}_1(\boldsymbol \tau)\boldsymbol  n_{F,e}) +\boldsymbol  n_{F,e}^{\intercal}\div_F({\rm tr}_1(\boldsymbol \tau))$ depends on $F$. For one edge $e$ in a tetrahedron $K$, there are two such traces. Lemma~\ref{lem:symcurlkey} shows that these two traces are linear dependent and only one d.o.f.~\eqref{Hsymcurlfem3ddof5} is needed. 

\begin{lemma}\label{lem:trace0}
Let $F\in\mathcal F(K)$ and $\boldsymbol \tau\in\mathbb P_{\ell+1}(K;\mathbb T)$.
If all the degrees of freedom~\eqref{Hsymcurlfem3ddof1}-\eqref{Hsymcurlfem3ddof7} vanish, then $\boldsymbol n\cdot \boldsymbol\tau\times\boldsymbol n =\boldsymbol 0$ and $\boldsymbol n\times\sym(\bs\tau\times\boldsymbol n)\times\boldsymbol n =\bs0$ on face $F$.
\end{lemma}
\begin{proof}
It follows from~\eqref{eq:trace1},~\eqref{Hsymcurlfem3ddof4} and the first part of~\eqref{Hsymcurlfem3ddof7} that
\[
(\boldsymbol  n^{\intercal}(\sym\curl\boldsymbol\tau)\boldsymbol  n, q)_F=(\div_F(\boldsymbol n\cdot \boldsymbol \tau\times\boldsymbol n), q)_F=0\quad\forall~q\in\mathbb P_{\ell-3}(F).
\]
This combined with~\eqref{Hsymcurlfem3ddof2}-\eqref{Hsymcurlfem3ddof3} yields $\boldsymbol n^{\intercal}(\sym\curl\boldsymbol\tau)\boldsymbol n|_F=\bs0$, i.e. $\div_F(\boldsymbol n\cdot \boldsymbol \tau\times\boldsymbol n|_F)=0$.
Thanks to the unisolvence of BDM element, we achieve $\boldsymbol n\cdot\boldsymbol\tau\times\boldsymbol n |_F=\boldsymbol 0$ from~\eqref{Hsymcurlfem3ddof4} and the second part of~\eqref{Hsymcurlfem3ddof7}.

Let $\boldsymbol \sigma=\Pi_F\sym(\boldsymbol\tau\times\boldsymbol n_F)\Pi_F$ for simplicity. 
Thanks to~\eqref{eq:symcurltracepart1}, we get from~\eqref{Hsymcurlfem3ddof4} that $\boldsymbol n_{F,e}^{\intercal}\bs\sigma\boldsymbol n_{F,e}=0$ on each edge $e\in\mathcal E(F)$.
By~\eqref{eq:edgedofprop1}-\eqref{eq:edgedofprop2}, it follows from~\eqref{Hsymcurlfem3ddof2}-\eqref{Hsymcurlfem3ddof3} and~\eqref{Hsymcurlfem3ddof5} that
$(\partial_{t}(\boldsymbol  t^{\intercal} \bs\sigma\boldsymbol n_{F,e}) +\boldsymbol  n_{F,e}^{\intercal}\div_F\bs\sigma)|_e=0$, which together with~\eqref{Hsymcurlfem3ddof6} and 
the unisolvence of divdiv element in two dimensions, 
i.e. Theorem~\ref{thm:unisovlenHdivdivfem2D},
implies that $\bs\sigma|_F=\bs0$.
\end{proof}

We are in the position to prove the unisolvence. 
\begin{theorem}\label{lem:unisovlenHsymcurlfem}
The degrees of freedom~\eqref{Hsymcurlfem3ddof1}-\eqref{Hsymcurlfem3ddof8} are unisolvent for $\mathbb P_{\ell+1}(K;\mathbb T)$.
\end{theorem}
\begin{proof}
It is easy to see that
\begin{align*}
\#\mathcal N_{\ell+1}(K)&= 56+ 6(6\ell-2) + 4\left(2\ell(\ell+1) + \frac{1}{2}(\ell-1)(\ell-2)-4\right) \\
&\quad+ \frac{1}{3}(4\ell^3+6\ell^2-10\ell)
=\frac{4}{3}(\ell+4)(\ell+3)(\ell+2)\\
&=\dim\mathbb P_{\ell+1}(K;\mathbb T).
\end{align*}

Take any $\boldsymbol \tau\in\mathbb P_{\ell+1}(K;\mathbb T)$ and
suppose all the degrees of freedom~\eqref{Hsymcurlfem3ddof1}-\eqref{Hsymcurlfem3ddof8} vanish. Then by Lemmas~\ref{lem:trace0}, $\boldsymbol \tau\in \boldsymbol B_{\ell+1}(\sym\curl, K;\mathbb T)$. Then taking $\boldsymbol q = \bs \tau$ in~\eqref{Hsymcurlfem3ddof8}, we conclude $\boldsymbol \tau = 0$. 
\end{proof}

\subsection{Lagrange-type Degree of freedoms}
The d.o.f. $\mathcal N_{\ell +1}$ is designed to form a finite element divdiv complex. If the exactness of the sequence is not the concern, we can construct simpler degree of freedoms.  Below is the Lagrange-type $\boldsymbol H(\sym\curl)$-conforming finite elements for trace-free tensors.
Take the space of shape functions as $\mathbb P_{\ell+1}(K;\mathbb T)$.
The degrees of freedom are given by
\begin{align}
\boldsymbol \tau(\delta) & \quad\forall~\delta\in \mathcal V(K), \label{Hsymcurlfem1dof1}\\
(\boldsymbol \tau, \boldsymbol q)_e & \quad\forall~\boldsymbol q\in\mathbb P_{\ell-1}(e;\mathbb T),  e\in\mathcal E(K),\label{Hsymcurlfem1dof2}\\
(\boldsymbol n\times \sym(\boldsymbol\tau\times\boldsymbol n)\times \boldsymbol n, \boldsymbol q)_F &\quad\forall~\boldsymbol q\in\mathbb P_{\ell-2}(F;\mathbb S), F\in\mathcal F(K),\label{Hsymcurlfem1dof3}\\
(\boldsymbol n\cdot \boldsymbol\tau\times\boldsymbol n, \boldsymbol q)_F & \quad\forall~\boldsymbol q\in\mathbb P_{\ell-2}(F;\mathbb R^2),  F\in\mathcal F(K),\label{Hsymcurlfem1dof4}\\
(\boldsymbol \tau, \boldsymbol q)_K & \quad\forall~\boldsymbol q\in\boldsymbol B_{\ell+1}(\sym\curl, K;\mathbb T). \label{Hsymcurlfem1dof5} 
\end{align}
It is straightforward to verify the unisolvence of~\eqref{Hsymcurlfem1dof1}-\eqref{Hsymcurlfem1dof5}  due to the characterization of trace operators and bubble functions. 


We can also take another set of degrees of freedom
\begin{align}
\boldsymbol \tau(\delta) & \quad\forall~\delta\in \mathcal V(K), \label{Hsymcurlfem2dof1}\\
(\boldsymbol  n_i^{\intercal}\boldsymbol \tau\boldsymbol t, q)_e & \quad\forall~q\in\mathbb P_{\ell-1}(e),  e\in\mathcal E(K), i=1,2,\label{Hsymcurlfem2dof2}\\
(\boldsymbol n\times \sym(\boldsymbol\tau\times\boldsymbol n)\times \boldsymbol n, \boldsymbol q)_F &\quad\forall~\boldsymbol q\in\mathring{\mathbb P}_{\ell}(F;\mathbb S), F\in\mathcal F(K),\label{Hsymcurlfem2dof3}\\
(\boldsymbol n\cdot \boldsymbol\tau\times\boldsymbol n, \boldsymbol q)_F & \quad\forall~\boldsymbol q\in\nabla_F\mathbb P_{\ell}(F)\oplus\boldsymbol x^{\perp}\mathbb P_{\ell-1}(F),  F\in\mathcal F(K),\label{Hsymcurlfem2dof4}\\
(\boldsymbol \tau, \boldsymbol q)_K & \quad\forall~\boldsymbol q\in\boldsymbol B_{\ell+1}(\sym\curl, K;\mathbb T), \label{Hsymcurlfem2dof5} 
\end{align}
where 
$$
\mathring{\mathbb P}_{\ell}(F;\mathbb S):=\{\boldsymbol q\in\mathbb P_{\ell}(F;\mathbb S): (\boldsymbol t_1^{\intercal}\boldsymbol q\boldsymbol t_2)(\delta)=0\textrm{ for each } \delta\in\mathcal V(K)\}
$$
with $\boldsymbol t_1$ and $\boldsymbol t_2$ being the unit tangential vectors of two edges of $F$ sharing $\delta$.
The degree of freedom~\eqref{Hsymcurlfem2dof3} is motivated by the Hellan-Herrmann-Johnson mixed method for the Kirchhoff plate bending problems \cite{Hellan1967,Herrmann1967,Johnson1973}.



\section{A Finite Element Divdiv Complex in Three Dimensions}\label{sec:femdivdivcomplex}
In this section, we collect finite element spaces defined before to form a finite element div-div complex. We assume $\mathcal T_h$ is a triangulation of a topological trivial domain $\Omega$. 

\subsection{A finite element $\textrm{divdiv}$ complex}
We start from the vectorial Hermite element space in three dimensions \cite{Ciarlet1978}
\begin{align*}
\boldsymbol V_h:=\{\boldsymbol v_h\in \boldsymbol H^1(\Omega;\mathbb R^3): &\boldsymbol v_h|_K\in\mathbb P_{\ell+2}(K;\mathbb R^3)\textrm{ for each } K\in\mathcal T_h,\\
&\nabla \boldsymbol v_h(\delta) \textrm{ is single-valued at each vertex } \delta\in\mathcal V_h  
\}.
\end{align*}
The local degrees of freedom for $\boldsymbol V_h(K):=\boldsymbol V_h|_K$ are
\begin{align}
\boldsymbol v(\delta), \nabla \boldsymbol v(\delta) & \quad\forall~\delta\in \mathcal V(K), \label{Hermitfem3ddof1}\\
(\boldsymbol v, \boldsymbol q)_e & \quad\forall~\boldsymbol q\in\mathbb P_{\ell-2}(e;\mathbb R^3),  e\in\mathcal E(K), \label{Hermitfem3ddof2}\\
(\boldsymbol v, \boldsymbol q)_F & \quad\forall~\boldsymbol q\in\mathbb P_{\ell-1}(F;\mathbb R^3),  F\in\mathcal F(K),\label{Hermitfem3ddof3}\\
(\boldsymbol v, \boldsymbol q)_K & \quad\forall~\boldsymbol q\in\mathbb P_{\ell-2}(K;\mathbb R^3). \label{Hermitfem3ddof4} 
\end{align}
The unisolvence for $\bs V_h(K)$ is trivial. 
And
$$
\dim\boldsymbol V_h= 12\#\mathcal V_h+ 3(\ell-1)\#\mathcal E_h + \frac{3}{2}(\ell+1)\ell\#\mathcal F_h +\frac{1}{2}(\ell^3-\ell)\#\mathcal T_h.
$$
Let
\begin{align*}
\bs\Sigma_h^{\mathbb T}:=\{\bs\tau_h\in \boldsymbol L^2(\Omega; \mathbb T): &\bs\tau_h|_K\in\mathbb P_{\ell+1}(K;\mathbb T) \textrm{ for each } K\in\mathcal T_h, \textrm{ all the } \\
& \textrm{ degrees of freedom~\eqref{Hsymcurlfem3ddof1}-\eqref{Hsymcurlfem3ddof7} are single-valued} \},
\end{align*}
then
\begin{align*}
\dim\bs\Sigma_h^{\mathbb T}&= 14\#\mathcal V_h+ (6\ell-2)\#\mathcal E_h + \left(2\ell(\ell+1) + \frac{1}{2}(\ell-1)(\ell-2)-4\right)\#\mathcal F_h \\
&\quad+ \frac{1}{3}(4\ell^3+6\ell^2-10\ell) \#\mathcal T_h.
\end{align*}
Clearly Lemma~\ref{lem:trace0} ensures $\bs\Sigma_h^{\mathbb T}\subset \boldsymbol H(\sym\curl,\Omega;\mathbb T)$. 
Let
\begin{align*}
\bs\Sigma_h^{\mathbb S}:=\{\bs\tau_h\in \boldsymbol L^2(\Omega; \mathbb S): &\bs\tau_h|_K\in\boldsymbol \Sigma_{\ell,k}(K) \textrm{ for each } K\in\mathcal T_h, \textrm{ all the } \\
& \textrm{ degrees of freedom~\eqref{Hdivdivfem3ddof1}-\eqref{Hdivdivfem3ddof4} are single-valued} \},
\end{align*}
then
\begin{align*}
\dim\bs\Sigma_h^{\mathbb S}&= 6\#\mathcal V_h+ 3(\ell-1)\#\mathcal E_h + (\ell^2 -\ell+1)\#\mathcal F_h \\
&\quad+\left(\frac{1}{2}\ell(\ell-1)+ \frac{1}{6}(\ell-1)\ell(5\ell+14) + \frac{1}{6}(k^3-k)-4\right)\#\mathcal T_h.
\end{align*}
It follows from the proof of Lemma~\ref{lem:vanishdof} ensures $\bs\Sigma_h^{\mathbb S}\subset \boldsymbol H(\div\div,\Omega;\mathbb S)$.
Let 
$$
\mathcal Q_h :=\mathbb P_{k-2}(\mathcal T_h)=\{q_h\in L^2(\Omega): q_h|_K\in \mathbb P_{k-2}(K) \textrm{ for each } K\in\mathcal T_h\}
$$
be the discontinuous polynomial space. Obviously 
$$
\dim \mathcal Q_h = \frac{1}{6}(k^3 - k)\#\mathcal T_h.
$$

\begin{lemma}\label{lem:divdivonto}
It holds
$$
\div\div\boldsymbol\Sigma_h^{\mathbb S}=\mathcal Q_h.
$$
\end{lemma}
\begin{proof}
Apparently $\div\div\boldsymbol\Sigma_h^{\mathbb S}\subseteq\mathcal Q_h$. Then we focus on $\mathcal Q_h\subseteq\div\div\boldsymbol\Sigma_h^{\mathbb S}$.

Take any $v_h\in\mathcal Q_h$.  By the fact $\div\div \boldsymbol H^2(\Omega;\mathbb S)=L^2(\Omega)$~\cite{CostabelMcIntosh2010}, 
there exists $\bs\tau\in\boldsymbol H^2(\Omega;\mathbb S)$ such that
\[
\div\div\bs\tau=v_h. 
\]
Let $I_h{\bs\tau}\in\boldsymbol\Sigma_h^{\mathbb S}$ be determined by 
$$
\bs N (I_h\bs \tau) = \bs N(\bs \tau)
$$
for all d.o.f. $\bs N$ from~\eqref{Hdivdivfem3ddof1} to~\eqref{Hdivdivfem3ddof6}.
Note that for functions in $H^2(K)$, the integrals on edge and pointwise value are well-defined.
Since $\ell\geq3$,
it follows from the Green's identity~\eqref{eq:greenidentitydivdiv} that
\[
\left(\div\div(\boldsymbol \tau - I_h{\bs\tau}), q \right)_K=0\quad\forall~q\in\mathbb P_1(K), \; K\in\mathcal T_h.
\]
Hence $(v_h-\div\div I_h \bs \tau)|_K=\div\div(\boldsymbol \tau - I_h \bs \tau)|_K\in \mathbb P_{k-2}(K)\cap \mathbb P_1^{\bot}(K)$.
Applying~\eqref{eq:divdivL2local0}, there exists $\bs\tau_b\in\boldsymbol\Sigma_h^{\mathbb S}$ such that $\bs\tau_b|_K \in \mathring{\boldsymbol \Sigma}_{\ell,k}(K)$ for each $K\in\mathcal T_h$, and
\[
v_h-\div\div I_h \bs \tau=\div\div\bs\tau_b.
\]
Therefore $v_h=\div\div(I_h \bs \tau+\bs\tau_b)$, where $I_h \bs \tau+\bs\tau_b\in\boldsymbol\Sigma_h^{\mathbb S}$, as required.
\end{proof}

\begin{theorem}
Assume $\Omega$ is a bounded and topologically trivial Lipschitz domain in $\mathbb R^3$.
The finite element divdiv complex
\begin{equation}\label{eq:divdivcomplex3dfem}
\boldsymbol{RT}\xrightarrow{\subset} \boldsymbol V_h\xrightarrow{\dev\grad}\bs\Sigma_h^{\mathbb T}\xrightarrow{\sym\curl} \bs\Sigma_h^{\mathbb S} \xrightarrow{\div{\div}} \mathcal Q_h\xrightarrow{}0
\end{equation}
is exact.
\end{theorem}
\begin{proof}
For any sufficient vector function $\boldsymbol v$ and $e\in\mathcal E(K)$, we have from $\boldsymbol t=\boldsymbol n_1\times \boldsymbol n_2$ that
\begin{align*}
&\quad\boldsymbol n_{2}^{\intercal}(\curl(\dev\grad\boldsymbol v)))\boldsymbol n_1 + \partial_{t}(\boldsymbol t^{\intercal}(\dev\grad\boldsymbol v)\boldsymbol t) \\
&= -\frac{1}{3}\boldsymbol n_1\cdot\curl(\boldsymbol n_{2}\div\boldsymbol v) + \partial_{tt}(\boldsymbol v\cdot\boldsymbol t)- \frac{1}{3}\partial_{t}(\div\boldsymbol v) \\
&= \frac{1}{3}(\boldsymbol n_1\times\boldsymbol n_2)\cdot\nabla(\div\boldsymbol v)  + \partial_{tt}(\boldsymbol v\cdot\boldsymbol t)- \frac{1}{3}\partial_{t}(\div\boldsymbol v)=\partial_{tt}(\boldsymbol v\cdot\boldsymbol t).
\end{align*}
Hence by~\eqref{eq:devgradtr1}-\eqref{eq:devgradtr2} it is easy to see that $\dev\grad\boldsymbol V_h\subset\bs\Sigma_h^{\mathbb T}$.
It holds from Lemma~\ref{lem:trace0} and the degrees of freedom~\eqref{Hsymcurlfem3ddof2}-\eqref{Hsymcurlfem3ddof3} that
\begin{equation}\label{eq:symcurlinsigmaS}
\sym\curl\bs\Sigma_h^{\mathbb T}\subset\bs\Sigma_h^{\mathbb S}.
\end{equation}
Thus we get from Lemma~\ref{lem:divdivonto} that~\eqref{eq:divdivcomplex3dfem} is a complex.

We then verify the exactness.

\medskip
\noindent {\em 1. $\boldsymbol V_h \cap\ker(\dev\grad)=\bs{RT}$.} By the exactness of the complex~\eqref{eq:divdivcomplex3d}, 
\[
\bs{RT}\subseteq\boldsymbol V_h \cap\ker(\dev\grad)\subseteq\boldsymbol H^1(\Omega;\mathbb R^3) \cap\ker(\dev\grad)=\bs{RT}.
\]

\medskip
\noindent {\em 2. $\bs\Sigma_h^{\mathbb T}\cap\ker(\sym\curl)=\dev\grad\boldsymbol V_h$, i.e. if $\sym\curl\boldsymbol \tau = 0$ and $\boldsymbol\tau\in\bs\Sigma_h^{\mathbb T}$, then there exists a $\boldsymbol v_h\in\boldsymbol V_h$, s.t. $\bs\tau =\dev\grad\boldsymbol v_h$}. 

Since $\sym\curl\boldsymbol \tau = 0$, by the divdiv complex~\eqref{eq:divdivcomplex3d}, there exists $\boldsymbol v\in\boldsymbol H^1(\Omega;\mathbb R^3)$ such that $\bs\tau=\dev\grad\boldsymbol v$.
Due to~\eqref{Hsymcurlfem3ddof4}, we obtain $(\boldsymbol v\cdot\boldsymbol n_i)|_e\in H^1(e)$ for each edge $e\in\mathcal E_h$ and $i=1,2$. Hence $\boldsymbol v(\delta)$ is well-defined for each vertex $\delta\in\mathcal V_h$. 
Take $\boldsymbol v_h\in\boldsymbol V_h$ satisfying
$
N(\boldsymbol v_h)=N(\boldsymbol v)
$, where $N$ goes through all the degrees of freedom~\eqref{Hermitfem3ddof1}-\eqref{Hermitfem3ddof4}.
Then it follows from the integration by parts that
$$
\dev\grad\boldsymbol v_h=\dev\grad\boldsymbol v=\bs\tau.
$$
This indicates $\bs\Sigma_h^{\mathbb T}\cap\ker(\sym\curl)\subseteq\dev\grad\boldsymbol V_h$.

\medskip
\noindent {\em 3. $\div\div\boldsymbol\Sigma_h^{\mathbb S}=\mathcal Q_h$.} This is Lemma~\ref{lem:divdivonto}.

\medskip
\noindent {\em 4. $\bs\Sigma_h^{\mathbb S}\cap\ker(\div\div)=\sym\curl\bs\Sigma_h^{\mathbb T}$.} 

We verify the identity by dimension count. By Lemma~\ref{lem:divdivonto},
\begin{align}
\dim(\bs\Sigma_h^{\mathbb S}\cap\ker(\div\div))&=\dim\bs\Sigma_h^{\mathbb S}-\dim\mathcal Q_h \notag\\
&= 6\#\mathcal V_h+ 3(\ell-1)\#\mathcal E_h + (\ell^2 -\ell+1)\#\mathcal F_h \notag \\
&\quad+\left(\frac{1}{6}(\ell-1)\ell(5\ell+17)-4\right)\#\mathcal T_h. \label{eq:20210112}
\end{align}
As a result of step 2,
\begin{align*}
\dim\sym\curl\bs\Sigma_h^{\mathbb T}&=\dim\bs\Sigma_h^{\mathbb T} -\dim\dev\grad\boldsymbol V_h=\dim\bs\Sigma_h^{\mathbb T} -\dim\boldsymbol V_h+4 \\
&=2\#\mathcal V_h+ (3\ell+1)\#\mathcal E_h + (\ell^2-\ell-3)\#\mathcal F_h \\
&\quad+\frac{1}{6}(\ell-1)\ell(5\ell+17)\#\mathcal T_h+4.
\end{align*}
Applying the Euler's formula $\#\mathcal V_h-\#\mathcal E_h+\#\mathcal F_h-\#\mathcal T_h=1$, we get from~\eqref{eq:20210112} that $\dim\sym\curl\bs\Sigma_h^{\mathbb T}=\dim(\bs\Sigma_h^{\mathbb S}\cap\ker(\div\div))$.
Then the result follows from~\eqref{eq:symcurlinsigmaS}. 

\medskip

Therefore the finite element divdiv complex~\eqref{eq:divdivcomplex3dfem} is exact.
\end{proof}


For the completeness, we present a two dimensional finite element divdiv complex but restricted to one element. A global version of~\eqref{eq:divdivcomplexPolyvar} as well as a commutative diagram involving quasi-interpolation operators from Sobolev spaces to finite element spaces can be found in~\cite{ChenHuang2020}.  

%
%
Let $\boldsymbol V_{\ell+1}(F):=\mathbb P_{\ell+1}(F;\mathbb R^2)$ with $\ell\geq2$ be the vectorial Hermite element~\cite{BrennerScott2008,Ciarlet1978}. 

\begin{lemma}
For any triangle $F$, the polynomial complex
\begin{equation}\label{eq:divdivcomplexPolyvar}
\boldsymbol  {RT}\xrightarrow{\subset} \boldsymbol V_{\ell+1}(F)\xrightarrow{\sym\curl_F} \boldsymbol \Sigma_{\ell,k}(F) \xrightarrow{\div_F{\div_F}} \mathbb P_{k-2}(F)\xrightarrow{}0
\end{equation}
is exact.
\end{lemma}

\section*{Acknowledgement.} The authors appreciate the anonymous reviewers for valuable suggestions and careful comments, which significantly improved the readability of an early version of the paper. 
The authors also want to thank Prof. Jun Hu, Dr. Yizhou Liang in Peking University, and Dr. Rui Ma in University of Duisburg-Essen for showing us the proof of the key Lemma~\ref{lem:symcurlkey} for constructing $H(\sym\curl)$-conforming element.

\bibliographystyle{abbrv}
\bibliography{./Hdivdiv3d}
\end{document}